\documentclass[a4paper,11pt]{article}
\usepackage[utf8]{inputenc}

\usepackage{amsmath,amssymb,amsthm}
\usepackage{graphicx,color}
\usepackage{url}
\usepackage{tikz-cd}
\graphicspath{{Fig/}}

\newtheorem{thm}{Theorem}[section]
\newtheorem{lem}[thm]{Lemma}
\newtheorem{cor}[thm]{Corollary}

\theoremstyle{definition}
\newtheorem{dfn}[thm]{Definition}
\newtheorem{ld}[thm]{Lemma-Definition}
\newtheorem{exl}[thm]{Example}

\theoremstyle{remark}
\newtheorem{rem}[thm]{Remark}

\def\R{\mathbb R}

\def\H{\mathbb H}
\def\dS{d\mathbb{S}}
\def\Sph{\mathbb S}

\def\dist{\operatorname{dist}}
\def\const{\mathrm{const}}

\def\Area{\operatorname{Area}}
\def\Id{\operatorname{Id}}
\def\arcosh{\operatorname{arcosh}}
\def\diag{\operatorname{diag}}
\def\PGL{\mathrm{PGL}}
\def\cP{\mathcal{P}}

\title{Spherical and hyperbolic conics}
\author{Ivan Izmestiev}

\begin{document}

\maketitle

\section{Introduction}
Most textbooks on classical geometry contain a chapter about conics. 
There are many well-known Euclidean, affine, and projective properties of 
conics. For broader modern presentations we can recommend the 
corresponding chapters in the textbook of Berger \cite{BerII} and two recent 
books \cite{AZ07,GSO16} devoted exclusively to conics. For a 
detailed survey of the results and history before the 20th century, 
see the encyclopedia articles \cite{Din03, DinFab}.

At the same time, one rarely speaks about non-Euclidean conics.
However, they should not be seen as something exotic. These 
are projective conics in the presence of a Cayley--Klein metric. In other 
words, 
a non-Euclidean conic is a pair of quadratic forms $(\Omega, S)$ on $\R^3$, 
where $\Omega$ (the absolute) is non-degenerate. Geometrically, a 
spherical conic is the intersection of the sphere with a quadratic cone; what 
are the metric properties of this curve with respect to the intrinsic metric of 
the sphere? In the Beltrami--Cayley--Klein model of the hyperbolic plane, a 
conic 
is the intersection of an affine conic with the disk standing for the plane.

Non-Euclidean conics share many properties with 
their Euclidean relatives. For example, the set of points on the sphere with a 
constant sum of distances from two given points is a spherical ellipse. The same 
is true in the hyperbolic plane. A reader interested in the bifocal properties 
of other hyperbolic conics can take a look at Theorem \ref{thm:SumConstHyp} or 
Figure \ref{fig:HypOptic}.

In the non-Euclidean geometry we have the polarity with respect to the 
absolute.
% In the spherical geometry, the polar of a point is the great 
% circle at distance $\frac{\pi}2$. In the Beltrami-Cayley-Klein model of 
% hyperbolic geometry, the polar of a point inside the disk is a line disjoint 
% from the disk (a de Sitter line), and the polar of a point outside of the 
disk 
% is a hyperbolic line.
The absolute polarity exchanges distances and angles, so that for every metric
theorem there is a dual one. For 
example, the theorem about constant sum of distances from two fixed points 
becomes a theorem about the envelope of lines cutting triangles of constant 
area 
from a given angle. See Figures \ref{fig:ConstSum} and \ref{fig:HypAreaHyp}.

The first systematic study of spherical conics was undertaken by Chasles. 
In the first part of \cite{ChGr} he proves dozens of metric properties of 
quadratic cones by elegant synthetic arguments, in the second 
part he interprets them as statements about spherical conics. 
The work of Chasles was translated from French to English and supplemented by 
Graves, see also \cite{Graves40}.
A coordinate approach to spherical conics is used in \cite{SP77} and 
\cite{GSO16}. Among the works on hyperbolic conics the article of Story 
\cite{Story83} can be singled out. It uses with much success pencils of conics 
and the relation between $\Omega(v,w)$ and $\dist(v,w)$. Story's article also 
contains canonical forms of equations of hyperbolic conics, see also 
\cite{Fladt1,Fladt2}.

This article is to a great extent based on the works of Chasles and Story. We 
were unable to reproduce all of their results; this could have doubled the 
length of the article. An interested reader is referred to the originals.

\section{Quadratic forms, conics, and pencils}
\subsection{The dual conic}
\label{sec:DualConic}
Let $S$ be a symmetric bilinear form on a real $3$-dimensional vector space 
$V$. The \emph{isotropic cone} of the corresponding quadratic form is the set
\[
I_S = \{v \in V \mid S(v,v)=0\}.
\]
The image of $I_S$ in the projective plane $P(V)$ is called a projective 
conic.
% It is often useful to take the complexification $I_S^{\C} \subset 
% V^{\C}$. For example, the quadratic form $x^2 + y^2$ defines an empty real 
% conic, whose complexification is a pair of complex conjugate lines $x+iy = 0$ 
% and $x-iy=0$.

The form $S$ defines a linear homomorphism
\[
H_S \colon V \to V^*, \quad \langle H_S(v), w \rangle = S(v,w).
\]
The symmetry of $S$ translates as the self-adjointness of $H_S$ with respect to 
the canonical pairing $\langle \cdot, \cdot \rangle$ between $V$ and $V^*$:
\[
\langle H_S(v), w \rangle = \langle v, H_S(w) \rangle.
\]
There are one-to-one correspondences between quadratic forms, symmetric 
bilinear forms, self-adjoint homomorphisms $V \to V^*$, and (complexified) 
isotropic cones. A \emph{conic} is an equivalence class of any of the above 
objects 
under scaling. By abuse of notation we will often denote the conic by the same 
symbol as the symmetric bilinear form; do not forget that scaling the form does 
not change the conic.

% Together with conics in $P(V)$, it is convenient to consider the conics in the 
% projectivization $P(V^*)$ of the dual vector space.
The elements of $P(V^*)$ 
can be interpreted as the lines in $P(V)$. Therefore a 
conic in $P(V)$ is called a \emph{point conic}, and a conic in $P(V^*)$ is 
called a \emph{line conic}.

With every non-degenerate point conic one can associate a line conic.
This can be done in various ways that turn out to be equivalent:
\begin{itemize}
\item
Invert the homomorphism $H_S$.
\item
Take the image $H_S(I_S)$.
\item
Take the set of tangents to the conic.
\end{itemize}

\begin{ld}
\label{ld:DualConic}
The three constructions listed above result in the same line conic, called 
the \emph{dual conic} of $S$ and denoted by $S^*$.
\end{ld}
\begin{proof}
The first construction yields the cone
\[
\{f \in V^* \mid \langle H_S^{-1}(f), f \rangle = 0\},
\]
and the second
\[
\{H_S(v) \mid v \in V,\, \langle 
v, H_S(v) \rangle = 0\},
\]
which is clearly the same. Also, for every $v \in I_S$ we have
\[
\ker H_S(v) = \{w \in v \mid S(v,w) = 0\},
\]
which is the plane tangent to $I_S$ along the line spanned by $v$.
\end{proof}

Thus for the dual conic $S^*$ we have
\[
H_{S^*} = H_S^{-1}, \quad I_{S^*} = H_S(I_S).
\]

% 
% There is a one-to-one correspondence between non-degenerate point conics 
% and non-degenerate line conics. Namely, for a quadratic form $Q$ on $V$ 
% denote by the same symbol the corresponding symmetric bilinear form and the 
% corresponding linear homomorphism
% \begin{equation}
% \label{eqn:FormHom}
% Q \colon V \to V^*, \quad Q(v)(w) = Q(v,w) = Q(w)(v).
% \end{equation}
% For non-degenerate quadratic forms this homomorphism is one-to-one, hence has 
% an inverse. The symmetry of $Q$ implies the symmetry of $Q^{-1}$, so that the 
% latter defines a quadratic form on $V^*$.
% 
% \begin{dfn}
% The quadratic form that corresponds to the inverse of the homomorphism 
% \eqref{eqn:FormHom} is called \emph{dual} to $Q$. The corresponding conic in 
% $P(V^*)$ is called the \emph{dual conic} of the conic in $P(V)$ defined by $Q$.
% \end{dfn}
% 
% Points of $P(V^*)$ correspond to lines in $P(V)$.
% Geometrically, the line conic $Q^{-1}$ consists of the tangents to the 
% point conic $Q$.

A degenerate point conic is a pair of lines or a double line, and 
a degenerate line conic is a pair of points or a double point (which means that 
the conic consists of all lines through these points, taken twice if the points 
coincide). The duality between non-degenerate point conics and line 
conics does not extend to degenerate ones. The best one can do is to extend the 
relation $H_S \circ H_{S^*} = \Id$ projectively, see \cite[\S 4.D]{Sam}.

\subsection{Euclidean vs non-Euclidean geometry}
\label{sec:EucNonEuc}
Let $\Omega$ be a non-degenerate conic in $P(V)$, hereafter called the 
\emph{absolute}. The projective transformations from $\PGL(V)$ that map 
$\Omega$ 
to itself preserve the \emph{Cayley--Klein distance} between the points, 
defined as half the logarithm of their cross-ratio with the collinear points on 
the absolute. They also preserve the \emph{Cayley--Klein angles} between the 
lines, defined through their cross-ratio with the concurrent tangents 
to $\Omega$ (which can be interpreted as the cross-ratio of the corresponding 
points in $P(V^*)$ with points on the line conic dual to $\Omega$).

For a sign-definite form $\Omega$ we obtain the elliptic geometry of $P(V)$; 
for an indefinite form we obtain the hyperbolic (or hyperbolic-de Sitter) 
geometry.

The Euclidean geometry is the geometry of a degenerate line conic that is 
positive definite and has rank two. Such a conic consists of lines 
that pass through one of two complex conjugate points. The line through 
these points is real and is fixed by the transformations that map the conic to 
itself. This designates it as the line at infinity. A degenerate line 
conic allows to define the Cayley--Klein angles, but not the Cayley--Klein 
distances. The corresponding transformation group consists of all 
similarity transformations.

\subsection{The absolute polarity}
A non-degenerate absolute $\Omega$ corresponds to a self-adjoint isomorphism 
$H_\Omega \colon V \to V^*$. Projectivization 
leads to a well-defined isomorphism between 
projective planes
\[
\label{eqn:AbsPol}
[H_\Omega] \colon P(V) \to P(V^*),
\]
which, together with its inverse, is called the \emph{absolute polarity}. The 
image of a point $p \in P(V)$ is called 
the \emph{(absolute) polar} of $p$ and denoted by $p^\circ$. The preimage of a 
line $\ell \in P(V^*)$ is called the \emph{(absolute) pole} of $\ell$ and 
denoted by $\ell^\circ$.

Equivalently, the absolute polar of a linear subspace of $V$ is its orthogonal 
complement with respect to $\Omega$. Here we identify the two-dimensional 
subspaces of $P(V)$ (that is, the lines in $P(V)$) with the points of 
$P(V^*)$).

The points $[v], [w] \in P(V)$ such that $\Omega(v,w) = 0$ are called 
\emph{conjugate} with respect to $\Omega$. In other words, the polar of a point 
consists of all points conjugate with it.

Two lines are called conjugate if (viewed as elements of $P(V^*)$) they are 
conjugate with respect to the dual quadratic form $\Omega^{-1}$. Conjugate 
lines contain the poles of each other.

The absolute polarity does not change when $\Omega$ is scaled, thus it is 
completely defined by the conic $\Omega$. The polar of a point lying on the 
conic is the tangent at that point. The fact that the absolute polarity 
preserves the incidences between the points and the lines leads to geometric 
constructions of poles and polars shown on Figure \ref{fig:PolesPolars}.

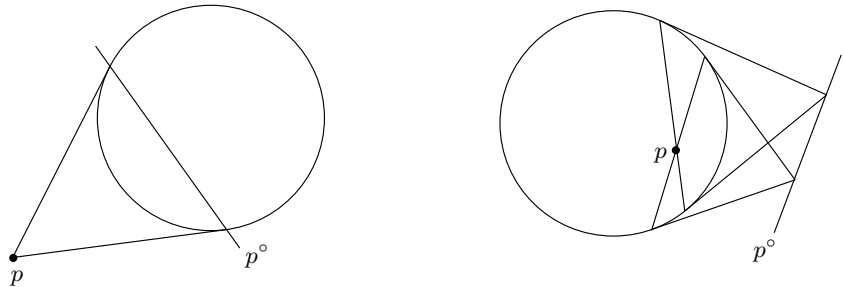
\begin{figure}[htb]
\begin{center}
\begin{picture}(0,0)%
\includegraphics{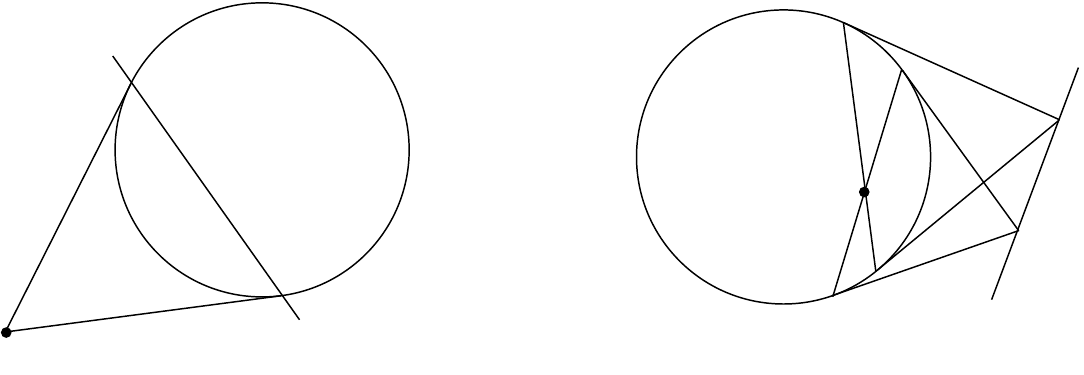}%
\end{picture}%
\setlength{\unitlength}{3729sp}%
\begingroup\makeatletter\ifx\SetFigFont\undefined%
\gdef\SetFigFont#1#2#3#4#5{%
  \reset@font\fontsize{#1}{#2pt}%
  \fontfamily{#3}\fontseries{#4}\fontshape{#5}%
  \selectfont}%
\fi\endgroup%
\begin{picture}(5490,1899)(-367,-1204)
\put(-352,-1140){\makebox(0,0)[lb]{\smash{{\SetFigFont{9}{10.8}{\rmdefault}{
\mddefault}{\updefault}{\color[rgb]{0,0,0}$p$}%
}}}}
\put(1193,-1019){\makebox(0,0)[lb]{\smash{{\SetFigFont{9}{10.8}{\rmdefault}{
\mddefault}{\updefault}{\color[rgb]{0,0,0}$p^\circ$}%
}}}}
\put(3883,-329){\makebox(0,0)[lb]{\smash{{\SetFigFont{9}{10.8}{\rmdefault}{
\mddefault}{\updefault}{\color[rgb]{0,0,0}$p$}%
}}}}
\put(4534,-982){\makebox(0,0)[lb]{\smash{{\SetFigFont{9}{10.8}{\rmdefault}{
\mddefault}{\updefault}{\color[rgb]{0,0,0}$p^\circ$}%
}}}}
\end{picture}%
\end{center}
\caption{Constructing poles and polars.}
\label{fig:PolesPolars}
\end{figure}

% 
% \subsection{Conjugacy and polarity}
% \textbf{Conjugacy wrt conic, orthogonality wrt quadratic form}
% 
% Let $Q$ be a symmetric bilinear form on a vector space $V$.
% Two vectors $v, w \in V$ are called \emph{orthogonal with respect to $Q$} if 
% $Q(v,w) = 0$. The quadric $C$ defined by $Q$ is the set of vectors conjugate 
% to itself:
% \[
% C = \{v \in V \mid Q(v,v) = 0\}.
% \]
% The set of all vectors conjugate to $v$ with respect to $Q$ is 
% called the \emph{polar of $v$ with respect to $Q$} and denoted by $v^\perp_Q$. 
% Similarly, the polar of a subspace $W \subset V$ consists of all vectors that 
% are conjugate to all vectors from $W$. If $Q$ is non-degenerate (which we will 
% assume to be the case), then the polar of every point is a hyperplane and $\dim 
% (W^\perp_Q) + \dim W = \dim V$ for every subspace $W$. Two subspaces $U$ and 
% $W$ are called \emph{conjugate with respect to $Q$}, if either $U^\perp_Q 
% \subset W$ or $W ^\perp_Q \subset U$.
% 
% For every point $v \in C$ the hyperplane $v^\perp_Q$ is 
% called the \emph{tangent to $C$ at $v$}.
% 
% Polarity with respect to the absolute quadratic form $\Omega$ will often be 
% denoted as
% \[
% v^\circ = v^\perp_\Omega, \quad W^\circ = W^\perp_\Omega.
% \]
% The quadric $C^\circ$ polar to $C$ with respect to $\Omega$ is the envelope of 
% the hyperplanes $v^\circ$ with $v \in C$. A more formal definition is 
% given in the next section.

\subsection{The polar conic}
Let, in addition to an absolute conic $\Omega$, another non-degenerate conic $S$ 
be given. Its dual $S^*$, defined in Section \ref{sec:DualConic}, is a line 
conic. We 
can use the absolute polarity to transform it to a point conic. Again, there 
are three equivalent descriptions of this construction.

\begin{itemize}
\item
The homomorphism $H_\Omega \circ H_{S^*} \circ H_\Omega \colon V \to 
V^*$.
\item
The cone $H_\Omega^{-1}(I_{S^*}) \subset V$.
\item
The poles of the tangents to the conic $S$.
\end{itemize}

\begin{ld}
\label{ld:PolCon}
The three constructions listed above result in the same conic, called the 
\emph{polar conic} of $S$ and denoted by $S^\circ$.
\end{ld}
\begin{proof}
The first construction yields the set of all $v \in V$ that satisfy
\[
\langle H_\Omega(H_{S^*}(H_\Omega(v))), v \rangle = 0.
\]
Since we have
\[
\langle H_\Omega(H_{S^*}(H_\Omega(v))), v \rangle = \langle 
H_{S^*}(H_\Omega(v)), H_\Omega(v) \rangle,
\]
this set is the image under $H_\Omega^{-1}$ of the set
\[
\{ f \in V^* \mid \langle H_{S^*}(f), f \rangle = 0\} = I_{S^*}.
\]
Hence the second construction is equivalent to the first.

The third construction is equivalent to the second, because the set
$H_S(I_S)$ consists of the tangents to the conic $S$, and the map 
$H_\Omega^{-1}$ represents the absolute polarity.
\end{proof}

As a consequence, the absolute polarity transforms conjugacy with respect to 
$S$ to conjugacy with respect to $S^\circ$.
\begin{cor}
\label{cor:ConjPolar}
A line $\ell$ is the polar of a point $p$ with respect to a conic $S$ if and 
only if the point $\ell^\circ$ is the pole of the line $p^\circ$ with respect to 
the conic $S^\circ$.
\end{cor}
\begin{proof}
The diagram
\[
\begin{tikzcd}
P(V^*) \arrow{r}{[H_{S^*}]} \arrow{d}{[H_\Omega]} & P(V) 
\arrow{d}{[H_\Omega^{-1}]}\\
P(V) \arrow{r}{[H_{S^\circ}]} & P(V^*)
\end{tikzcd}
\]
is commutative, since $H_{S^\circ} = H_\Omega \circ H_{S^*} \circ H_\Omega$ 
by 
Lemma-Definition \ref{ld:PolCon}.
\end{proof}

\subsection{Pencils of conics}
Since quadratic forms on a three-dimensional vector space form a vector space 
of dimension six, the projective conics form a 
five-dimensional projective space.

\begin{dfn}
A \emph{pencil of conics} is a line in the projective space of conics.
\end{dfn}

A pencil spanned by the conics $S$ and $T$ consists of conics of the form 
$\lambda P + \mu Q$. The degeneracy condition
\[
\det(\lambda P + \mu Q) = 0
\]
is a homogeneous polynomial of degree $3$ in $\lambda$ and $\mu$. 
Therefore either all or at most three conics of a pencil are degenerate.

\begin{exl}[Conics through four points]
\label{exl:4PtsPencil}
Consider a pencil that contains two pairs of lines. If a line of one pair 
coincides with a line of the other pair, then this line is contained in all 
conics of the pencil, hence all conics are degenerate. If all four lines are 
distinct, then the first pair intersects the second pair in four 
points, and the pencil consists of the conics through these four points, see 
Figure \ref{fig:RealPencils}, left. In 
particular, it contains a third degenerate conic made of a third pair of lines.
\end{exl}

\begin{exl}[Double contact pencil]
\label{exl:DoubleContact}
Consider a pencil that contains a pair of lines $\ell_1, \ell_2$ and a double 
line through the points $p_1 \in \ell_1$ and $p_2 \in \ell_2$ (assuming that 
neither $p_1$ nor $p_2$ is the intersection point of $\ell_1$ and $\ell_2$). 
Then every conic of the pencil goes through the points $p_1$ and $p_2$ and is 
tangent to the lines $\ell_1$ and $\ell_2$ at these points (the latter can be 
seen by viewing the double line as a limit of pairs of lines). See Figure 
\ref{fig:RealPencils}, right.
\end{exl}

\begin{figure}[htb]
\begin{center}
\includegraphics[width=.4\textwidth]{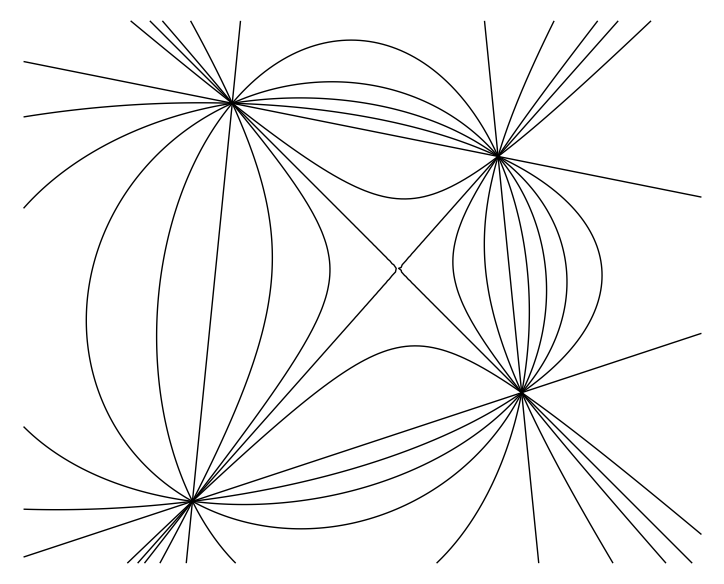} \hspace{.5cm} 
\includegraphics[width=.4\textwidth]{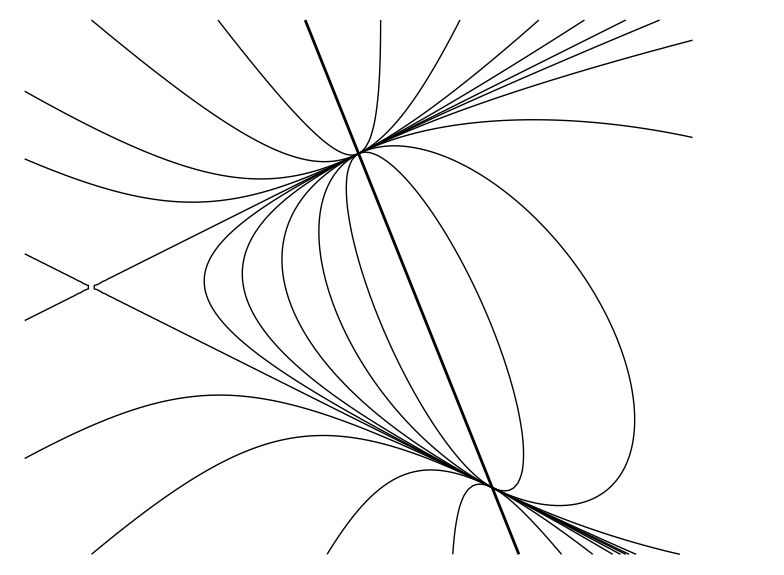}
\end{center}
\caption{Pencil of conics through four real points and a real double contact 
pencil.}
\label{fig:RealPencils}
\end{figure}

There are pencils that involve imaginary elements, for example the pencil of 
conics through two pairs of complex conjugate points.

For more details, including other types of pencils with illustrations, see 
\cite[3.3]{AZ07}, \cite[16.4]{BerII}, and \cite[7.3]{GSO16}

\subsection{Dual pencils and confocal conics}
\label{sec:DualPencils}
Take a pencil $\cP$ of line conics, that is a line in the space of conics in 
$P(V^*)$. The non-degenerate conics of $\cP$ can be dualized (see Section 
\ref{sec:DualConic}). This gives a family of point conics, called a
\emph{dual pencil} or a tangential pencil.

\begin{exl}
Let $\cP$ be a pencil of line conics through four lines in general position. 
The corresponding dual pencil consists of non-degenerate point conics tangent 
to those lines.
The degenerate members of $\cP$ are three pairs of interection points of the 
four lines. Thus, in a sense, the dual pencil is spanned by the three diagonals 
of a complete quadrilateral. Intuitively, these diagonals are degenerate 
ellipses tangent to the four lines.
\end{exl}

\begin{exl}
Let $\cP$ be a double contact pencil. The dual pencil is also (the 
non-degenerate part of) a double contact pencil: since the dual conic is 
made of tangents, the dual of a conic tangent to $\ell$ at $p$ is a line conic 
``tangent to'' $p$ at $\ell$. The degenerate members of this dual pencil are a 
pair of points (the points of contact) and a double point (the intersection 
point of the lines of contact).
\end{exl}

% 
% \begin{dfn}
% Conics of a pencil that contains the absolute conic are called \emph{concyclic}.
% \end{dfn}
% 
% \begin{lem}
% Concyclic conics share focal lines.
% \end{lem}
% \begin{proof}
% Indeed, a focal line of a conic $Q$ is a line on which the restrictions of the 
% quadratic form $Q$ is proportional to the absolute form $\Omega$. Hence the 
% restriction of a conic $\lambda\Omega + \mu Q$ to the same line will also be 
% proportional to $\Omega$.
% \end{proof}
% 
% 
% In the hyperbolic case, concyclic pencils can have different forms. Here is a 
% couple of examples.
% \begin{itemize}
% \item
% If a conic has four real absolute points, then its concyclic 
% pencil consists of convex and concave hyperbolas, as well as of three pairs of 
% lines through these points.
% \item
% The concyclic pencil of a hypercycle consists of 
% all hypercycles around the same line.
% \item
% The concyclic pencil of a point consists of circles centered at this point.
% \end{itemize}

\begin{dfn}
If a pencil of line conics contains the dual absolute, then the corresponding 
dual pencil is called a \emph{confocal family of conics}.
\end{dfn}
The definition also makes sense in the Euclidean geometry, where the dual 
absolute conic is degenerate, see Section \ref{sec:EucNonEuc}.

We define the foci of spherical and hyperbolic conics in Sections 
\ref{sec:FocSph} and \ref{sec:FocHyp} and discuss confocal families in Sections 
\ref{sec:FamSph} and \ref{sec:FamHyp} in more detail.

\begin{lem}
Confocal conics intersect orthogonally.
\end{lem}
\begin{proof}
This statement means that the tangents of two confocal conics at their 
intersection point are conjugate with respect to the absolute. Equivalently 
(in terms of line conics), the points of tangency of a common tangent to two 
conics of a pencil are conjugate with respect to any conic from this pencil, 
see Figure \ref{fig:ConfOrth}. But since these points are conjugate with 
respect to the conics to which they belong, they are so with respect to any 
linear combination.
\end{proof}

\begin{figure}[htb]
\begin{center}
\includegraphics{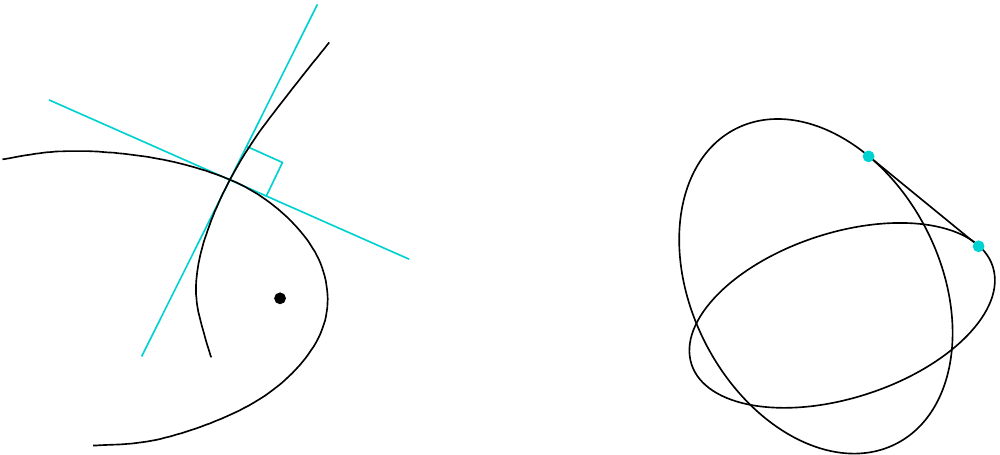}
\end{center}
\caption{Confocal conics intersect orthogonally, and the dual statement.}
\label{fig:ConfOrth}
\end{figure}

\begin{exl}
A special case of confocal conics is the dual of the pencil 
spanned by the dual absolute and a double point $2p$. If $p$ does not lie on 
the absolute, then this is a double contact pencil spanned by the tangents 
from $p$ to $\Omega$ and the double polar of $p$.
This pencil is made by the circles centered at $p$, which can be shown by a 
simple computation. (In the hyperbolic case, $p$ can be hyperbolic, de Sitter, 
or ideal point, see Lemma \ref{lem:Cycles}.)
\end{exl}

\subsection{Chasles' theorems}
Chasles' article \cite{Chasles60} contains a variety of theorems, whose proofs 
follow all the same principle. A pencil of conics is a line in the 
projective space of all conics. If in a collection of lines there are many 
concurrent pairs, then all lines are coplanar, and hence any two of them are 
concurrent. Translated to pencils of conics, this means that if in a collection 
of pencils many pairs share a conic, then any two of these pencils share a 
conic. The argument also works for dual pencils, since they correspond to 
lines in the space of line conics, and in particular for confocal families of 
conics.

The following theorem is one of those contained in \cite{Chasles60}.

\begin{thm}[Chasles]
\label{thm:Chasles}
Let $\{\ell_i\}_{i=1}^4$ be four tangents to a conic $A$. Denote by $p_{ij}$ 
the intersection point of $\ell_i$ and $\ell_j$. Assume that the points 
$p_{12}$ and $p_{34}$ lie on a conic $B$ confocal to $A$.
Then the following holds.
\begin{enumerate}
\item
The pairs of points $p_{13}, p_{24}$ and $p_{14}, p_{23}$ also lie on conics 
confocal to~$A$.
\item The tangents at the points $p_{ij}$ to the three conics that contain 
pairs of points meet at the same point~$q$.
\item There is a circle tangent to the lines $\ell_i$, and this circle is 
centered at~$q$.
\end{enumerate}
If, in the hyperbolic case, the point $q$ is ideal or de Sitter, then the 
role of a circle centered at $q$ is played by a horocycle or a hypercycle.
\end{thm}

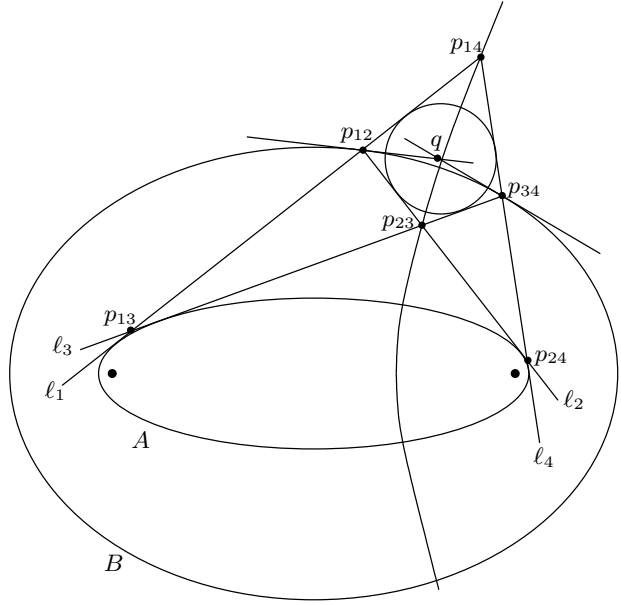
\begin{figure}[htb]
\begin{center}
\begin{picture}(0,0)%
\includegraphics{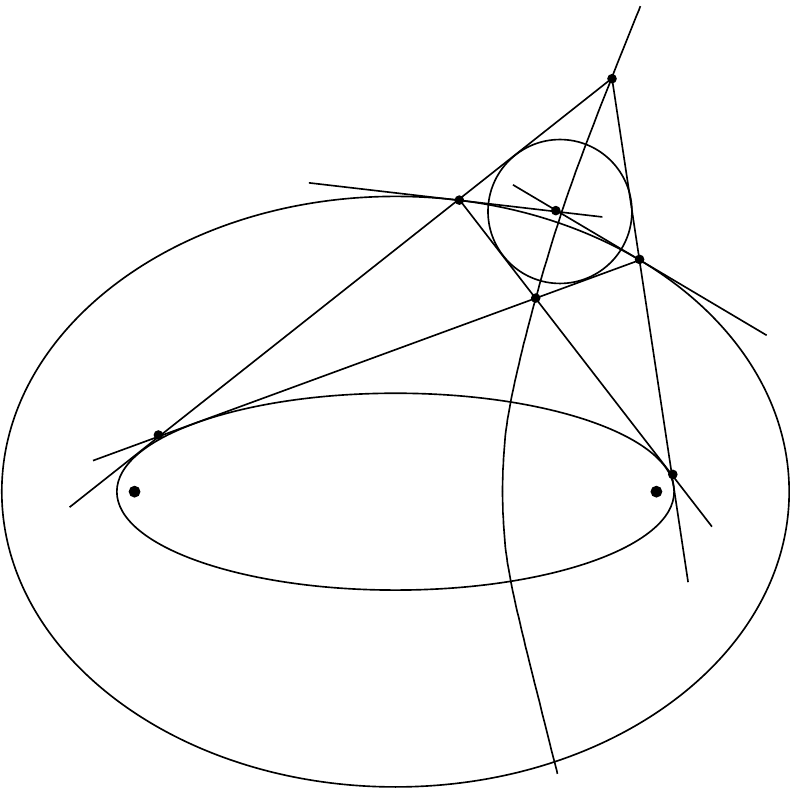}%
\end{picture}%
\setlength{\unitlength}{4144sp}%
\begingroup\makeatletter\ifx\SetFigFont\undefined%
\gdef\SetFigFont#1#2#3#4#5{%
  \reset@font\fontsize{#1}{#2pt}%
  \fontfamily{#3}\fontseries{#4}\fontshape{#5}%
  \selectfont}%
\fi\endgroup%
\begin{picture}(3616,3588)(-1807,-517)
\put(-1241,-342){\makebox(0,0)[lb]{\smash{{\SetFigFont{9}{10.8}{\rmdefault}{
\mddefault}{\updefault}{\color[rgb]{0,0,0}$B$}%
}}}}
\put(1151,1916){\makebox(0,0)[lb]{\smash{{\SetFigFont{9}{10.8}{\rmdefault}{
\mddefault}{\updefault}{\color[rgb]{0,0,0}$p_{34}$}%
}}}}
\put(161,2245){\makebox(0,0)[lb]{\smash{{\SetFigFont{9}{10.8}{\rmdefault}{
\mddefault}{\updefault}{\color[rgb]{0,0,0}$p_{12}$}%
}}}}
\put(693,2196){\makebox(0,0)[lb]{\smash{{\SetFigFont{9}{10.8}{\rmdefault}{
\mddefault}{\updefault}{\color[rgb]{0,0,0}$q$}%
}}}}
\put(-1077,395){\makebox(0,0)[lb]{\smash{{\SetFigFont{9}{10.8}{\rmdefault}{
\mddefault}{\updefault}{\color[rgb]{0,0,0}$A$}%
}}}}
\put(-1590,694){\makebox(0,0)[lb]{\smash{{\SetFigFont{9}{10.8}{\rmdefault}{
\mddefault}{\updefault}{\color[rgb]{0,0,0}$\ell_1$}%
}}}}
\put(-1546,963){\makebox(0,0)[lb]{\smash{{\SetFigFont{9}{10.8}{\rmdefault}{
\mddefault}{\updefault}{\color[rgb]{0,0,0}$\ell_3$}%
}}}}
\put(1480,641){\makebox(0,0)[lb]{\smash{{\SetFigFont{9}{10.8}{\rmdefault}{
\mddefault}{\updefault}{\color[rgb]{0,0,0}$\ell_2$}%
}}}}
\put(1303,297){\makebox(0,0)[lb]{\smash{{\SetFigFont{9}{10.8}{\rmdefault}{
\mddefault}{\updefault}{\color[rgb]{0,0,0}$\ell_4$}%
}}}}
\put(1311,916){\makebox(0,0)[lb]{\smash{{\SetFigFont{9}{10.8}{\rmdefault}{
\mddefault}{\updefault}{\color[rgb]{0,0,0}$p_{24}$}%
}}}}
\put(407,1720){\makebox(0,0)[lb]{\smash{{\SetFigFont{9}{10.8}{\rmdefault}{
\mddefault}{\updefault}{\color[rgb]{0,0,0}$p_{23}$}%
}}}}
\put(811,2775){\makebox(0,0)[lb]{\smash{{\SetFigFont{9}{10.8}{\rmdefault}{
\mddefault}{\updefault}{\color[rgb]{0,0,0}$p_{14}$}%
}}}}
\put(-1240,1152){\makebox(0,0)[lb]{\smash{{\SetFigFont{9}{10.8}{\rmdefault}{
\mddefault}{\updefault}{\color[rgb]{0,0,0}$p_{13}$}%
}}}}
\end{picture}%

\end{center}
\caption{Chasles' theorem.}
\end{figure}

\begin{proof}
The idea is to dualize the picture, so that confocal conics become conics 
collinear with the dual absolute, see Section \ref{sec:DualPencils}.

Consider two pencils of line conics: one spanned by $A^*$ and the dual absolute, 
the other made by conics through the lines $\ell_i$ (duals of conics tangent to 
these lines). The second pencil contains the 
conics $A^*$, $p_{12} + p_{34}$, $p_{13} + p_{24}$, and $p_{14} + p_{23}$
(the latter three line conics are degenerate, see the last paragraph of 
Section \ref{sec:DualConic}). Since the two pencils share the conic $A^*$, they 
lie in a plane in the projective space of line conics.

Consider the pencil of line conics spanned by $B^*$ and $p_{12} + p_{34}$. This 
is a double contact pencil; it contains the 
double point $2q$, where $q$ is the intersection point of the tangents to $B$ 
at $p_{12}$ and $p_{34}$. See Figure \ref{fig:ChaslesProof}, where the line 
conics are depicted as points.

\begin{figure}[htb]
\begin{center}
\begin{picture}(0,0)%
\includegraphics{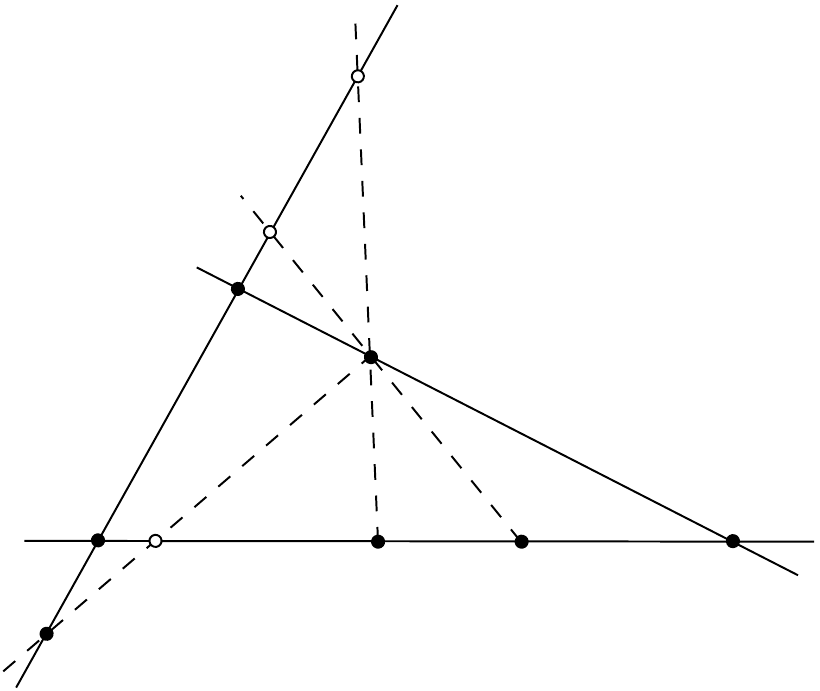}%
\end{picture}%
\setlength{\unitlength}{4972sp}%
\begingroup\makeatletter\ifx\SetFigFont\undefined%
\gdef\SetFigFont#1#2#3#4#5{%
  \reset@font\fontsize{#1}{#2pt}%
  \fontfamily{#3}\fontseries{#4}\fontshape{#5}%
  \selectfont}%
\fi\endgroup%
\begin{picture}(3113,2625)(-910,-1299)
\put(-697,-1235){\makebox(0,0)[lb]{\smash{{\SetFigFont{10}{12.0}{\rmdefault}{
\mddefault}{\updefault}{\color[rgb]{0,0,0}$\Omega^*$}%
}}}}
\put(-659,-673){\makebox(0,0)[lb]{\smash{{\SetFigFont{10}{12.0}{\rmdefault}{
\mddefault}{\updefault}{\color[rgb]{0,0,0}$A^*$}%
}}}}
\put(863,-890){\makebox(0,0)[lb]{\smash{{\SetFigFont{10}{12.0}{\rmdefault}{
\mddefault}{\updefault}{\color[rgb]{0,0,0}$p_{14}+p_{23}$}%
}}}}
\put(172,-898){\makebox(0,0)[lb]{\smash{{\SetFigFont{10}{12.0}{\rmdefault}{
\mddefault}{\updefault}{\color[rgb]{0,0,0}$p_{13}+p_{24}$}%
}}}}
\put(-235,119){\makebox(0,0)[lb]{\smash{{\SetFigFont{10}{12.0}{\rmdefault}{
\mddefault}{\updefault}{\color[rgb]{0,0,0}$B^*$}%
}}}}
\put(1905,-667){\makebox(0,0)[lb]{\smash{{\SetFigFont{10}{12.0}{\rmdefault}{
\mddefault}{\updefault}{\color[rgb]{0,0,0}$p_{12}+p_{34}$}%
}}}}
\put(593,  
4){\makebox(0,0)[lb]{\smash{{\SetFigFont{10}{12.0}{\rmdefault}{\mddefault}{
\updefault}{\color[rgb]{0,0,0}$2q$}%
}}}}
\end{picture}%
\end{center}
\caption{Proof of the Chasles theorem.}
\label{fig:ChaslesProof}
\end{figure}
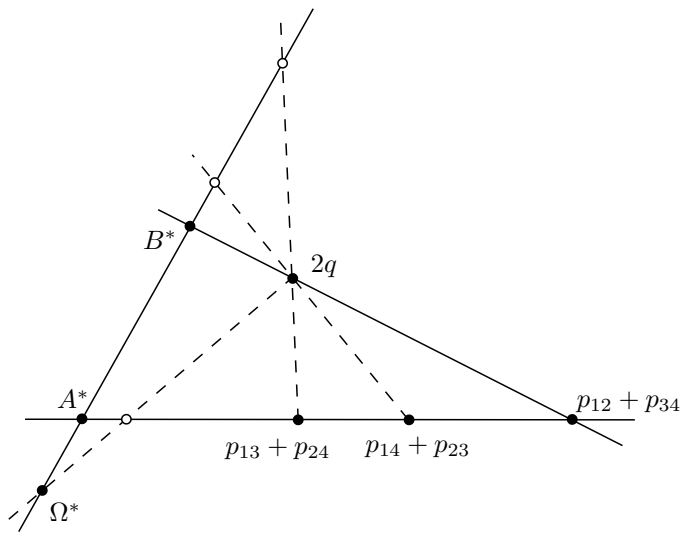

The pencil of line conics spanned by $p_{14}+p_{23}$ and $2q$
and the pencil spanned by $A^*$ and $\Omega^*$ have a conic in common. Its dual 
is a point conic confocal with $A$ and tangent at the points $p_{14}$ and 
$p_{23}$ to the lines $p_{14}q$ and $p_{23}q$. The same is true with 
$p_{13}+p_{24}$ in place of $p_{14}+p_{23}$, which proves the first two parts 
of the theorem.

% is a double 
% contact pencil; its dual is a double contact pencil of point conics tangent 
% at 
% the points $p_{14}$ and $p_{23}$ to the lines $p_{14}q$ and $p_{23}q$.

% It follows that the line conics $\Omega$, $p_3 + p_4$, and $2q$ belong to the 
% plane spanned by $A$, $A'$, and $p_1 + p_2$ in the projective space of 
% conics. 
% In particular, the pencil spanned by $A$ and $A'$ and the pencil spanned by 
% $p_3 + p_4$ and $2q$ have a conic in common. The first pencil consists of 
% conics confocal to $A$ and $A'$. The second pencil is degenerate and consists 
% of conics through $p_3$ and $p_4$ whose tangents at $p_3$ and $p_4$ pass 
% through $q$. This proves the first two parts of the theorem.

For the third part, consider the pencil spanned by $\Omega^*$ and $2q$. It 
consists of circles centered at $q$ (for the hyperbolic case, see Lemma 
\ref{lem:Cycles}). This pencil intersects the pencil of conics through the 
lines $\ell_i$, hence there is a circle centered at $q$ tangent to those lines.
\end{proof}

The third part of the theorem is related to the following property of confocal 
conics: a billiard trajectory inside a conic is tangent to a confocal conic.

\subsection{Projective properties}
Theorems of Pascal, Brianchon, and Poncelet deal with projective properties of 
conics, therefore they hold for non-Euclidean conics as well as for Euclidean 
ones. However, in the non-Euclidean case one can obtain new theorems from known 
ones by modifying them as follows.
\begin{itemize}
\item
Apply the absolute polarity to some of the elements.
\item
If there is one or several conics in the premises of the theorem, assume one of 
them to be the absolute conic.
\item
If there are several conics, assume two of them to be polar to each other.
\end{itemize}

Theorems below illustrate this.

\begin{thm}
The common perpendiculars to the pairs of opposite sides of a 
spherical or hyperbolic hexagon intersect in a point if and only if the hexagon 
is inscribed in a conic.
\end{thm}
\begin{proof}
The common perpendicular to two lines $\ell_1$ and $\ell_2$ is the line through 
the poles $\ell_1^\circ$ and $\ell_2^\circ$. A hexagon is inscribed in a conic 
if and only if its polar dual is circumscribed about a conic. Hence this 
theorem follows from the Brianchon theorem by applying the absolute polarity.
\end{proof}

\begin{thm}
The diagonals and the common perpendiculars to the opposite pairs of sides in 
an ideal hyperbolic quadrilateral meet at a point.
\end{thm}
\begin{proof}
The common perpendiculars are the diagonals of the polar quadrilateral, which 
is circumscribed about the absolute. The diagonals and the lines joining the 
opposite points of tangency meet at a point; this is a limit case of the 
Brianchon theorem.
\end{proof}

\begin{thm}
The set of points from which a (Euclidean or non-Euclidean) conic is seen under 
the right angle is also a conic.
\end{thm}
\begin{proof}
The tangents drawn from a point to the conic are orthogonal if and only if 
they harmonically separate the tangents to the absolute conic. This set of 
points is called the \emph{harmonic locus} of the conic and the absolute. 
The harmonic locus of any two conics is a conic; this can be proved with the 
help of Chasles' theory of $(2\textendash 2)$ correspondences, see \cite[\S 
50]{Fau}.
\end{proof}

The set of points from which a curve is seen under the right angle is called 
the \emph{orthoptic curve}. In the Euclidean case, the orthoptic curve of a 
conic is a circle.

\begin{thm}
Choose any tangent $\ell_1$ to a non-Euclidean conic $S$. Let $\ell_2$ be a 
tangent to $S$ perpendicular to $\ell_1$, let $\ell_3$ be a tangent to $S$ 
perpendicular to $\ell_2$ and different from $\ell_1$, and so on. Assume that 
for some $n$ we have $\ell_{n+1} = \ell_1$. Then the same holds for any other 
choice of the tangent $\ell_1$.
\end{thm}
\begin{proof}[First proof]
This is the Poncelet theorem for the conic $S$ and its orthoptic conic.
\end{proof}
\begin{proof}[Second proof]
The intersection point of the lines $\ell_1$ and $\ell_3$ is the pole of 
$\ell_2$, which lies on $S^\circ$. Hence $\ell_1, \ell_3, \ell_5, \ldots$ is a 
Poncelet sequence for the conic $S$ and its polar $S^\circ$, see Figure 
\ref{fig:Poncelet}, left. By assumption, for $n$ odd it closes after $n$ steps, 
hence it closes 
after $n$ steps for any choice of $\ell_1$. For $n$ even, we have two sequences 
that close after $\frac{n}2$ steps.
\end{proof}

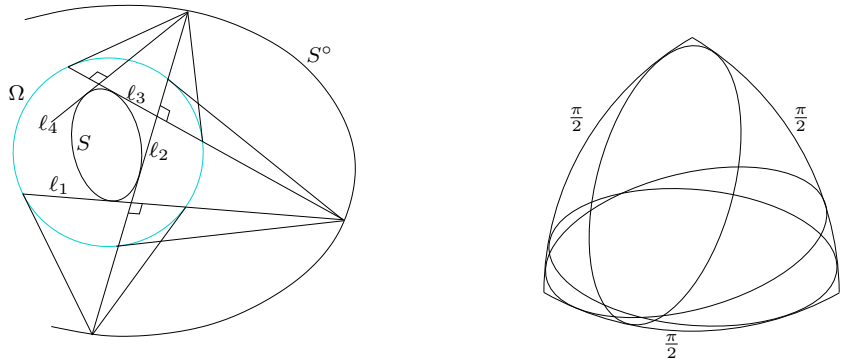
\begin{figure}[htb]
\begin{center}
\begin{picture}(0,0)%
\includegraphics{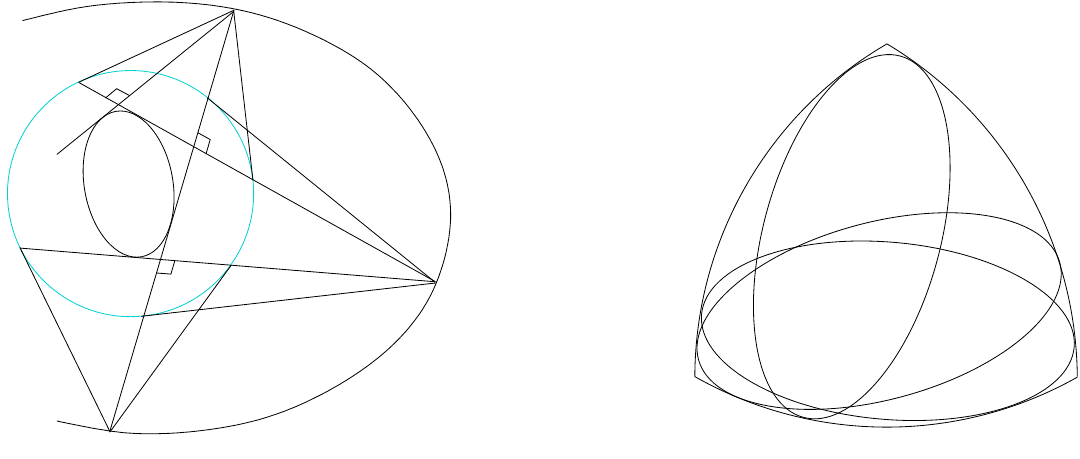}%
\end{picture}%
\setlength{\unitlength}{2072sp}%
\begingroup\makeatletter\ifx\SetFigFont\undefined%
\gdef\SetFigFont#1#2#3#4#5{%
  \reset@font\fontsize{#1}{#2pt}%
  \fontfamily{#3}\fontseries{#4}\fontshape{#5}%
  \selectfont}%
\fi\endgroup%
\begin{picture}(9860,4185)(841,-40)
\put(856,2999){\makebox(0,0)[lb]{\smash{{\SetFigFont{8}{9.6}{\rmdefault}{
\mddefault}{\updefault}{\color[rgb]{0,0,0}$\Omega$}%
}}}}
\put(1210,2644){\makebox(0,0)[lb]{\smash{{\SetFigFont{8}{9.6}{\rmdefault}{
\mddefault}{\updefault}{\color[rgb]{0,0,0}$\ell_4$}%
}}}}
\put(1335,1917){\makebox(0,0)[lb]{\smash{{\SetFigFont{8}{9.6}{\rmdefault}{
\mddefault}{\updefault}{\color[rgb]{0,0,0}$\ell_1$}%
}}}}
\put(1655,2406){\makebox(0,0)[lb]{\smash{{\SetFigFont{8}{9.6}{\rmdefault}{
\mddefault}{\updefault}{\color[rgb]{0,0,0}$S$}%
}}}}
\put(4385,3458){\makebox(0,0)[lb]{\smash{{\SetFigFont{8}{9.6}{\rmdefault}{
\mddefault}{\updefault}{\color[rgb]{0,0,0}$S^\circ$}%
}}}}
\put(7471,2729){\makebox(0,0)[lb]{\smash{{\SetFigFont{8}{9.6}{\rmdefault}{
\mddefault}{\updefault}{\color[rgb]{0,0,0}$\frac{\pi}2$}%
}}}}
\put(10126,2729){\makebox(0,0)[lb]{\smash{{\SetFigFont{8}{9.6}{\rmdefault}{
\mddefault}{\updefault}{\color[rgb]{0,0,0}$\frac{\pi}2$}%
}}}}
\put(8596, 
29){\makebox(0,0)[lb]{\smash{{\SetFigFont{8}{9.6}{\rmdefault}{\mddefault}{
\updefault}{\color[rgb]{0,0,0}$\frac{\pi}2$}%
}}}}
\put(2521,2324){\makebox(0,0)[lb]{\smash{{\SetFigFont{8}{9.6}{\rmdefault}{
\mddefault}{\updefault}{\color[rgb]{0,0,0}$\ell_2$}%
}}}}
\put(2251,2999){\makebox(0,0)[lb]{\smash{{\SetFigFont{8}{9.6}{\rmdefault}{
\mddefault}{\updefault}{\color[rgb]{0,0,0}$\ell_3$}%
}}}}
\end{picture}%
\end{center}
\caption{A non-Euclidean variation of the Poncelet theorem.}
\label{fig:Poncelet}
\end{figure}

\begin{cor}
A spherical ellipse inscribed in a regular right-angled triangle can be rotated 
while staying inscribed in that triangle. See Figure \ref{fig:Poncelet}, right.
\end{cor}

Dually, a spherical ellipse circumscribed about a regular right-angled triangle 
can be rotated while staying circumscribed. See also [Theorem 10.1.11]{GSO16}.

\section{Spherical conics}

\subsection{A spherical conic and its polar}
\label{sec:ConPol}
Algebraically, a spherical conic is a pair of quadratic forms $(\Omega, S)$ on 
a $3$-dimensional vector space, with $\Omega$ positive definite. In an 
appropriate 
basis we have
\[
\Omega(v,v) = x^2 + y^2 + z^2.
\]

Geometrically, a spherical conic is an intersection of the unit 
sphere in $\R^3$ with a quadratic cone. If the cone is 
degenerate, then the conic consists of two great circles or of one ``double'' 
great circle. If the cone is circular, then the conic is a pair of 
diametrically opposite small circles. The most interesting is the case when the 
cone is non-degenerate and non-circular, which we always assume in the sequel.

By the principal axes theorem, there is a basis in which $\Omega$ keeps its 
form, while the quadratic form defining the cone becomes diagonal:
\begin{equation}
\label{eqn:SphConCoord}
C = \{v \in \R^3 \mid S(v,v) = 0\}, \quad
S(v,v) = \frac{x^2}{a^2} + \frac{y^2}{b^2} - \frac{z^2}{c^2}, \quad a>b.
\end{equation}
The polar conic is then given by
\begin{equation}
\label{eqn:PolConCoord}
C^\circ = \{v \in \R^3 \mid S^\circ(v,v) = 0\}, \quad S^\circ(v,v) = a^2 x^2 + 
b^2 y^2 - c^2 z^2.
\end{equation}

% A conic consists of two diametrically opposite components. In the coordinate 
% system \eqref{eqn:SphConCoord} both the conic and its dual have one component in 
% the upper half-space $z > 0$ and one component in the lower 
% half-space $z < 0$.

Following Chasles, we call the $z$-axis the \emph{principal 
axis}, the $x$-axis the \emph{major axis}, and the $y$-axis the \emph{minor 
axis} of $C$. The polar cone has the same principal axis, but the major and the 
minor axes become interchanged. The points where the axes intersect the sphere 
are called the \emph{centers} of the conic.

\subsection{Projections of a spherical conic}
A spherical conic is a spatial curve of degree $4$, but from a carefully chosen 
point of view it takes a more recognizable shape.

\begin{thm}
The orthogonal projection of a spherical conic along the principal axis is an 
ellipse. The projection along the major axis 
consists of two arcs of a hyperbola, and the projection along the minor axis 
consists of two arcs of an ellipse. See Figure \ref{fig:Proj}.
\end{thm}

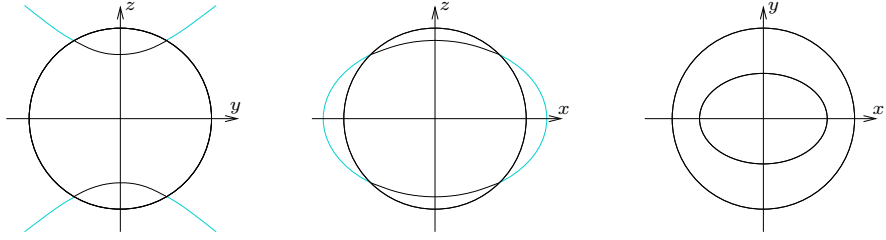
\begin{figure}[htb]
\begin{center}
\begin{picture}(0,0)%
\includegraphics{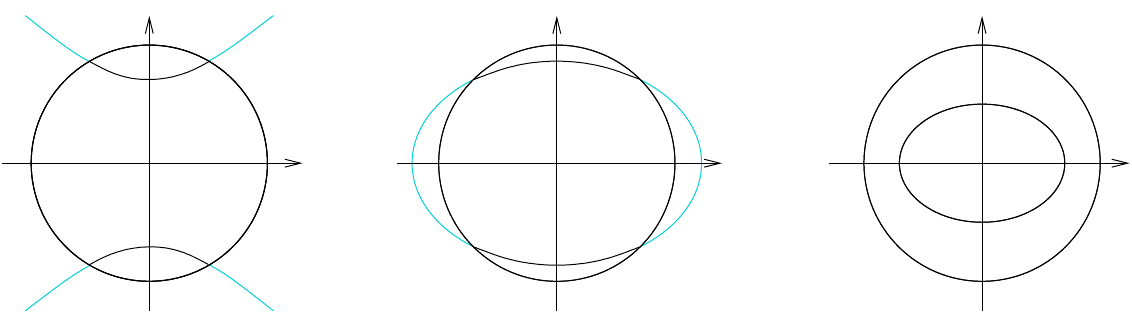}%
\end{picture}%
\setlength{\unitlength}{2486sp}%
\begingroup\makeatletter\ifx\SetFigFont\undefined%
\gdef\SetFigFont#1#2#3#4#5{%
  \reset@font\fontsize{#1}{#2pt}%
  \fontfamily{#3}\fontseries{#4}\fontshape{#5}%
  \selectfont}%
\fi\endgroup%
\begin{picture}(8619,2364)(-1136,-298)
\put(4321,884){\makebox(0,0)[lb]{\smash{{\SetFigFont{7}{8.4}{\rmdefault}{
\mddefault}{\updefault}{\color[rgb]{0,0,0}$x$}%
}}}}
\put(7426,884){\makebox(0,0)[lb]{\smash{{\SetFigFont{7}{8.4}{\rmdefault}{
\mddefault}{\updefault}{\color[rgb]{0,0,0}$x$}%
}}}}
\put(6391,1919){\makebox(0,0)[lb]{\smash{{\SetFigFont{7}{8.4}{\rmdefault}{
\mddefault}{\updefault}{\color[rgb]{0,0,0}$y$}%
}}}}
\put( 
46,1919){\makebox(0,0)[lb]{\smash{{\SetFigFont{7}{8.4}{\rmdefault}{\mddefault}{
\updefault}{\color[rgb]{0,0,0}$z$}%
}}}}
\put(3151,1919){\makebox(0,0)[lb]{\smash{{\SetFigFont{7}{8.4}{\rmdefault}{
\mddefault}{\updefault}{\color[rgb]{0,0,0}$z$}%
}}}}
\put(1081,929){\makebox(0,0)[lb]{\smash{{\SetFigFont{7}{8.4}{\rmdefault}{
\mddefault}{\updefault}{\color[rgb]{0,0,0}$y$}%
}}}}
\end{picture}%
\end{center}
\caption{Orthogonal projections of a spherical conic along the axes.}
\label{fig:Proj}
\end{figure}

\begin{proof}
The pencil of affine quadrics spanned by $\Omega = 1$ and $S = 0$ contains the 
following three cylinders.
\begin{gather*}
\left( \frac1{a^2} + \frac1{c^2} \right) x^2 + \left( \frac1{b^2} + \frac1{c^2} 
\right) y^2 = \frac1{c^2}\\
\left( \frac1{a^2} + \frac1{c^2} \right) z^2 - \left( \frac1{b^2} - \frac1{a^2} 
\right) y^2 = \frac1{a^2}\\
\left( \frac1{b^2} - \frac1{a^2} \right) x^2 + \left( \frac1{b^2} + \frac1{c^2} 
\right) z^2 = \frac1{b^2}
\end{gather*}
Each of these cylinders intersect the sphere along the conic $S$. Hence 
the projection of the conic along the axis of a cylinder is the part of an 
orthogonal section of the cylinder that lies inside the unit disk.
\end{proof}

\subsection{Foci and focal lines of a spherical conic}
\label{sec:FocSph}
Every quadratic cone has a circular section, a fact that was established by 
Descartes. Any two parallel sections (not passing through the origin) are 
similar.
\begin{dfn}
\label{dfn:CycPlane}
A \emph{cyclic plane} of a quadratic cone is a plane through the 
origin such that every parallel plane intersects the cone in a circle. 
The intersection of a cyclic plane with the unit sphere is a great circle 
called a \emph{focal line} of the spherical conic.
\end{dfn}

A non-circular cone has two cyclic planes, both passing through the major axis 
and symmetric to each other with respect to the plane spanned by the principal 
and the major axes.

\begin{dfn}
\label{dfn:CocycLine}
A \emph{cocyclic line} of a quadratic cone is the orthogonal complement of the 
cyclic plane of the polar cone. The intersection of a 
cocyclic line with the sphere is an antipodal pair of points called a 
\emph{focus} of the spherical conic.
\end{dfn}
A non-circular cone has two cocyclic lines, both of which lie in the plane 
of the principal and the major axes. Accordingly, a spherical conic has two 
foci.

\begin{thm}
\label{thm:CyclCoord}
The cyclic planes of the cone $C$ from \eqref{eqn:SphConCoord} are given 
by the equation
\[
\sqrt{\frac1{b^2} - \frac1{a^2}} y \pm \sqrt{\frac1{c^2} + \frac1{a^2}} z = 0.
\]
The cocyclic lines of the cone $C$ from \eqref{eqn:PolConCoord} are spanned by 
the vectors
\[
(\sqrt{a^2 - b^2},\, 0,\, \pm \sqrt{b^2 + c^2}).
\]
\end{thm}
\begin{proof}
The pencil of quadrics $S + \lambda \Omega$ contains a pair of planes
\[
S - \frac1{a^2} \Omega = \left( \frac1{b^2} - \frac1{a^2} \right) y^2 - 
\left( \frac1{c^2} + \frac1{a^2} \right) z^2.
\]
Hence the restriction of $S$ to each of the planes is proportional to the 
restriction of $\Omega$. This implies that the sections of $C^\circ$ 
parallel to these planes are circular.

The formulas for the cocyclic lines of $C$ follow by computing the 
cyclic planes of $C^\circ$ and applying polarity.
\end{proof}

\begin{cor}
For every spherical conic, the intersection of its cyclic planes with the plane 
of the principal and minor axes are the asymptotes of the hyperbola 
that contains the projection of the conic to that plane. See 
Figure~\ref{fig:Proj}, left.
\end{cor}

Let us now describe some geometric properties of the cocyclic lines.
\begin{thm}
\label{thm:CocycLine}
Two planes through a cocyclic line of a quadratic cone are orthogonal if and 
only if they are conjugate with respect to the cone.
\end{thm}
\begin{proof}
By Corollary \ref{cor:ConjPolar}, orthogonality with respect to $S$ 
translates into orthogonality of the polars with respect to~$S^\circ$.
Polars of planes through a cocyclic line of $C$ are lines in a cyclic 
plane of $C^\circ$. In this plane, the forms $S^\circ$ and $\Omega$ are 
proportional to each other; in particular, two lines are 
orthogonal with respect to $S$ if and only if they are orthogonal with respect 
to~$\Omega$. 
\end{proof}

\begin{figure}[htb]
\begin{center}
\begin{picture}(0,0)%
\includegraphics{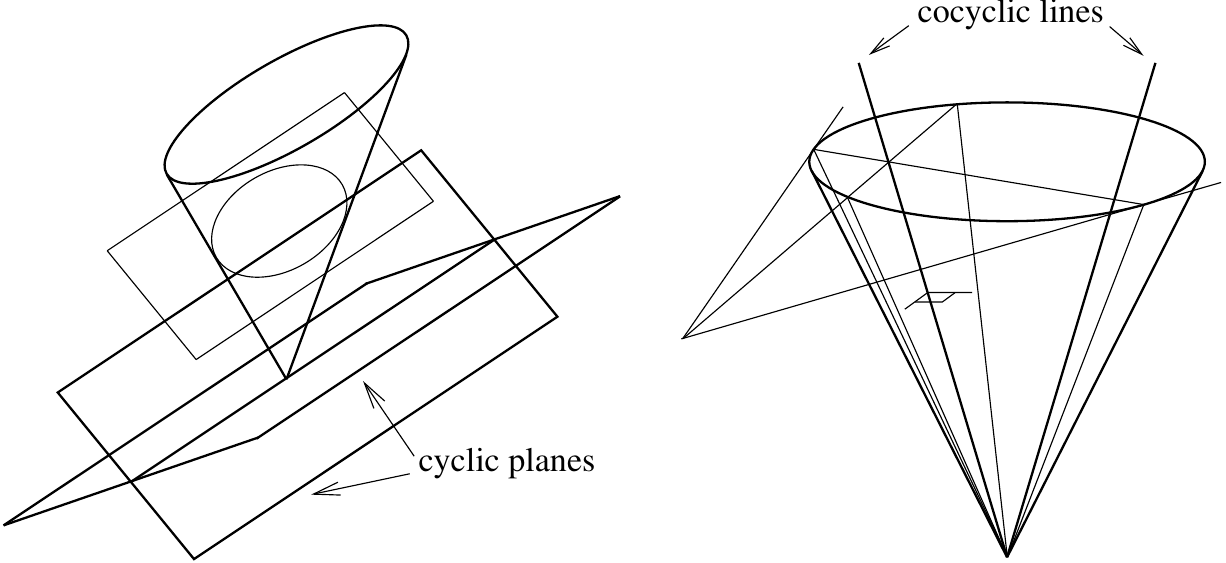}%
\end{picture}%
\setlength{\unitlength}{2901sp}%
\begingroup\makeatletter\ifx\SetFigFont\undefined%
\gdef\SetFigFont#1#2#3#4#5{%
  \reset@font\fontsize{#1}{#2pt}%
  \fontfamily{#3}\fontseries{#4}\fontshape{#5}%
  \selectfont}%
\fi\endgroup%
\begin{picture}(7987,3673)(-4577,-975)
\end{picture}%
\end{center}
\caption{Planes parallel to cyclic planes intersect a cone along a circle. 
Planes through cocyclic lines are orthogonal if and only if they are conjugate.}
\end{figure}

\begin{cor}
The plane tangent to $\Sph^2$ at a focus $F$ of a spherical conic intersects the 
corresponding cone along an ellipse with a focus at $F$.
\end{cor}
\begin{proof}
Let $\pi$ be the tangent plane to $\Sph^2$ at $F$.
Orthogonal planes through a focal line intersect $\pi$ along 
orthogonal lines. Thus by the previous theorem two lines in $\pi$ 
that pass through $F$ are orthogonal if and only if they are conjugate with 
respect to the ellipse $C \cap \pi$. This property characterizes the foci of a 
Euclidean ellipse, see \cite[\S 17.2.1.6]{BerII}
\end{proof}

The following lemma about circular sections will be particularly useful.

\begin{lem}
\label{lem:TwoCircSec}
Every sphere passing through a circular section of a quadratic cone intersects 
it along another circle. Conversely, any two non-parallel circular sections of 
a quadratic cone are contained in a sphere.
\end{lem}
\begin{proof}
Consider the inversion with center at the apex of the cone that sends the 
sphere to itself. Since the inversion also fixes the cone, it 
exchanges the two components of the intersection. Therefore, if one component is 
a circle, so is the other.

For the second part, find an inversion that sends one circular section to the 
other. Then an invariant sphere passing through one of the circles passes also
through the other.
\end{proof}

In the limit, as one of the circular sections tends to a point, we obtain the 
following.

\begin{cor}
Every sphere tangent to a cyclic plane at the apex of the cone intersects the 
cone in a circle.
\end{cor}

\subsection{Directors and directrices}
\begin{dfn}
\label{dfn:SpherDir}
A line through the origin conjugate with respect to a quadra\-tic cone to one 
of 
its cyclic planes is called a \emph{director line} of the cone. The 
intersection of a director line with the sphere (that is, the pair of poles of 
a focal line with respect to the conic) is called a \emph{director} of a 
spherical conic.
\end{dfn}

In other words, the director lines are formed by the centers of circular 
sections of the cone.

\begin{dfn}
\label{dfn:SpherDira}
A plane through the origin conjugate with respect to a quadratic cone to one of 
its cocyclic lines is called a \emph{directrix plane} of the cone.
The intersection of a directrix plane with the sphere (that is, the polar of a 
focus with respect to a conic) is called a 
\emph{directrix} of a spherical conic.
\end{dfn}

By Corollary \ref{cor:ConjPolar}, the directrix of a conic is polar to the 
director point of the polar conic.

From the formulas of Theorem \ref{thm:CyclCoord} it follows that the director 
lines of the cone $C$  are spanned by the vectors
\[
\left( 0,\, b\sqrt{1-\frac{b^2}{a^2}},\, \pm c\sqrt{1+ \frac{c^2}{b^2}} 
\right),
\]
and its directrix planes have the equations
\[
\frac{\sqrt{a^2-b^2}}{a^2} x \pm \frac{\sqrt{b^2+c^2}}{c^2} z = 0.
\]

\subsection{Families of spherical conics}
\label{sec:FamSph}
A spherical conic $S$ intersects the absolute conic $\Omega$ in four imaginary 
points that form two complex conjugate pairs (the points are distinct since we 
assumed the cone to be non-circular).
\begin{thm}
\label{thm:SphFocLines}
The focal lines of a spherical conic are the lines through the complex conjugate
pairs of intersection points of the conic with the absolute.
\end{thm}
\begin{proof}
This follows from the computation in the proof of Theorem 
\ref{thm:CyclCoord}. But here is an alternative coordinate-free argument. A 
cyclic plane of a quadratic cone is a plane on which the restrictions of 
$\Omega$ and $S$ are proportional. In particular, $\Omega$ and $S$ must vanish 
at the same time. This implies that a cyclic plane is spanned by two common 
(imaginary) isotropic lines of $\Omega$ and $S$. Projectively, a 
cyclic line is a line through two intersection points of $\Omega$ and $S$. For 
this line to be real, the points must be complex conjugate.
\end{proof}

As a consequence, conics that share the focal lines with the conic $S$ form a 
pencil of conics through four imaginary points, spanned by $\Omega$ and $S$. In 
coordinates, this pencil is given by the equations
\begin{equation}
\label{eqn:SphConcyclic}
\left( \frac1{a^2} - \lambda \right) x^2 + \left( \frac1{b^2} - \lambda \right) 
y^2 - \left( \frac1{c^2} + \lambda \right) z^2 = 0.
\end{equation}

One can visualize the conics with common focal lines as follows.
\begin{thm}
Every hyperboloid asymptotic to the cone over a spherical conic $S$ intersects 
the sphere along a spherical conic that shares the focal lines with $S$.
\end{thm}
\begin{proof}
A hyperboloid asymptotic to the cone over $S$ is given by 
an equation of the form
\[
\frac{x^2}{a^2} + \frac{y^2}{b^2} - \frac{z^2}{c^2} = \lambda.
\]
By taking a linear combination with the equation of the sphere, we see that the 
hyperboloid intersects the sphere along the same curve as the 
cone~\eqref{eqn:SphConcyclic}.
\end{proof}

By definition of the foci and directrices, we have the following.
\begin{thm}
\label{thm:SphFoci}
A focus of a spherical conic is the intersection point of two (complex 
conjugate) common tangents to the conic and the absolute. The corresponding 
directrix is the line through the points where these tangents touch the conic.
\end{thm}

That is, conics confocal with $S$ form a dual pencil to the one 
spanned by $S^\circ$ and $\Omega$ and are given by the equations
\begin{equation}
\label{eqn:ConfocSph}
\frac{x^2}{a^2-\lambda} + \frac{y^2}{b^2-\lambda} - 
\frac{z^2}{c^2+\lambda} = 0, \quad \lambda \in (-1, b^2) \cap (b^2, a^2).
\end{equation}

\begin{cor}
\label{cor:FocDirPencil}
Conics that share a focus and a corresponding directrix form a double contact 
pencil (determined by two complex conjugate points on two complex conjugate 
lines).
\end{cor}

\section{Theorems about spherical conics}
\subsection{Ivory's lemma}
\label{sec:IvorySph}
On the turn of the 18th century, Laplace and Ivory \cite{Iv09} 
computed the gravitational field created by a solid homogeneous ellipsoid. 
Ivory's solution was immediately and widely recognized for its elegance, see 
\cite{Todhunter}[\S\S 1141, 1146]. Chasles \cite{Chasles60} further simplified 
Ivory's argument; it is under this form that it is presented nowadays, see e.~g.
\cite[Lecture 30]{FT}. In particular, Chasles stated explicitely what is now 
called the Ivory lemma.

Ivory's lemma is a statement about confocal conics and quadrics. Its original 
version deals with confocal quadrics in the Euclidean $3$-space, but it 
holds in any dimension with respect to any Cayley--Klein metric, see 
\cite{SW04}.

In this section we prove Ivory's lemma on the $2$-sphere.

\begin{thm}[Ivory's lemma]
\label{thm:IvorySph}
The diagonals in a quadrilateral formed by four confocal spherical conics 
have equal lengths.
\end{thm}

Figure \ref{fig:IvorySph} illustrates Ivory's lemma in the Euclidean plane. It 
can also be interpreted as a distorted view of confocal spherical conics.
\begin{figure}[htb]
\begin{center}
\includegraphics[width=.4\textwidth]{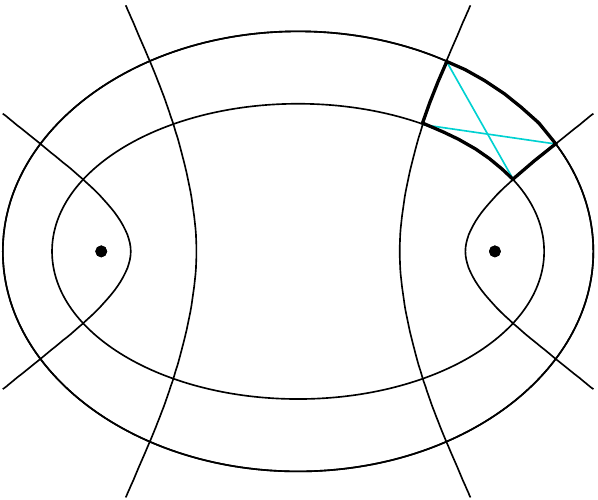}
\end{center}
\caption{Ivory's lemma.}
\label{fig:IvorySph}
\end{figure}

Take our quadratic cone $C$ and a cone $C_\lambda$ given by equation 
\eqref{eqn:ConfocSph} with $ -1 < \lambda < b^2$, which ensures that the 
corresponding confocal conics don't intersect.
Consider a linear map $\R^3 \to \R^3$ given by the diagonal matrix
\[
F_\lambda = \diag\left( \frac{\sqrt{a^2-\lambda}}{a}, 
\frac{\sqrt{b^2-\lambda}}{b}, \frac{\sqrt{c^2+\lambda}}{c} \right).
\]

\begin{lem}
\label{lem:IvoryLemSph}
The map $F_\lambda$ sends the spherical conic $S = C \cap \Sph^2$ to the conic 
$S_\lambda = C_\lambda \cap \Sph^2$. Besides, if a point $v \in S$ belongs to a 
conic 
$S_\mu$, then $F_\lambda(v)$ also belongs to $S_\mu$.
\end{lem}
\begin{proof}
% We have to show that $v \in \Sph^2 \cap C \cap C_\mu$ implies $F_\lambda(v) \in 
% \Sph^2 \cap C_\lambda \cap C_\mu$.
Denote $A = \diag(a^2, b^2, -c^2)$. Then $C$, $C_\lambda$, and $C_\mu$ 
are the isotropic cones of the matrices
\[
S = A^{-1}, \quad S_\lambda = (A - \lambda\Id)^{-1}, \quad S_\mu = (A - 
\mu\Id)^{-1},
\]
respectively. By definition of $F_\lambda$ we have
\[
F_\lambda^2 = (A-\lambda\Id)A^{-1},
\]
which immediately implies $F_\lambda(C) = C_\lambda$. Further, observe that
\[
\|F_\lambda(v)\|^2 = v^\top (A-\lambda\Id)A^{-1} v = v^\top v - \lambda v^\top 
A^{-1} v = \|v\|^2 - \lambda v^\top S v = 1,
\]
provided that $v \in \Sph^2 \cap C$. Hence, $v \in S$ implies $v 
\in S_\lambda$.

We also have
\begin{multline*}
F_\lambda S_\mu F_\lambda = A^{-1} (A-\lambda\Id) (A-\mu\Id)^{-1} =
(A-\mu\Omega)^{-1} - \lambda A^{-1}(A-\mu\Id)^{-1}\\
= (A- \mu\Id)^{-1} - \frac\lambda\mu ((A-\mu\Id)^{-1} - A^{-1}) = 
\frac\lambda\mu A^{-1} + \left(1-\frac\lambda\mu\right) (A-\mu\Id)^{-1}.
\end{multline*}
In other words,
\[
F_\lambda(v)^\top S_\mu F_\lambda(v) = \frac\lambda\mu v^\top S v + 
\left(1-\frac\lambda\mu\right) v^\top S_\mu v,
\]
so that $v \in C \cap C_\mu$ implies $v \in C_\mu$.
\end{proof}

\begin{proof}[Proof of Theorem \ref{thm:IvorySph}]
Let $S$, $S_\lambda$, $S_{\mu_1}$, and $S_{\mu_2}$ be confocal conics such that 
$S$ and $S_\lambda$ are disjoint, and $S_{\mu_1}$, $S_{\mu_2}$ intersect them. 
Let 
$v_i \in S \cap S_{\mu_i}$ and $w_i \in S_\lambda \cap S_{\mu_i}$, $i=1,2$. We 
want to show that $\dist(v_1, w_2) = \dist(v_2, w_1)$.

By Lemma \ref{lem:IvoryLemSph}, $w_i = F_\lambda(v_i)$. Since the matrix 
$F_\lambda$ is diagonal, we have
\[
v_1^\top w_2 = v_1^\top F_\lambda v_2 = w_1^\top v_2,
\]
which implies the desired equality of diagonals.
\end{proof}

\begin{rem}
The most general form of the Ivory lemma was discovered by Blaschke
\cite{Bla28}: it holds in the coordinate nets of St\"ackel metrics (in 
dimension~$2$ known as Liouville nets). Conversely, if a coordinate net on a 
Riemannian manifold has the equal diagonals property, then it is a St\"ackel 
net. The argument uses integrals of the geodesic flow and is inspired by 
Jacobi's ideas from his study of geodesics on ellipsoids. For a modern 
exposition, see \cite{Thimm78,IT16}.
\end{rem}

\subsection{Bifocal properties of spherical conics}
% A spherical conic consists of two diametrically opposite closed curves and has 
% two diametrically opposite pairs of foci, each pair lying in the region bounded 
% by one of the components of the conic. It is convenient to work in a hemisphere 
% that contains one of the components or to take the quotient $\RP^2$ 
% of $\Sph^2$. Either of these conventions is tacitly subsumed in the theorems 
% that follow.

On the illustrations to the following theorems we are using the 
projective model of the sphere: the projection of the sphere from its 
center to a tangent plane. A spherical conic becomes an affine conic, all 
great circles (in particular, the focal lines of the conic) become straight 
lines. As a projection plane, it is convenient to choose a plane that is 
tangent at one of the centers of the conic, see Figure \ref{fig:CKSphere}.

\begin{figure}[htb]
\begin{center}
\includegraphics[width=.25\textwidth]{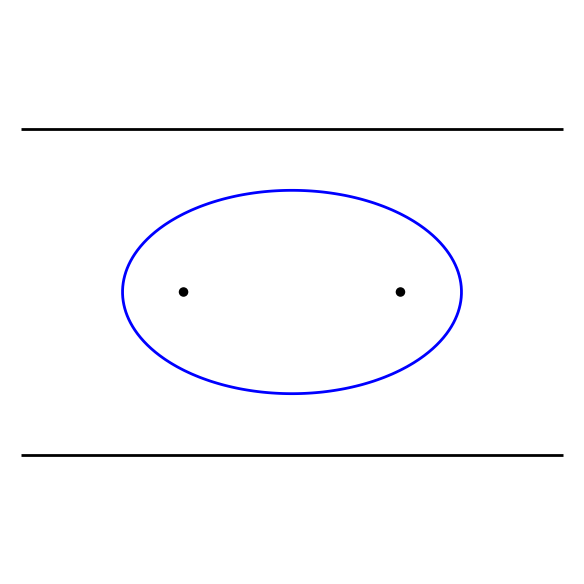} 
\hspace{.05\textwidth}
\includegraphics[width=.25\textwidth]{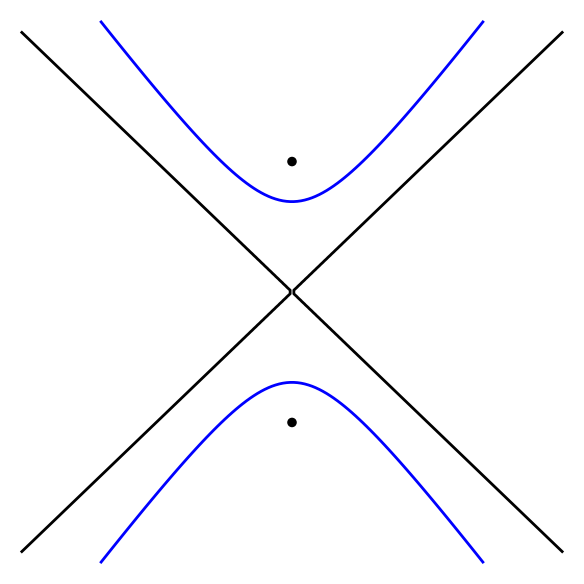} 
\hspace{.05\textwidth}
\includegraphics[width=.25\textwidth]{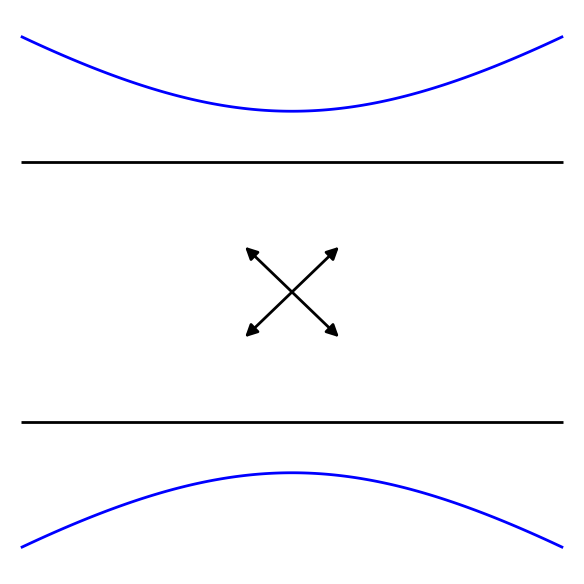} 
\end{center}
\caption{Projections of a spherical conic to the planes tangent to its centers.}
\label{fig:CKSphere}
\end{figure}

The focal lines on Figure \ref{fig:CKSphere}, middle, are not the asymptotes of 
the hyperbola; the arrows in the center of Figure \ref{fig:CKSphere}, right, 
indicate the position of the foci, which belong to the line at infinity.

Each of the theorems below consists of two parts, dual to each other 
via the absolute polarity. Therefore we are proving only one part of each 
theorem. Usually this is the part that deals with the focal lines, because in 
the space it deals with circles and spheres and can be proved by an elegant
synthetic argument. All proofs were given by Chasles in \cite{ChGr}.

\begin{thm}
\label{thm:Bisect}
\begin{enumerate}
\item
For every tangent to a spherical conic, the point of tangency bisects the 
segment comprised between the focal lines.
\item
The lines joining the foci of a spherical conic with a point on the conic form 
equal angles with the tangent at that point.
\end{enumerate}
\end{thm}

\begin{figure}[htb]
\begin{center}
\includegraphics[width=.8\textwidth]{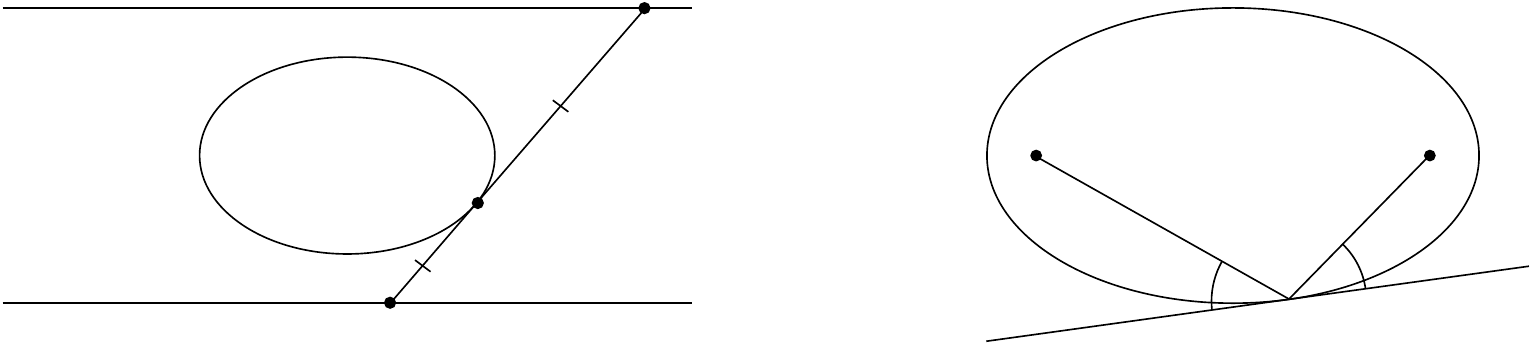}
\end{center}
\caption{The bisector property.}
\end{figure}

\begin{proof}
Interpreted in terms of the corresponding quadratic cone, the first statement 
becomes as follows: any plane tangent to the cone intersects the cyclic 
planes along the lines that make equal angles with the line of tangency.

Take two planes parallel to the cyclic planes; they intersect the cone in two 
non-parallel circles. By Lemma \ref{lem:TwoCircSec}, these circular sections 
are contained in a sphere. A tangent plane to the cone intersects the planes of 
the circles in two lines tangent to the sphere, see Figure \ref{fig:TangSec}, 
left. We need to show that these lines make equal angles with the line of 
tangency. But any two tangents to the sphere make equal angles with the segment 
joining their points of tangency, and the theorem is proved.
\end{proof}

Theorem \ref{thm:Bisect} can be viewed as a limit case of the following.

\begin{thm}
\label{thm:Bisect1}
\begin{enumerate}
\item
For every secant line of a spherical conic the segments comprised between the 
conic and the focal lines have equal lengths.
\item
For every point outside of a spherical conic, the angle between a tangent 
through this point and the line joining the point to a focus is equal to the 
angle between the other tangent and the line to the other focus.
\end{enumerate}
\end{thm}

\begin{figure}[htb]
\begin{center}
\includegraphics[width=.8\textwidth]{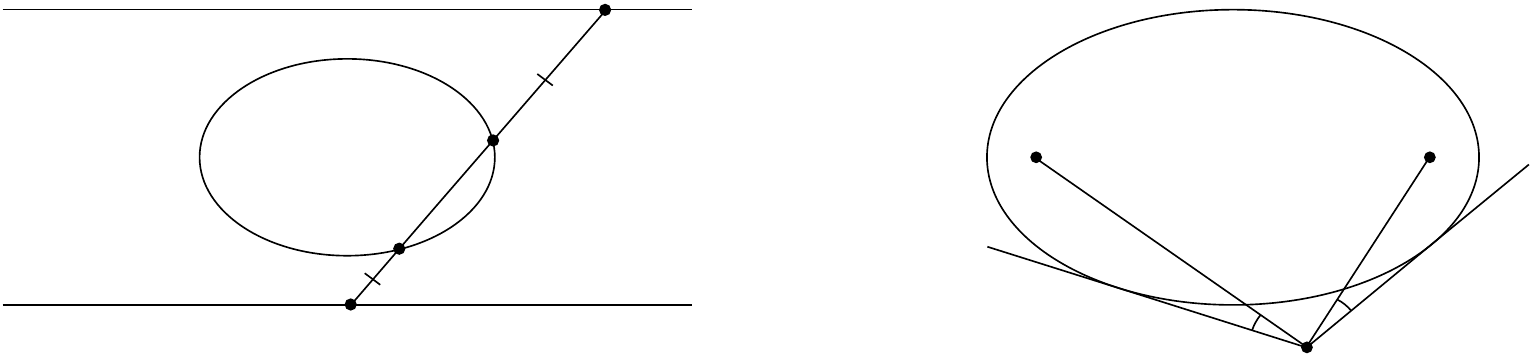}
\end{center}
\caption{The generalized bisector property.}
\end{figure}

\begin{proof}
The first statement says that every plane 
through two generatrices of the cone intersects the cyclic plane along two 
lines such that the angle between a line and a generatrix is equal to the angle 
between the other line and the other generatrix.

Translate the cyclic planes so that they intersect the cone along two circles. 
By Lemma \ref{lem:TwoCircSec} these two circles are contained in a sphere. The 
generatrices and the intersection lines of the secant plane with the planes 
spanned by the circles form a planar quadrilateral inscribed in the sphere, see 
Figure \ref{fig:TangSec}, right. 
Two opposite angles of this quadrilateral complement each other to $\pi$, which 
implies the theorem.
\end{proof}

\begin{figure}[htb]
\begin{center}
\includegraphics[width=\textwidth]{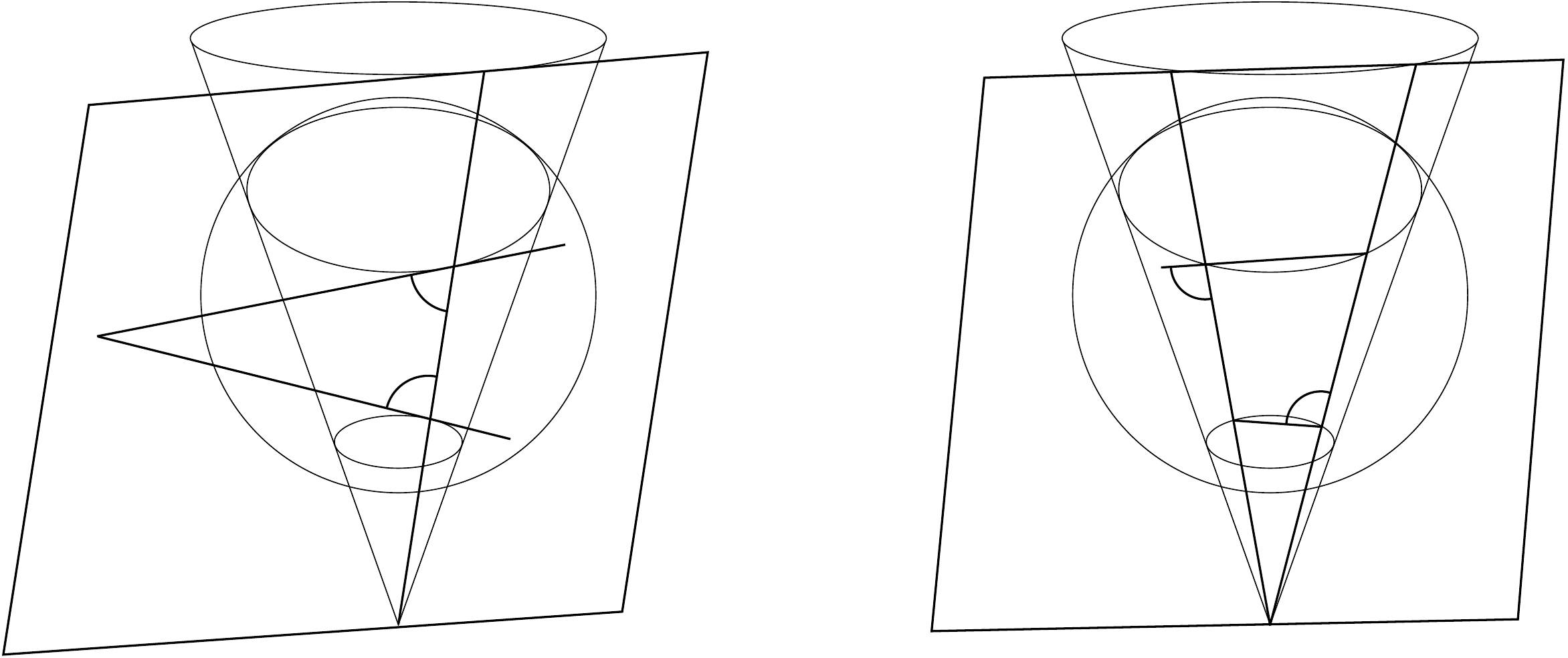}
\end{center}
\caption{Proofs of Theorems \ref{thm:Bisect} and \ref{thm:Bisect1}.}
\label{fig:TangSec}
\end{figure}

\begin{thm}
\label{thm:ConstSum}
\begin{enumerate}
\item
A tangent to a spherical conic cuts from the lune formed by the focal lines a 
triangle of constant area.
\item
The sum of the distances from a point on a spherical conic to its foci is 
constant.
\end{enumerate}
\end{thm}

The first part is dual to the second, because the area of a spherical triangle 
is equal to its angle sum minus $\pi$. Since the angle at the vertex $E$ is 
constant, the statement is equivalent to the constancy of the sum of the angles 
at $X$ and $Y$. These angles are equal to the lengths of $F_1Z$ and $F_2Z$ in 
the second part of the theorem.

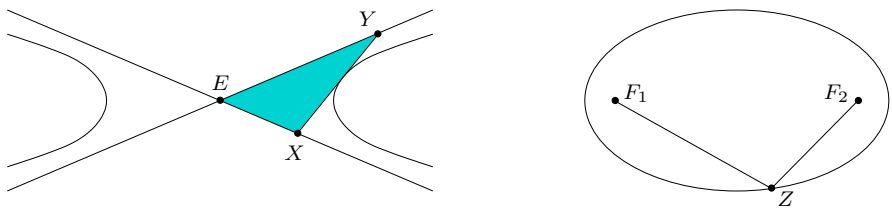
\begin{figure}[htb]
\begin{center}
\begin{picture}(0,0)%
\includegraphics{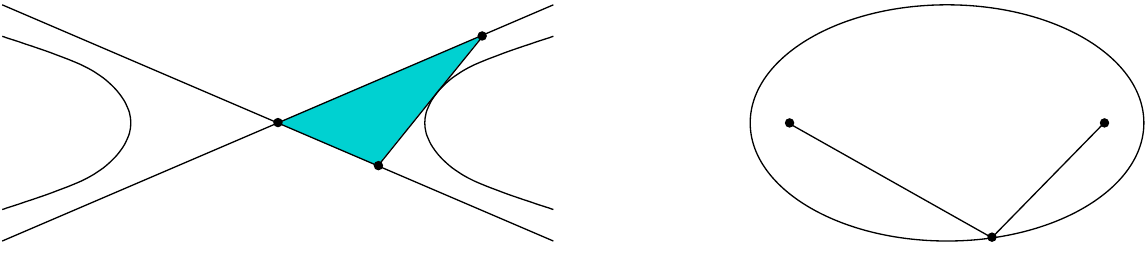}%
\end{picture}%
\setlength{\unitlength}{3315sp}%
\begingroup\makeatletter\ifx\SetFigFont\undefined%
\gdef\SetFigFont#1#2#3#4#5{%
  \reset@font\fontsize{#1}{#2pt}%
  \fontfamily{#3}\fontseries{#4}\fontshape{#5}%
  \selectfont}%
\fi\endgroup%
\begin{picture}(6545,1533)(-6311,-7)
\put(-254,851){\makebox(0,0)[lb]{\smash{{\SetFigFont{8}{9.6}{\rmdefault}{
\mddefault}{\updefault}{\color[rgb]{0,0,0}$F_2$}%
}}}}
\put(-597, 
48){\makebox(0,0)[lb]{\smash{{\SetFigFont{8}{9.6}{\rmdefault}{\mddefault}{
\updefault}{\color[rgb]{0,0,0}$Z$}%
}}}}
\put(-1738,851){\makebox(0,0)[lb]{\smash{{\SetFigFont{8}{9.6}{\rmdefault}{
\mddefault}{\updefault}{\color[rgb]{0,0,0}$F_1$}%
}}}}
\put(-4786,919){\makebox(0,0)[lb]{\smash{{\SetFigFont{8}{9.6}{\rmdefault}{
\mddefault}{\updefault}{\color[rgb]{0,0,0}$E$}%
}}}}
\put(-4247,392){\makebox(0,0)[lb]{\smash{{\SetFigFont{8}{9.6}{\rmdefault}{
\mddefault}{\updefault}{\color[rgb]{0,0,0}$X$}%
}}}}
\put(-3697,1392){\makebox(0,0)[lb]{\smash{{\SetFigFont{8}{9.6}{\rmdefault}{
\mddefault}{\updefault}{\color[rgb]{0,0,0}$Y$}%
}}}}
\end{picture}%
\end{center}
\caption{Bifocal properties: $\Area(\triangle EXY) = \const$, $F_1Z + F_2Z = 
\const$.}
\label{fig:ConstSum}
\end{figure}

The second part is due to 
Magnus \cite{Mag25}, who proved it by a direct computation. Two different 
proofs are given below: the first one uses differentiation, the second one 
synthetic geometry. The latter is due to Chasles and uses Theorem 
\ref{thm:QuadCircle} below.

\begin{proof}[First proof]
When we move a line keeping it tangent to the conic, the instantaneous 
change in the area of the triangle is zero because, by the first part of
Theorem~\ref{thm:Bisect}, the point of tangency bisects the segment $XY$. 
Similarly, for a point moving along the conic, the derivative of the sum of its 
distances from the foci is zero by the second part of Theorem \ref{thm:Bisect}.
\end{proof}

\begin{proof}[Second proof]
The second part of Theorem \ref{thm:ConstSum} can be derived from the second 
part of Theorem \ref{thm:QuadCircle}
similarly to the proof that in a circumscribed quadrilateral the 
sums of opposite pairs of sides are equal. Take two points $A$ and $B$ on a 
conic such that the segments $F_1A$ and $F_2B$ intersect. By Theorem 
\ref{thm:QuadCircle} 
there is a circle tangent to the lines $F_1A, F_1B, F_2A, F_2B$. Since tangent 
segments drawn from a point to a circle have equal lengths, we have (see Figure 
\ref{fig:QuadCircLem})
\[
F_1A + F_2A = F_1K + F_2L = F_1M + F_2N = F_1B + F_2B.
\]
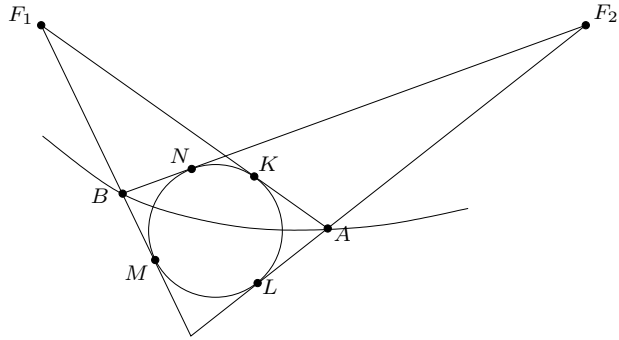
\begin{figure}[htb]
\begin{center}
\begin{picture}(0,0)%
\includegraphics{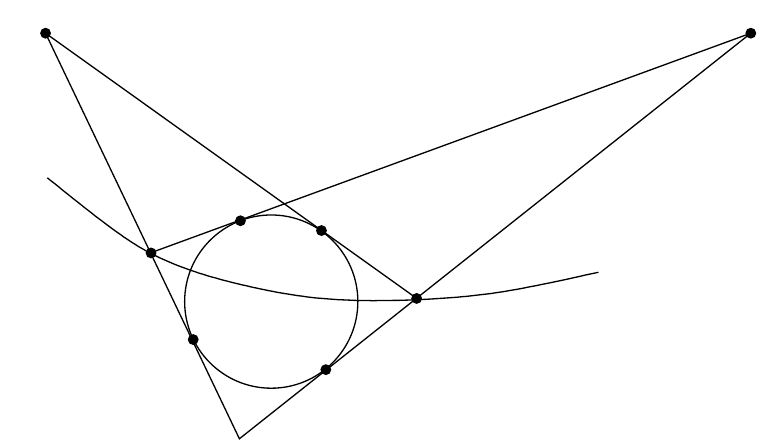}%
\end{picture}%
\setlength{\unitlength}{3315sp}%
\begingroup\makeatletter\ifx\SetFigFont\undefined%
\gdef\SetFigFont#1#2#3#4#5{%
  \reset@font\fontsize{#1}{#2pt}%
  \fontfamily{#3}\fontseries{#4}\fontshape{#5}%
  \selectfont}%
\fi\endgroup%
\begin{picture}(4362,2512)(-4318,-1520)
\put(-4303,844){\makebox(0,0)[lb]{\smash{{\SetFigFont{8}{9.6}{\rmdefault}{
\mddefault}{\updefault}{\color[rgb]{0,0,0}$F_1$}%
}}}}
\put( 
29,857){\makebox(0,0)[lb]{\smash{{\SetFigFont{8}{9.6}{\rmdefault}{\mddefault}{
\updefault}{\color[rgb]{0,0,0}$F_2$}%
}}}}
\put(-1893,-798){\makebox(0,0)[lb]{\smash{{\SetFigFont{8}{9.6}{\rmdefault}{
\mddefault}{\updefault}{\color[rgb]{0,0,0}$A$}%
}}}}
\put(-3690,-519){\makebox(0,0)[lb]{\smash{{\SetFigFont{8}{9.6}{\rmdefault}{
\mddefault}{\updefault}{\color[rgb]{0,0,0}$B$}%
}}}}
\put(-2452,-279){\makebox(0,0)[lb]{\smash{{\SetFigFont{8}{9.6}{\rmdefault}{
\mddefault}{\updefault}{\color[rgb]{0,0,0}$K$}%
}}}}
\put(-2422,-1193){\makebox(0,0)[lb]{\smash{{\SetFigFont{8}{9.6}{\rmdefault}{
\mddefault}{\updefault}{\color[rgb]{0,0,0}$L$}%
}}}}
\put(-3106,-217){\makebox(0,0)[lb]{\smash{{\SetFigFont{8}{9.6}{\rmdefault}{
\mddefault}{\updefault}{\color[rgb]{0,0,0}$N$}%
}}}}
\put(-3442,-1087){\makebox(0,0)[lb]{\smash{{\SetFigFont{8}{9.6}{\rmdefault}{
\mddefault}{\updefault}{\color[rgb]{0,0,0}$M$}%
}}}}
\end{picture}%
\end{center}
\caption{Second proof of Theorem \ref{thm:ConstSum}.}
\label{fig:QuadCircLem}
\end{figure}
% 
% The first part of Theorem \ref{thm:ConstSum} follows from the second part by 
% duality. Alternatively one can deal with cyclic quadrilaterals: the four points 
% equidistant from a line on Figure \ref{fig:QuadCircle} actually lie on a pair 
% of diametrically opposite circles (whose centers are the poles of the line). 
% Replacing two of them by their antipodes we obtain a cyclic quadrilateral. It 
% remains to use the fact that the sums of opposite pairs of angles in such a 
% quadrilateral are equal.
\end{proof}

\begin{thm}
\label{thm:QuadCircle}
\begin{enumerate}
\item
Two tangents to a spherical conic intersect the focal lines in four points 
that are equidistant from the line through the points of tangency.
\item
Four lines joining the foci of a spherical conic with two points on the conic 
are tangent to a circle. The center of this circle is the pole of the line 
through the two points on the conic.
\end{enumerate}
\end{thm}

\begin{figure}[htb]
\begin{center}
\begin{picture}(0,0)%
\includegraphics{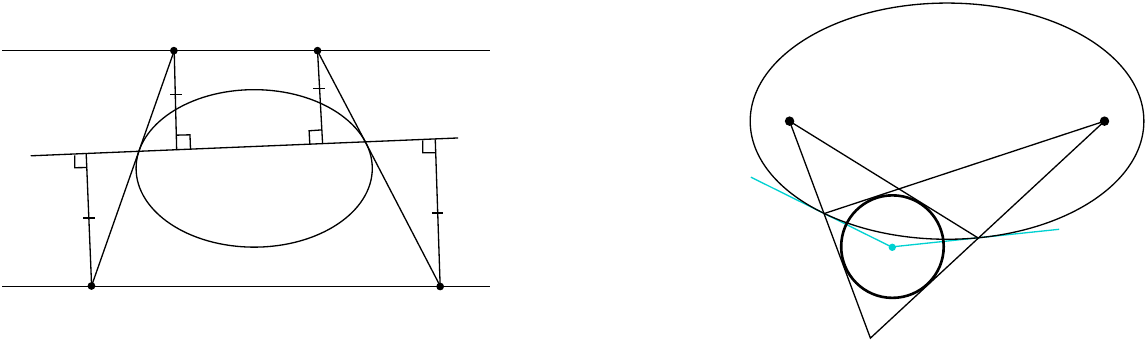}%
\end{picture}%
\setlength{\unitlength}{3315sp}%
\begingroup\makeatletter\ifx\SetFigFont\undefined%
\gdef\SetFigFont#1#2#3#4#5{%
  \reset@font\fontsize{#1}{#2pt}%
  \fontfamily{#3}\fontseries{#4}\fontshape{#5}%
  \selectfont}%
\fi\endgroup%
\begin{picture}(6545,1934)(-6311,-143)
\put(-3628,961){\makebox(0,0)[lb]{\smash{{\SetFigFont{8}{9.6}{\rmdefault}{
\mddefault}{\updefault}{\color[rgb]{0,0,0}$L$}%
}}}}
\put(-5389,1564){\makebox(0,0)[lb]{\smash{{\SetFigFont{8}{9.6}{\rmdefault}{
\mddefault}{\updefault}{\color[rgb]{0,0,0}$p_1$}%
}}}}
\put(-4613,1568){\makebox(0,0)[lb]{\smash{{\SetFigFont{8}{9.6}{\rmdefault}{
\mddefault}{\updefault}{\color[rgb]{0,0,0}$p_2$}%
}}}}
\end{picture}%
\end{center}
\caption{Illustration to Theorem \ref{thm:QuadCircle}.}
\label{fig:QuadCircle}
\end{figure}
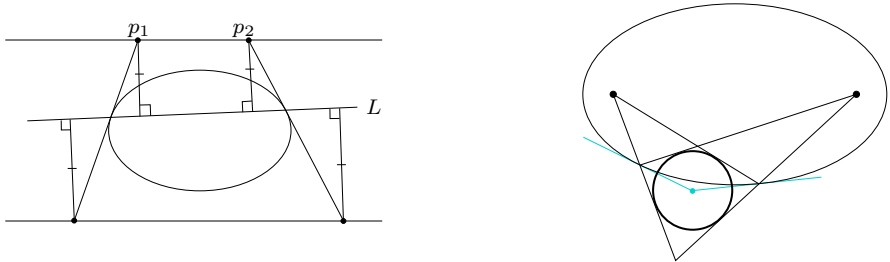
\begin{proof}
In $\R^3$, the first statement says that 
for any two tangent planes to the cone their intersection lines with the cyclic 
planes form equal angles with the plane through the lines of tangency. Theorem 
\ref{thm:Bisect} implies this for the intersection lines of the same tangent 
plane with different cyclic planes. Let us prove this for the intersection 
lines of different tangent planes with the same cyclic plane. Figure 
\ref{fig:QuadCircleProof} shows two such lines $p_1$ and $p_2$. The shaded 
triangle lies in the plane $L$ spanned by the tangent lines, compare 
Figure~\ref{fig:QuadCircle}, left.

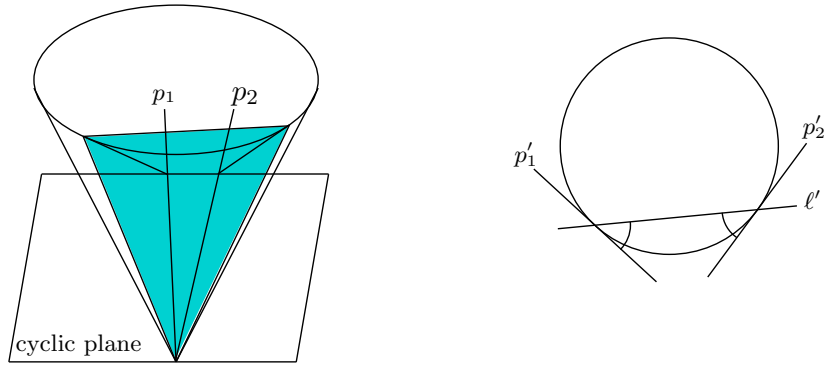
\begin{figure}[htb]
\begin{center}
\begin{picture}(0,0)%
\includegraphics{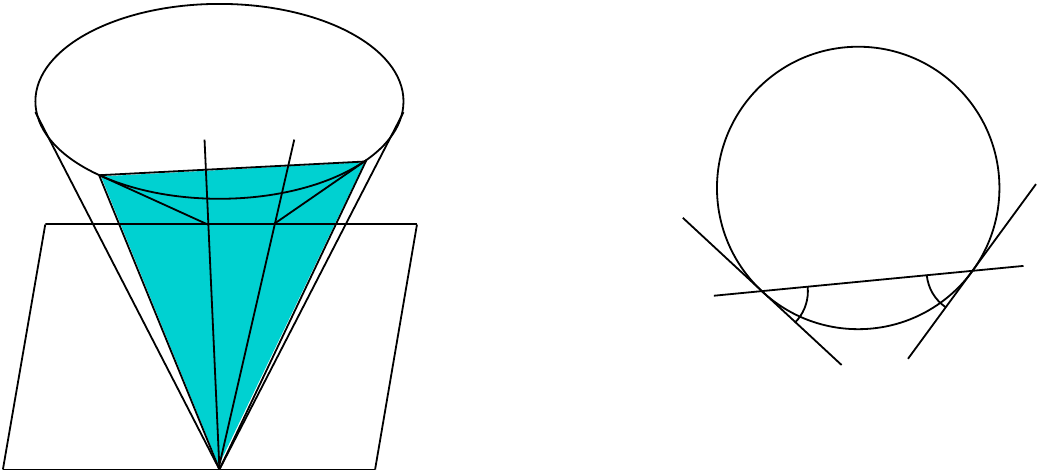}%
\end{picture}%
\setlength{\unitlength}{4558sp}%
\begingroup\makeatletter\ifx\SetFigFont\undefined%
\gdef\SetFigFont#1#2#3#4#5{%
  \reset@font\fontsize{#1}{#2pt}%
  \fontfamily{#3}\fontseries{#4}\fontshape{#5}%
  \selectfont}%
\fi\endgroup%
\begin{picture}(4318,1955)(-911,827)
\put(-125,2260){\makebox(0,0)[lb]{\smash{{\SetFigFont{9}{10.8}{\rmdefault}{
\mddefault}{\updefault}{\color[rgb]{0,0,0}$p_1$}%
}}}}
\put(301,2260){\makebox(0,0)[lb]{\smash{{\SetFigFont{11}{13.2}{\rmdefault}{
\mddefault}{\updefault}{\color[rgb]{0,0,0}$p_2$}%
}}}}
\put(-861,899){\makebox(0,0)[lb]{\smash{{\SetFigFont{9}{10.8}{\rmdefault}{
\mddefault}{\updefault}{\color[rgb]{0,0,0}cyclic plane}%
}}}}
\put(3382,1667){\makebox(0,0)[lb]{\smash{{\SetFigFont{9}{10.8}{\rmdefault}{
\mddefault}{\updefault}{\color[rgb]{0,0,0}$\ell'$}%
}}}}
\put(3375,2075){\makebox(0,0)[lb]{\smash{{\SetFigFont{9}{10.8}{\rmdefault}{
\mddefault}{\updefault}{\color[rgb]{0,0,0}$p_2'$}%
}}}}
\put(1822,1925){\makebox(0,0)[lb]{\smash{{\SetFigFont{9}{10.8}{\rmdefault}{
\mddefault}{\updefault}{\color[rgb]{0,0,0}$p_1'$}%
}}}}
\end{picture}%
\end{center}
\caption{Proof of Theorem \ref{thm:QuadCircle}.}
\label{fig:QuadCircleProof}
\end{figure}

Translate the cyclic plane parallelly; it will 
intersect the cone along a circle, the tangent planes along lines $p_1'$ and 
$p_2'$ parallel to $p_1$ and $p_2$, and the plane $L$ along a line $\ell'$. The 
lines $p_1'$ and $p_2'$ make equal angles with the line $\ell'$, hence they 
make equal angles with any plane through this line, in particular with $L$. The 
theorem is proved.
\end{proof}

The second part of Theorem \ref{thm:QuadCircle} is the third part of Theorem 
\ref{thm:Chasles}.

\begin{thm}
\label{thm:SinSin}
\begin{enumerate}
\item
The product of the sines of distances from the points on a spherical conic to 
the focal lines is constant.
\item
The product of the sines of distances from the foci of a spherical conic to its 
tangents is constant.
\end{enumerate}
\end{thm}
\begin{figure}[htb]
\begin{center}
\begin{picture}(0,0)%
\includegraphics{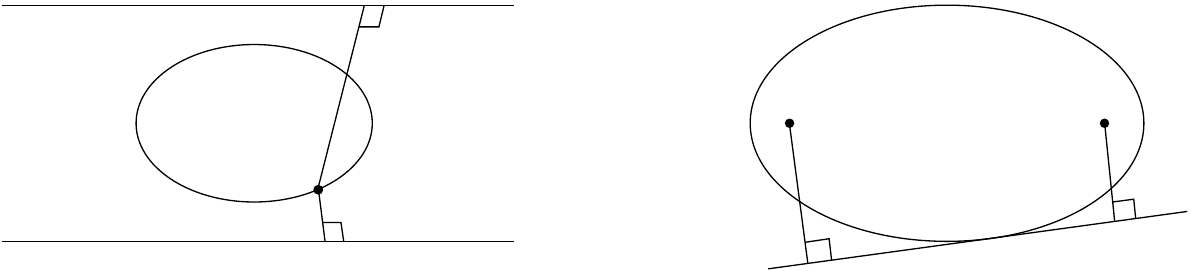}%
\end{picture}%
\setlength{\unitlength}{3315sp}%
\begingroup\makeatletter\ifx\SetFigFont\undefined%
\gdef\SetFigFont#1#2#3#4#5{%
  \reset@font\fontsize{#1}{#2pt}%
  \fontfamily{#3}\fontseries{#4}\fontshape{#5}%
  \selectfont}%
\fi\endgroup%
\begin{picture}(6797,1531)(-6311,-5)
\put(-4466,919){\makebox(0,0)[lb]{\smash{{\SetFigFont{8}{9.6}{\rmdefault}{
\mddefault}{\updefault}{\color[rgb]{0,0,0}$b$}%
}}}}
\put(-4580,254){\makebox(0,0)[lb]{\smash{{\SetFigFont{8}{9.6}{\rmdefault}{
\mddefault}{\updefault}{\color[rgb]{0,0,0}$a$}%
}}}}
\end{picture}%
\end{center}
\caption{$\sin a \cdot \sin b = \const$.}
\end{figure}
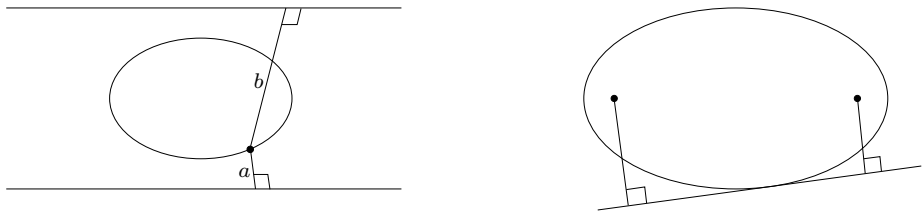
\begin{proof}
Take two non-parallel circular sections of the cone. By Lemma 
\ref{lem:TwoCircSec}, there is a sphere through these two sections. Therefore 
for every generatrix of the cone the product of lengths of the segments between 
the apex and the circular sections is constant:
\[
OX \cdot OY = \const,
\]
see Figure \ref{fig:SinSin}. On the other hand, we have
\[
OX = \frac{OA}{\sin a}, \quad OY = \frac{OB}{\sin b},
\]
where $OA$ and $OB$ are the distances from the apex to the chosen planes, and 
$a$, $b$ are the angles between the generatrix and those planes, that is the 
distances from the point corresponding to the generatrix to the focal lines of 
the conic. Since $OA$ and $OB$ are constant, the theorem follows.
\end{proof}
\begin{figure}[htb]
\begin{center}
\begin{picture}(0,0)%
\includegraphics{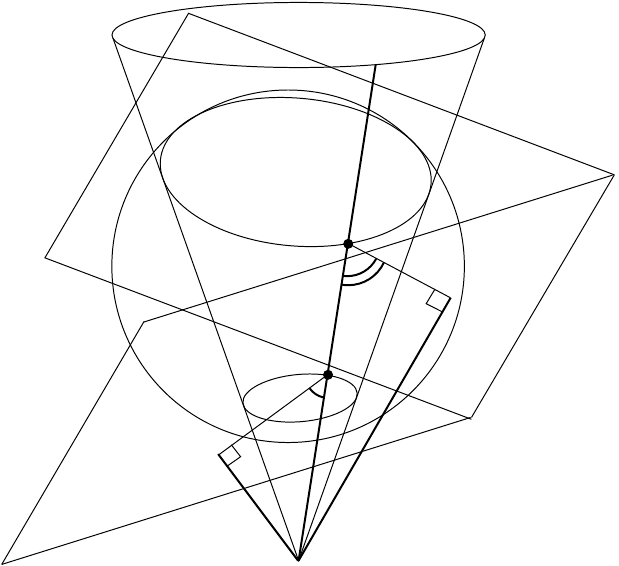}%
\end{picture}%
\setlength{\unitlength}{2486sp}%
\begingroup\makeatletter\ifx\SetFigFont\undefined%
\gdef\SetFigFont#1#2#3#4#5{%
  \reset@font\fontsize{#1}{#2pt}%
  \fontfamily{#3}\fontseries{#4}\fontshape{#5}%
  \selectfont}%
\fi\endgroup%
\begin{picture}(4694,4300)(-2274,-2352)
\put(164,190){\makebox(0,0)[lb]{\smash{{\SetFigFont{7}{8.4}{\rmdefault}{
\mddefault}{\updefault}{\color[rgb]{0,0,0}$Y$}%
}}}}
\put(-796,-1630){\makebox(0,0)[lb]{\smash{{\SetFigFont{7}{8.4}{\rmdefault}{
\mddefault}{\updefault}{\color[rgb]{0,0,0}$A$}%
}}}}
\put(1224,-343){\makebox(0,0)[lb]{\smash{{\SetFigFont{7}{8.4}{\rmdefault}{
\mddefault}{\updefault}{\color[rgb]{0,0,0}$B$}%
}}}}
\put(  
0,-1173){\makebox(0,0)[lb]{\smash{{\SetFigFont{7}{8.4}{\rmdefault}{\mddefault}{
\updefault}{\color[rgb]{0,0,0}$a$}%
}}}}
\put(412,-372){\makebox(0,0)[lb]{\smash{{\SetFigFont{7}{8.4}{\rmdefault}{
\mddefault}{\updefault}{\color[rgb]{0,0,0}$b$}%
}}}}
\put( 
19,-806){\makebox(0,0)[lb]{\smash{{\SetFigFont{7}{8.4}{\rmdefault}{\mddefault}{
\updefault}{\color[rgb]{0,0,0}$X$}%
}}}}
\end{picture}%
\end{center}
\caption{Proof of Theorem \ref{thm:SinSin}.}
\label{fig:SinSin}
\end{figure}
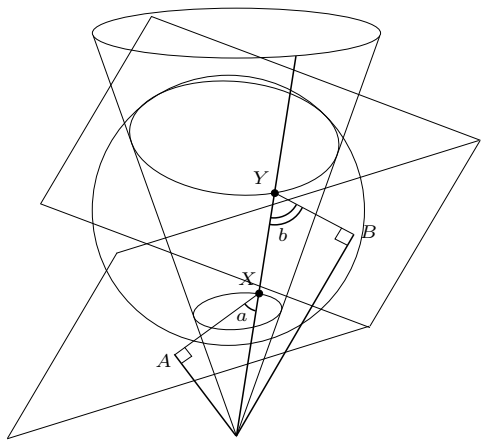

\subsection{The focus-directrix property}
Recall that for a point on a Euclidean conic its distance from a directrix is 
in a constant ratio to its distance from the corresponding focus. The following 
theorem provides a spherical analog.

\begin{thm}
\label{thm:FocDir}
\begin{enumerate}
\item
For a tangent to a spherical conic, the sine of its distance from a director 
point is in a constant ratio to the sine of the angle it makes with the 
corresponding focal line.
\item
For a point on a spherical conic, the sine of its distance to a directrix is in 
a constant ratio to the sine of its distance from the corresponding focus.
\end{enumerate}
\end{thm}
\begin{figure}[htb]
\begin{center}
\includegraphics[width=.8\textwidth]{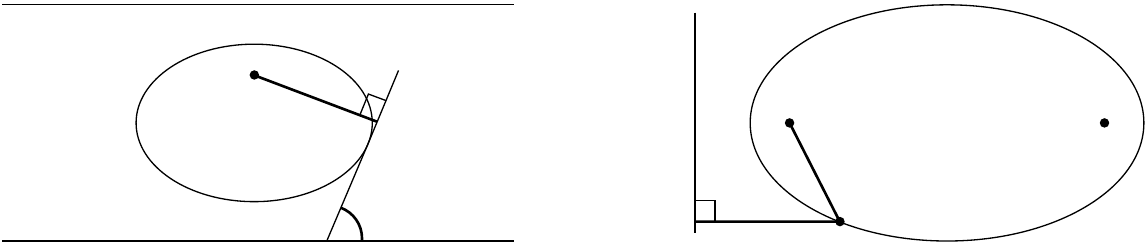}
\end{center}
\caption{The focus-directrix property and its dual.}
\end{figure}
\begin{proof}
Recall that a director point is the intersection point of the sphere with a 
line formed by the centers of circular sections of a quadratic cone.
Take a circular section of the cone. The distance from 
its center to a plane tangent to the cone is equal to $r \sin a$, where 
$r$ is the radius of the circle, and $a$ is the angle between 
the plane of the circle and the tangent plane, see Figure 
\ref{fig:FocDirProof}.
On the other hand, the 
same distance is equal to $\ell \sin b$, where $\ell$ is the length of the 
segment joining the center of the circle to the apex of the cone, and $b$ 
is the angle between this segment and the tangent plane. Thus we have
\[
\frac{\sin a}{\sin b} = \frac{\ell}{r} = \const.
\]
At the same time, $a$ is the angle made by the tangent and a focal line, and 
$b$ is the distance from the corresponding director to that tangent. The 
theorem is proved.
\end{proof}

\begin{figure}[htb]
\begin{center}
\begin{picture}(0,0)%
\includegraphics{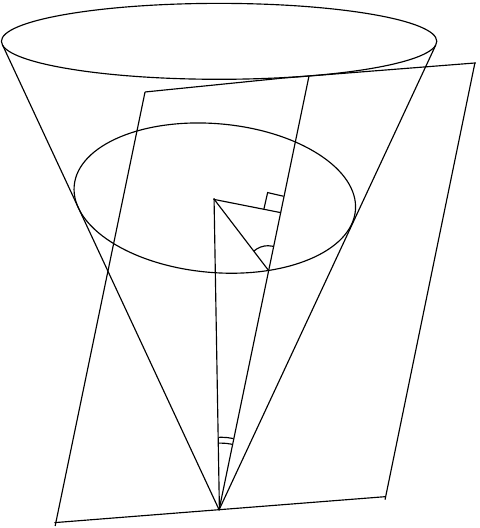}%
\end{picture}%
\setlength{\unitlength}{2901sp}%
\begingroup\makeatletter\ifx\SetFigFont\undefined%
\gdef\SetFigFont#1#2#3#4#5{%
  \reset@font\fontsize{#1}{#2pt}%
  \fontfamily{#3}\fontseries{#4}\fontshape{#5}%
  \selectfont}%
\fi\endgroup%
\begin{picture}(3119,3432)(-1429,-1484)
\put(252,390){\makebox(0,0)[lb]{\smash{{\SetFigFont{7}{8.4}{\rmdefault}{
\mddefault}{\updefault}{\color[rgb]{0,0,0}$a$}%
}}}}
\put(-124,-417){\makebox(0,0)[lb]{\smash{{\SetFigFont{7}{8.4}{\rmdefault}{
\mddefault}{\updefault}{\color[rgb]{0,0,0}$\ell$}%
}}}}
\put( 
16,-849){\makebox(0,0)[lb]{\smash{{\SetFigFont{7}{8.4}{\rmdefault}{\mddefault}{
\updefault}{\color[rgb]{0,0,0}$b$}%
}}}}
\put( 
73,332){\makebox(0,0)[lb]{\smash{{\SetFigFont{7}{8.4}{\rmdefault}{\mddefault}{
\updefault}{\color[rgb]{0,0,0}$r$}%
}}}}
\end{picture}%
\end{center}
\caption{Proof of Theorem \ref{thm:FocDir}.}
\label{fig:FocDirProof}
\end{figure}
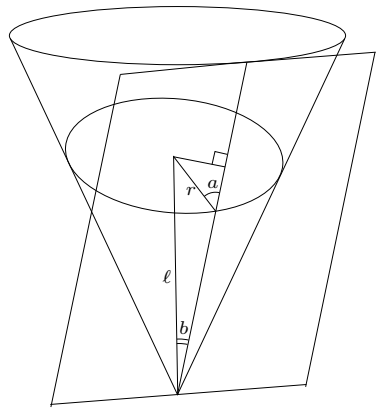

\subsection{Special spherical conics}
All spherical conics look essentially the same. However, in certain respects 
some 
of them are special.

\begin{thm}
\label{thm:SpecSphCon}
\begin{enumerate}
\item
The locus of the points from which a spherical arc is seen under 
the right angle is a spherical conic. The endpoints of the arc belong to 
the conic and are the poles of its cyclic lines.
\item
An arc of length $\frac{\pi}2$ with endpoints moving along two given great 
circles is tangent to a spherical conic whose foci are the poles 
of these great circles.
\end{enumerate}
\end{thm}
\begin{figure}[htb]
\begin{center}
\begin{picture}(0,0)%
\includegraphics{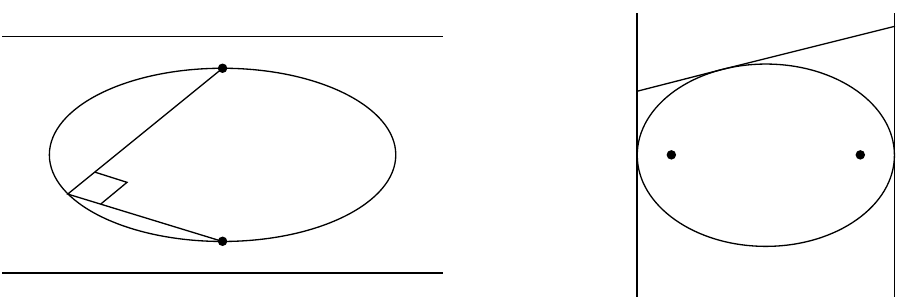}%
\end{picture}%
\setlength{\unitlength}{3315sp}%
\begingroup\makeatletter\ifx\SetFigFont\undefined%
\gdef\SetFigFont#1#2#3#4#5{%
  \reset@font\fontsize{#1}{#2pt}%
  \fontfamily{#3}\fontseries{#4}\fontshape{#5}%
  \selectfont}%
\fi\endgroup%
\begin{picture}(5127,1701)(-1271,17)
\put(3331,1559){\makebox(0,0)[lb]{\smash{{\SetFigFont{6}{7.2}{\rmdefault}{
\mddefault}{\updefault}{\color[rgb]{0,0,0}$\frac{\pi}2$}%
}}}}
\end{picture}%
\end{center}
\caption{Special spherical conics.}
\end{figure}
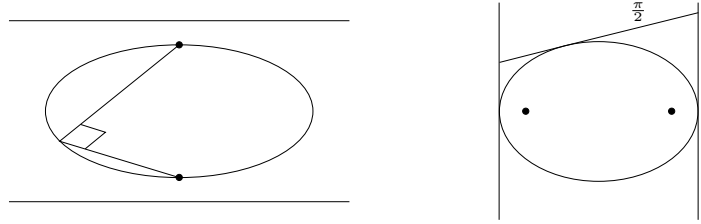
\begin{proof}
The first statement translates as follows. Choose two lines $p_1$ and $p_2$ 
through the origin 
and let two planes rotate around these lines while being 
perpendicular to each other. Then their intersection line 
describes a quadratic cone 
whose cyclic planes are orthogonal to $p_1$ and $p_2$.

Draw a plane $\pi$ perpendicular to the line $p_1$. Planes through $p_1$ and 
$p_2$ are perpendicular if and only if their lines of intersection with $\pi$ 
are, see Figure \ref{fig:OrthOptProof}. Hence the intersection line of these 
planes describes a cone over a circle with the segment $[p_1 \cap \pi, p_2 \cap 
\pi]$ as a diameter. This is a quadratic cone, and the plane $\pi$ is parallel 
to one of its cyclic planes. The theorem is proved.
% 
% Let $a$ and $b$ be the chosen lines. Take on $a$ a point $A$ different from 
% the origin and draw a plane $\pi$ through $A$ perpendicular to the line $a$. 
% Let $B$ be the intersection point of this plane with the line $b$. In the 
% plane $\pi$, draw a circle with a diameter $AB$. We claim that if a plane 
% $\rho$ through $a$ and $\sigma$ through $b$ intersect orthogonally, then their 
% intersection line goes through the circle. Indeed, every plane $\rho$ through 
% the line $a$ is orthogonal to the plane $\pi$. Therefore if $\rho$ and $\sigma$ 
% intersect orthogonally, then so do the intersection lines $\rho \cap \pi$ and 
% $\sigma \cap \pi$. As a consequence, these lines meet at a point of our circle, 
% and the theorem is proved.
\end{proof}

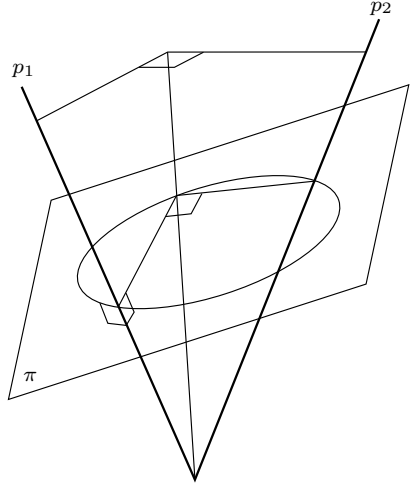
\begin{figure}[htb]
\begin{center}
\begin{picture}(0,0)%
\includegraphics{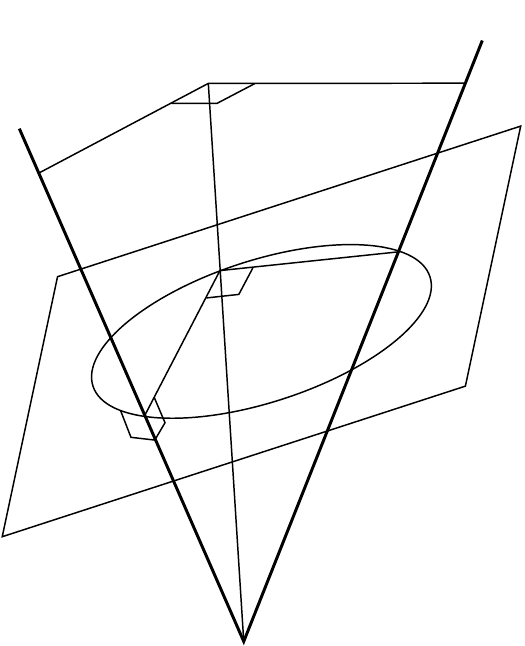}%
\end{picture}%
\setlength{\unitlength}{3729sp}%
\begingroup\makeatletter\ifx\SetFigFont\undefined%
\gdef\SetFigFont#1#2#3#4#5{%
  \reset@font\fontsize{#1}{#2pt}%
  \fontfamily{#3}\fontseries{#4}\fontshape{#5}%
  \selectfont}%
\fi\endgroup%
\begin{picture}(2658,3280)(-1238,-713)
\put(-1128,-38){\makebox(0,0)[lb]{\smash{{\SetFigFont{8}{9.6}{\rmdefault}{
\mddefault}{\updefault}{\color[rgb]{0,0,0}$\pi$}%
}}}}
\put(-1199,2008){\makebox(0,0)[lb]{\smash{{\SetFigFont{8}{9.6}{\rmdefault}{
\mddefault}{\updefault}{\color[rgb]{0,0,0}$p_1$}%
}}}}
\put(1155,2420){\makebox(0,0)[lb]{\smash{{\SetFigFont{8}{9.6}{\rmdefault}{
\mddefault}{\updefault}{\color[rgb]{0,0,0}$p_2$}%
}}}}
\end{picture}%
\end{center}
\caption{Proof of Theorem \ref{thm:FocDir}.}
\label{fig:OrthOptProof}
\end{figure}

\begin{lem}
Spherical conic from the first part of Theorem 
\ref{thm:SpecSphCon} are given by the equations
\[
\frac{x^2}{a^2} + \frac{y^2}{b^2} - \frac{z^2}{c^2} = 0, \quad a > b
\]
with $\frac1{a^2} = \frac1{b^2} + \frac1{c^2}$. Spherical conics from the 
second part of the same theorem satisfy $a^2 = b^2 + c^2$.
\end{lem}
\begin{proof}
By polarity, the statements of the lemma are equivalent. Any one of them can be 
proved with the help of the formulas from Theorem \ref{thm:CyclCoord}.
\end{proof}

Another special class of spherical conics is formed by those 
for which the distance to a focus is equal to the distance 
to the corresponding directrix. In this respect, they are similar to the
Euclidean parabolas.

\begin{thm}
The locus of points on the sphere at equal distances from a point and a great 
circle is a spherical conic with each component of diameter $\frac{\pi}2$.
The equation \eqref{eqn:SphConCoord} of this conic satisfies $a = c$.
\end{thm}
\begin{proof}
That this curve is a component of a spherical conic, follows from (the inverse 
of) the second part of Theorem \ref{thm:FocDir} (the focus-directrix property). 
The two most distant points on the curve are the midpoints of the 
perpendiculars from the point to the great circle.

For a general conic \eqref{eqn:SphConCoord}, the diameter of a component is 
equal to the angle between the rays spanned by the vectors $(a,0,c)$ and 
$(-a,0,c)$. This angle is 
equal to $\frac{\pi}2$ if and only if $a=c$.
\end{proof}

Alternatively, one can derive the last theorem from (the inverse of) the 
second part of Theorem \ref{thm:ConstSum}. For a point $F$ and a great 
circle $d$, the condition $\dist(x,F) = \dist(x,d)$ is equivalent to 
$\dist(x,F) + \dist(x,d^\circ) = \frac{\pi}2$, where $d^\circ$ is the pole 
of $d$ lying in the same hemisphere with respect to~$d$ as $F$. Thus the 
``spherical parabolas'' are also characterized by their foci being the poles of 
their directrices (the directrix corresponding to a focus must be the pole of 
the other focus).

For other special spherical conics see \cite{GSO16}.

\section{Hyperbolic conics}
\subsection{The hyperbolic-de Sitter plane}
Let $\Omega$ be a quadratic form on $\R^3$ of signature $(-, +, +)$.
The hyperbolic plane is a component of the hyperboloid of two sheets 
$\Omega(x,x) 
= -1$, equipped with a Riemannian metric induced by the form $\Omega$. In the 
Beltrami--Cayley--Klein model, the hyperboloid is projected from the origin to 
an 
affine
plane and becomes the interior region of a conic, the 
\emph{absolute conic}. Points on the absolute are called \emph{ideal} or 
\emph{absolute points}. Geodesics in the Beltrami--Cayley--Klein model are 
straight line 
segments with ideal endpoints. The geodesic distance between two points is half 
the logarithm of their cross-ratio with the collinear ideal points.

It is convenient to view the plane of the Beltrami--Cayley--Klein model as a 
projective 
plane; it is then
nothing else but the projectivization of $\R^3$. On pictures, we show an affine 
chart of this plane, with the absolute in the form of a circle.

The exterior of the absolute, homeomorphic to an open M\"obius band, 
is called the \emph{de Sitter plane}.
The polarity with respect to the absolute conic sends hyperbolic points to de 
Sitter lines (projective lines disjoint from the absolute), ideal points to 
lines tangent to the absolute, and de Sitter points to hyperbolic lines.

For more details on the de Sitter geometry, see \cite{FS}.

\subsection{Classification of hyperbolic conics}
Algebraically, a hyperbolic conic is a pair of quadratic forms $(\Omega, Q)$ in 
$\R^3$, where the absolute form $\Omega$ has signature $(-, +, +)$. We assume 
$Q$ to be indefinite (thus with non-empty isotropic cone) and usually 
non-degenerate, so that without loss of generality it has signature $(-, +, +)$ 
as well.
% Geometrically, a 
% hyperbolic conic is the intersection of a quadratic cone with the hyperboloid 
% of two sheets $\Omega(x,x) = -1$,
% where the hyperboloid is equipped with a Riemannian metric that is the 
% restriction of $\Omega$.

Geometrically, in the Beltrami--Cayley--Klein model, 
a hyperbolic conic is the part of an affine conic inside the absolute. 
The part outside of the absolute may be called a de Sitter conic. However, it 
is convenient to consider both parts at the same time. Under the absolute 
polarity, the hyperbolic (respectively, de Sitter) points of a conic correspond 
to the tangents at the de Sitter (respectively, hyperbolic) points of the polar 
conic.

For a pair of indefinite quadratic forms the principal axes theorem in 
general does not hold. A geometric manifestation of this is the variety of
different relative positions of two real conics, and hence the variety of
different types of hyperbolic conics.
Following Klein \cite{Klein}, we classify hyperbolic conics according to the 
multiplicity and the reality of their ideal points.

\begin{dfn}
An intersection point of a hyperbolic conic with the absolute 
$\Omega = 0$ is called an \emph{absolute point} of the conic. A common tangent 
to the conic 
and the absolute is called an \emph{absolute tangent} to the conic.
\end{dfn}
%The polar of an absolute point is an absolute tangent of the polar conic.

There are four absolute points and four absolute tangents, counted with 
multiplicity and including imaginary elements. Imaginary points or lines come 
in conjugate pairs.

Non-degenerate hyperbolic conics are subdivided into ellipses, hyperbolas, 
parabolas, and cycles. Ellipses and 
hyperbolas (see Figure \ref{fig:EH}) have four distinct absolute 
points. Parabolas (see Figure \ref{fig:HypPar}) have at least one simple and at 
least one multiple absolute point. Finally, 
cycles (see Figure \ref{fig:HypCirc}) have either two double ore one quadruple 
absolute point.

% \begin{table}[ht]
% \begin{tabular}{l|c|c}
% & Absolute points & Absolute tangents\\
% \hline
% Ellipse & $4$ imaginary & $4$ imaginary\\
% De Sitter ellipse & $4$ imaginary & $4$ imaginary\\
% Convex hyperbola & $4$ real & $4$ imaginary\\
% De Sitter hyperbola & $4$ imaginary & $4$ real\\ 
% Concave hyperbola & $4$ real & $4$ real\\
% Semihyperbola & $2$ re $+$ $2$ im & $2$ re $+$ $2$ im
% \end{tabular}
% \caption{Hyperbolic ellipses and hyperbolas.}
% \label{tbl:EH}
% \end{table}

\begin{figure}[ht]
\begin{center}
\begin{picture}(0,0)%
\includegraphics{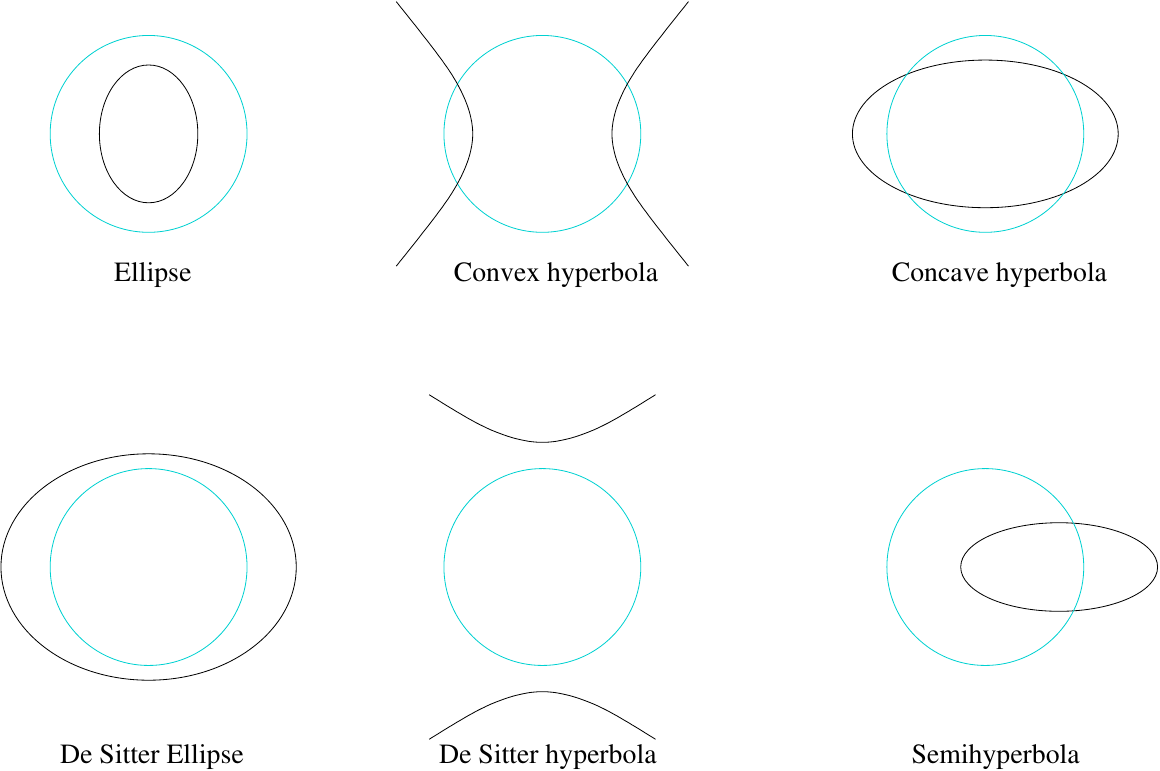}%
\end{picture}%
\setlength{\unitlength}{2072sp}%
\begingroup\makeatletter\ifx\SetFigFont\undefined%
\gdef\SetFigFont#1#2#3#4#5{%
  \reset@font\fontsize{#1}{#2pt}%
  \fontfamily{#3}\fontseries{#4}\fontshape{#5}%
  \selectfont}%
\fi\endgroup%
\begin{picture}(10591,7045)(-1357,-1025)
\end{picture}%
\end{center}
\caption{Hyperbolic ellipses and hyperbolas.}
\label{fig:EH}
\end{figure}
% 
% The polar of an ellipse is a de Sitter ellipse; the polar of a convex hyperbola 
% is a de Sitter hyperbola. Concave hyperbolas and semihyperbolas are self-dual.

% \begin{table}[ht]
% \begin{tabular}{l|c|c}
% & Absolute points & Absolute tangents\\
% \hline
% Concave hyperbolic parabolas & $2 \cdot 1 + 2$ real & $2 \cdot 1 + 2$ real\\
% Convex hyperbolic parabola & $2 \cdot 1 + 2$ real & $2 \cdot 1$ re + $2$ im\\
% Elliptic parabola & $2 \cdot 1$ re + $2$ im & $2 \cdot 1$ re + $2$ im\\
% Small de Sitter parabola & $2 \cdot 1$ re + $2$ im & $2 \cdot 1 + 2$ real\\
% Big de Sitter parabola & $2 \cdot 1$ re + $2$ im & $2 \cdot 1$ re + $2$ im\\
% Osculating parabola & $3 \cdot 1 + 1$ real & $3 \cdot 1 + 1$ real
% \end{tabular}
% \caption{Hyperbolic parabolas.}
% \label{tbl:HypPar}
% \end{table}

\begin{figure}[htb]
\begin{center}
\begin{picture}(0,0)%
\includegraphics{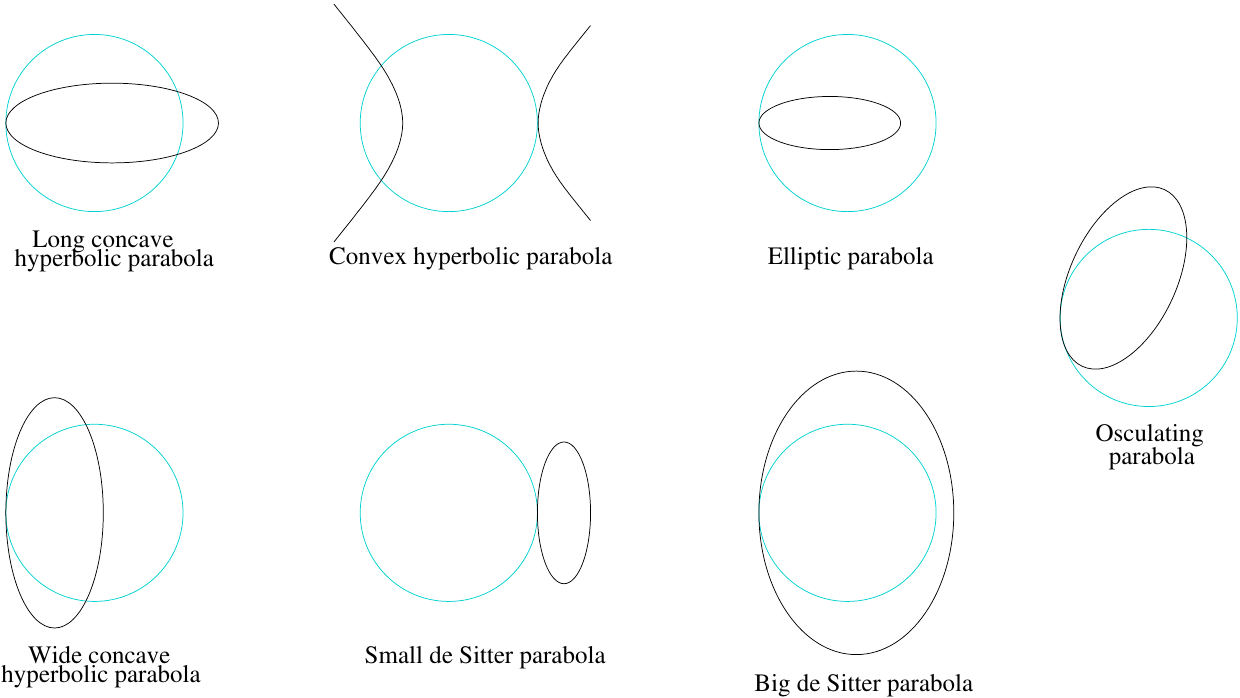}%
\end{picture}%
\setlength{\unitlength}{1865sp}%
\begingroup\makeatletter\ifx\SetFigFont\undefined%
\gdef\SetFigFont#1#2#3#4#5{%
  \reset@font\fontsize{#1}{#2pt}%
  \fontfamily{#3}\fontseries{#4}\fontshape{#5}%
  \selectfont}%
\fi\endgroup%
\begin{picture}(12578,7060)(-959,-1040)
\end{picture}%
\end{center}
\caption{Hyperbolic parabolas.}
\label{fig:HypPar}
\end{figure}

% \begin{table}[ht]
% \begin{tabular}{l|c|c}
% & Absolute points & Absolute tangents\\
% \hline
% Circles & $2 \cdot 1 + 2 \cdot 1$ imaginary & $2 \cdot 1 + 2 \cdot 1$ 
% imaginary\\
% Horocycles & $4 \cdot 1$ real & $4 \cdot 1$ real\\
% Hypercycles & $2 \cdot 1 + 2 \cdot 1$ real & $2 \cdot 1 + 2 \cdot 1$ real
% \end{tabular}
% \caption{Hyperbolic circles.}
% \label{tbl:HypCirc}
% \end{table}

\begin{figure}[htb]
\begin{center}
\begin{picture}(0,0)%
\includegraphics{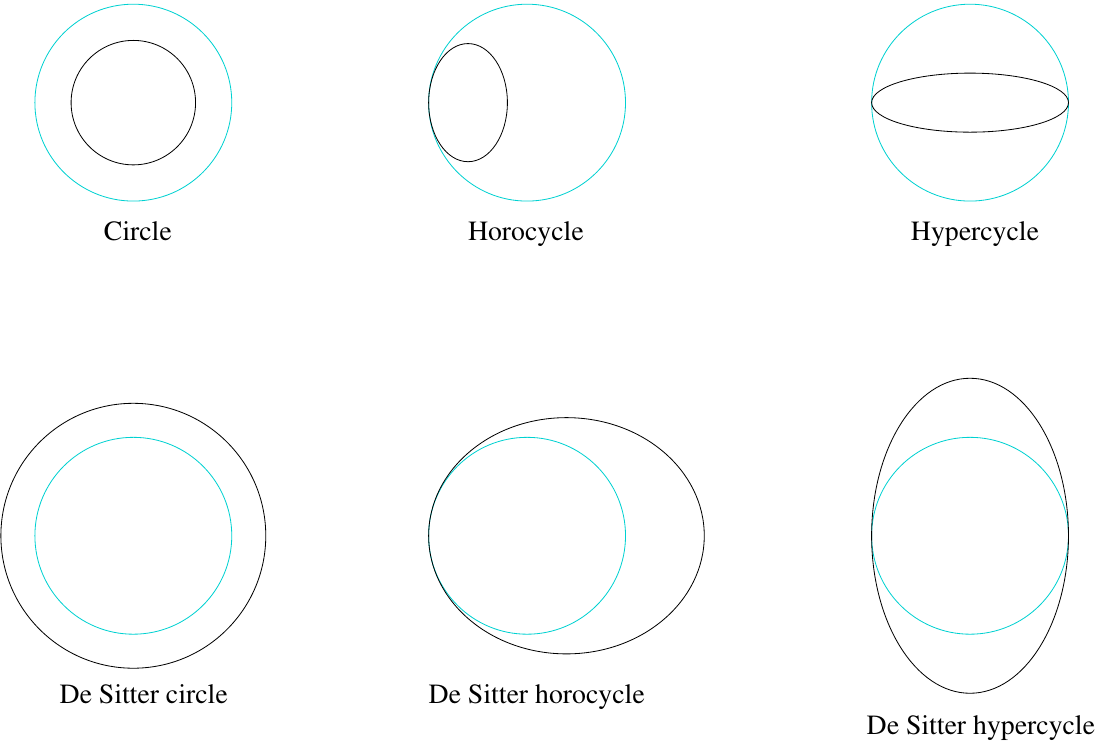}%
\end{picture}%
\setlength{\unitlength}{2072sp}%
\begingroup\makeatletter\ifx\SetFigFont\undefined%
\gdef\SetFigFont#1#2#3#4#5{%
  \reset@font\fontsize{#1}{#2pt}%
  \fontfamily{#3}\fontseries{#4}\fontshape{#5}%
  \selectfont}%
\fi\endgroup%
\begin{picture}(10051,6859)(-1218,-1153)
\end{picture}%
\end{center}
\caption{Hyperbolic circles.}
\label{fig:HypCirc}
\end{figure}

Conics occupying the same column on Figures \ref{fig:EH}--\ref{fig:HypCirc} are 
dual to each other, with the exception of concave hyperbolas and 
semihyperbolas, which are self-dual. Another self-dual class of conics 
are osculating parabolas.

\subsection{Foci and focal lines of a hyperbolic conic}
\label{sec:FocHyp}
\begin{dfn}
A \emph{focal line} of a conic is a line through two of its absolute points.
A \emph{focus} of a conic is an intersection point of two of its absolute 
tangents.

If a conic is tangent to the absolute, then the point of tangency is considered 
as a focus, and the tangent at this point as a focal line of the conic.
\end{dfn}

A focal line and a focus are dual notions: the pole of a focal line 
is a focus of the polar conic.

Two complex conjugate lines intersect in a real point, and two 
complex conjugate points lie on a real line. Therefore every conic has at least 
one pair of (possibly coincident) real foci and at least one pair of 
(possibly coincident) real focal lines.

\begin{lem}
A real focus of a hyperbolic conic is a hyperbolic, ideal, or a de Sitter point 
at the same time as it lies inside, on the boundary, or outside of the oval 
bounded by the conic.
\end{lem}
\begin{proof}
This is quite obvious, and becomes even more obvious in the dual formulation: 
a real focal line is hyperbolic, ideal, or de Sitter at the same time as it 
intersects, is tangent, or is disjoint from the conic.
\end{proof}

% 
% 
% If the absolute tangents through a real
% focus are imaginary, then the focus is a de Sitter point and lies outside the 
% convex oval bounded by the conic. If the absolute tangents through a real focus 
% are real, then the focus is a hyperbolic point and lies inside the oval bounded 
% by the conic.

Ellipses and hyperbolas have three pairs of focal lines and three pairs of foci, 
every pair corresponding to different matchings of four absolute points or four 
absolute lines.
All foci of concave hyperbolas and de Sitter hyperbolas are real and lie in 
the de Sitter plane. The other ellipses and hyperbolas have only one pair 
of real foci.

Non-osculating parabolas have two pairs of foci. One pair is real and 
contains an ideal point. The other pair is 
double, that is it splits in two pairs if the parabola is perturbed so that to 
become an ellipse or a hyperbola. The double pair of foci is real for the 
concave parabolas and for the de Sitter parabola. The osculating parabola has a 
triple pair of foci, consisting of the point of tangency and of a de Sitter 
point on the corresponding absolute tangent.

All cycles have a pair of coincident foci: the center of a circle, the ideal 
point of a horocycle, and the pole of the center line of a hypercycle. 
Hypercycles have in addition a double pair of foci: the endpoints of the center 
line.

\begin{lem}
\label{lem:ConAbsOnFoc}
A conic and the absolute induce the same Cayley--Klein metric on each focal 
line 
and on the pencil of lines through each focus.
\end{lem}
\begin{proof}
A Cayley--Klein distance between two points is determined by their cross-ratio 
with the two (possibly imaginary) collinear points on the conic. A focal line 
intersects the conic and the absolute in the same points, therefore the two 
metrics on this line coincide.

Similarly, a Cayley--Klein angle between two lines is determined by the 
cross-ratio of these lines with the two tangents from the same pencil. In the 
line pencil through a focus the same two lines are tangent to the conic and to 
the absolute.
\end{proof}

Compare this with Theorems \ref{thm:SphFocLines} and \ref{thm:SphFoci} in the 
spherical case.

% In Definition \ref{dfn:CycPlane} and Theorem \ref{thm:CocycLine} we postulated 
% or proved in a synthetic manner the same properties for the focal lines and 
% foci of a spherical conic. The focal lines and the foci of a 
% spherical conic can be defined in the above analytic manner: they are the lines 
% through intersection points of the conic and the absolute and the 
% intersections of their common tangents, respectively.

\begin{figure}[htb]
\begin{center}
\includegraphics[width=.8\textwidth]{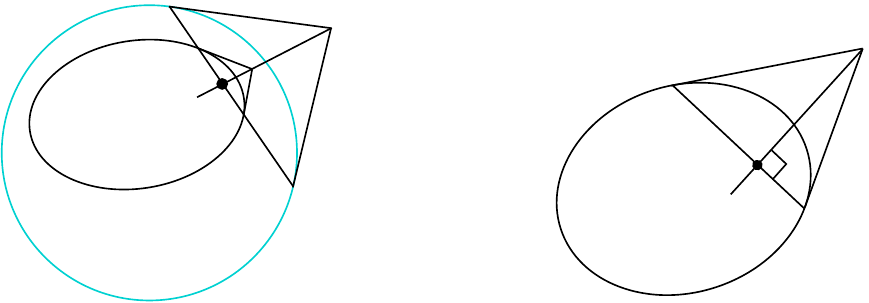}
\end{center}
\caption{Characteristic properties of foci of hyperbolic and Euclidean conics.}
\label{fig:FocPerp}
\end{figure}

Lemma \ref{lem:ConAbsOnFoc} implies the following property of a focus of a 
hyperbolic conic: for any line through the 
focus, its absolute pole and the pole with respect to the 
conic are collinear with the focus, see Figure \ref{fig:FocPerp}, left. The 
Euclidean analog of this is shown on Figure \ref{fig:FocPerp}, right.

\subsection{Axes and centers}
As we already noted, foci and focal lines come in pairs.
\begin{dfn}
The intersection of a pair of focal lines is called a \emph{center} 
of the conic. The line through a pair of foci is called an \emph{axis} of the 
conic.
\end{dfn}

\begin{figure}[htb]
\begin{center}
\includegraphics[width=.9\textwidth]{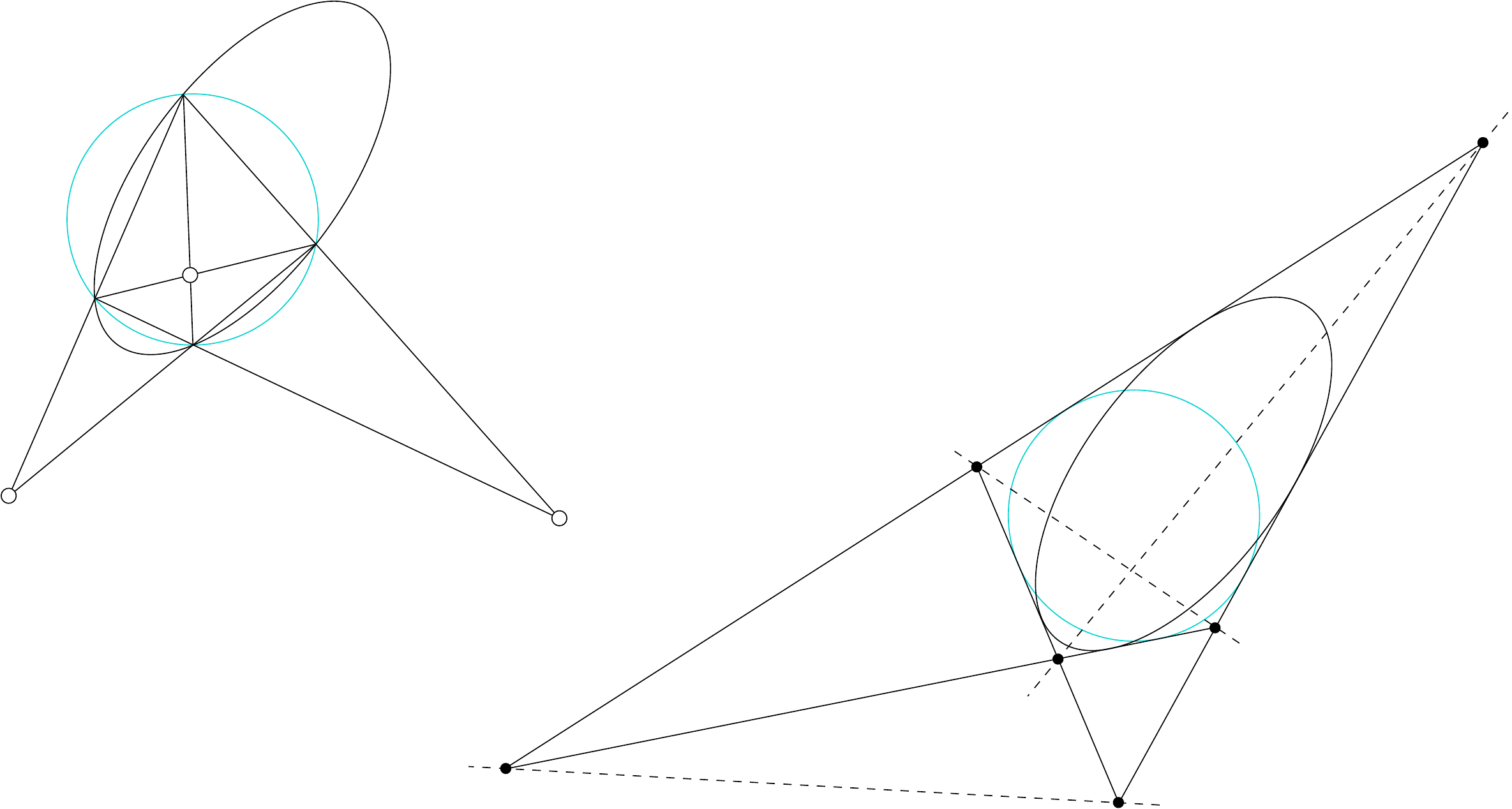}
\end{center}
\caption{Centers and axes of a concave hyperbola.}
\label{fig:FocCentAx}
\end{figure}

All hyperbolic ellipses and hyperbolas, except the semi-hyperbola, have three 
distinct centers and three distinct axes. The case of a concave hyperbola is 
illustrated on Figure \ref{fig:FocCentAx}.

\begin{lem}
\begin{enumerate}
\item
The centers of a conic are pairwise conjugate with respect to the conic as 
well as with respect to the absolute.
\item
The axes of a conic are pairwise conjugate with respect to the conic as well as 
with respect to the absolute.
\item
A line through two centers is an axis; an intersection point of two axes is a 
center.
\end{enumerate}
\end{lem}
\begin{proof}
It is a classical fact that the three diagonal points of a quadrangle inscribed 
in a conic are pairwise conjugate with respect to the conic, see e.~g. 
\cite[\S 14.5.2]{BerII}. The quadrangle of 
the absolute points is inscribed both in the conic and in the absolute. This 
implies the first part of the lemma.

The second part is dual to the first one.

Three pairwise conjugate lines or three pairwise conjugate points form a 
self-polar triangle. Since two conics in general position have a unique common 
self-polar triangle, it follows that the centers and the axes span the same 
triangle.
\end{proof}

The third part of the lemma has the following reformulation: for two conics 
meeting in four points, the diagonals of the common inscribed quadrilateral and 
those of the common circumscribed one meet at the same point. An alternative 
proof of this is to apply a projective transformation that sends the four 
intersection points to the vertices of a square.

% \begin{figure}[htb]
% \begin{center}
% \includegraphics[width=.7\textwidth]{UniqueCent.pdf}
% \end{center}
% \caption{Two axes meet at a center.}
% \label{fig:UniqueCent}
% \end{figure}

\subsection{Directors and directrices}
Directors and directrices of a hyperbolic conic are defined in the same way as 
for a spherical conic, see Definitions \ref{dfn:SpherDir} and 
\ref{dfn:SpherDira}.

\begin{dfn}
The pole of a focal line with respect to the conic is called a \emph{director 
point} of the conic. The polar of a focus with respect to the conic is called 
a \emph{directrix} of the conic.
\end{dfn}

\begin{figure}[htb]
\begin{center}
\includegraphics[width=.7\textwidth]{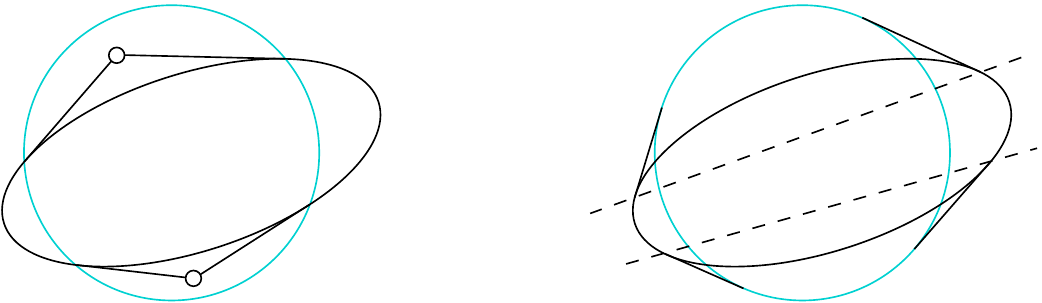}
\end{center}
\caption{A pair of director points and a pair of directrices.}
\label{fig:DirDira}
\end{figure}

\begin{lem}
\label{lem:FocDir}
Director points of a conic lie on its axes. Directrices of a conic pass through 
its centers.
\end{lem}
\begin{proof}
Draw a line through a pair of director points. By the basic properties of 
polarity, the pole of this line with respect to the conic is the intersection 
point of the corresponding 
focal lines, see Figure \ref{fig:DirAxis}, that is a center of the conic. It 
follows that the line through a pair of director points is an axis.

The second part of the lemma is dual to the first.
\end{proof}

\begin{figure}[htb]
\begin{center}
\includegraphics[width=.5\textwidth]{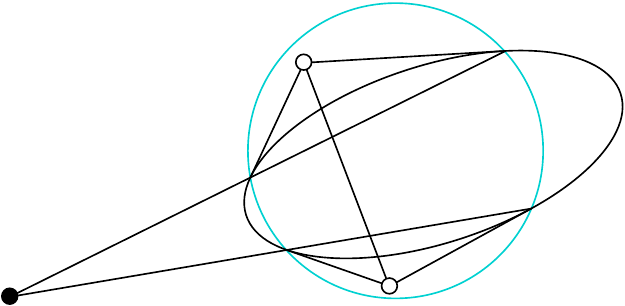}
\end{center}
\caption{A pair of director points span an axis.}
\label{fig:DirAxis}
\end{figure}

\begin{cor}
Through every center of a hyperbolic conic pass two focal lines, two axes, 
and two directrices.
\end{cor}

% 
% Adding two directrices to Figure \ref{fig:UniqueCent}, we get two more lines 
% through the same point, see Figure \ref{fig:UniqueCent1}. A quick way to prove 
% the concurrence of all those lines is to apply a projective transformation that 
% sends the four intersection points of the conic and the absolute to the 
% vertices of a rectangle. The conic and the absolute become concentric with the 
% rectangle, and all lines will pass through the center.
% 
% \begin{figure}[htb]
% \begin{center}
% \includegraphics[width=.7\textwidth]{UniqueCent1.pdf}
% \end{center}
% \caption{Two axes and a pair of directrices meet at a center.}
% \label{fig:UniqueCent1}
% \end{figure}

\subsection{Examples}
Let us locate foci and directrices for some types 
of hyperbolic conics. We assume that $\Omega(v,v) = x^2 + y^2 - z^2$, so that 
in the Beltrami--Cayley--Klein model in the plane $z=1$ the absolute is the 
unit circle centered at the origin.

\begin{exl} By a linear transformation that preserves the absolute, a 
hyperbolic ellipse can be brought to a canonical form
\[
\frac{x^2}{a^2} + \frac{y^2}{b^2} = 1, \quad 1 > a > b
\]
The real foci, real focal lines, and real directrices are
\[
\left( \pm\sqrt{\frac{a^2 - b^2}{1-b^2}}, 0 \right), \quad x = \pm a 
\sqrt{\frac{1-b^2}{a^2-b^2}}, \quad x = \pm a^2 \sqrt{\frac{1-b^2}{a^2-b^2}}.
\]
The directrices are tangent to the absolute if and only if $b = 
\frac{a}{\sqrt{1+a^2}}$. For smaller $b$ the directrices are hyperbolic lines, 
for larger $b$ they are de Sitter lines, see Figure \ref{fig:EllDir}.
\end{exl}

\begin{figure}[htb]
\begin{center}
\includegraphics[height=.14\textheight]{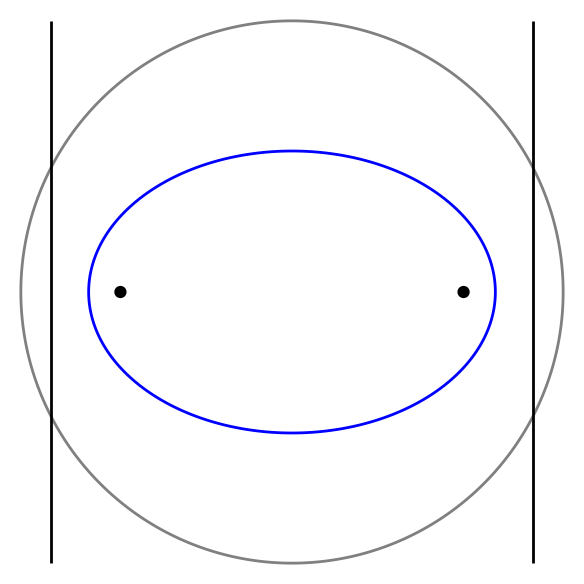} \hspace{.5cm}
\includegraphics[height=.14\textheight]{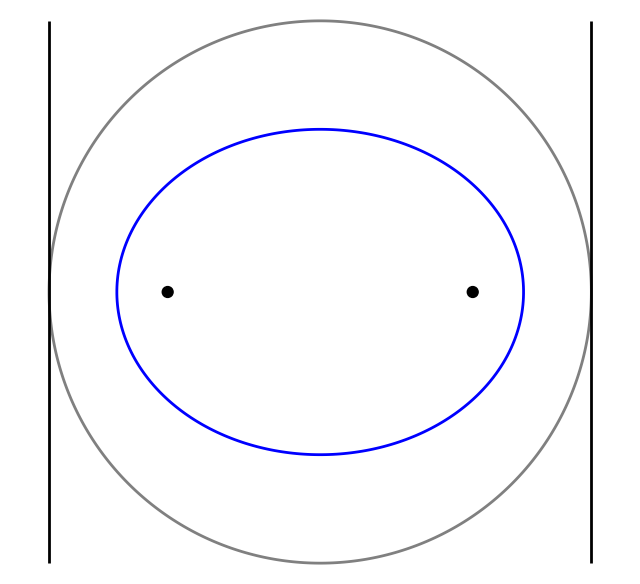} \hspace{.5cm}
\includegraphics[height=.14\textheight]{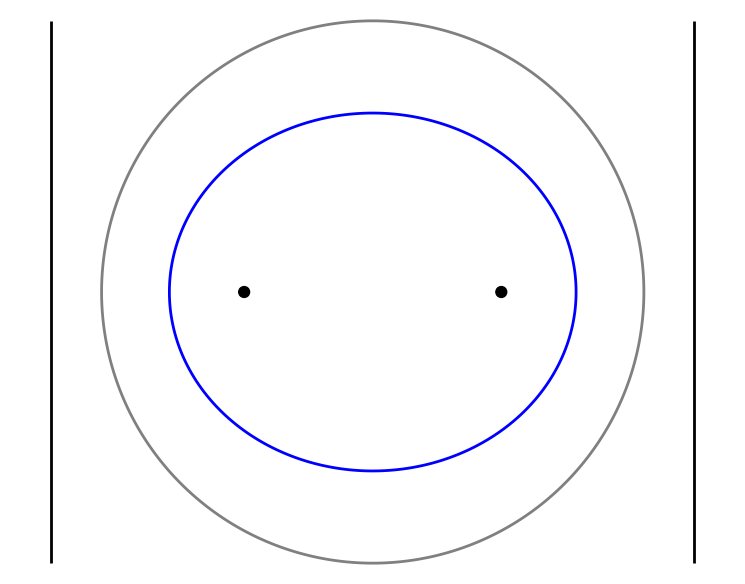}
\end{center}
\caption{Foci and directrices of hyperbolic ellipses.}
\label{fig:EllDir}
\end{figure}

\begin{exl}
\label{exl:SH}
A semi-hyperbola can be brought to a canonical form
\[
ax^2 - 2x + \frac{y^2}{b^2} = 0, \quad |a| < 2.
\]
The only pair of real foci is
\[
F_1 = \left( \frac{b^2}{1-\sqrt{b^4-ab^2+1}}, 0 \right), \quad F_2 = 
\left( \frac{b^2}{1+\sqrt{b^4-ab^2+1}}, 0 \right).
\]
The focus $F_1$ is a de Sitter point, the focus $F_2$ is a hyperbolic point. 
The corresponding directrices $d_1$ and $d_2$ are given by equations
\[
x = \frac{b^2}{ab^2 - 1 + \sqrt{b^4-ab^2+1}}, \quad
x = \frac{b^2}{ab^2 - 1 - \sqrt{b^4-ab^2+1}}, 
\]
respectively. For $a=\frac{1}{b^2}$ (that is, when the semi-hyperbola is 
represented by a 
circular arc) both directrices are tangent to the absolute, for $a<\frac1{b^2}$ 
$d_1$ is de Sitter $d_2$ is hyperbolic, for $a>b$ $d_1$ is hyperbolic and $d_2$ 
is de Sitter, see Figure \ref{fig:SHDir}.

If $a=0$ (the semi-hyperbola is represented by a parabola), then $d_1 = 
F_2^\circ$, and $d_2 = F_2^\circ$. In this case the semi-hyperbola is the locus 
of points equidistant from the point $F_1$ and the line $d_1$, see Example 
\ref{exl:SpecSH}.
\end{exl}

\begin{figure}[htb]
\begin{center}
\includegraphics[height=.12\textheight]{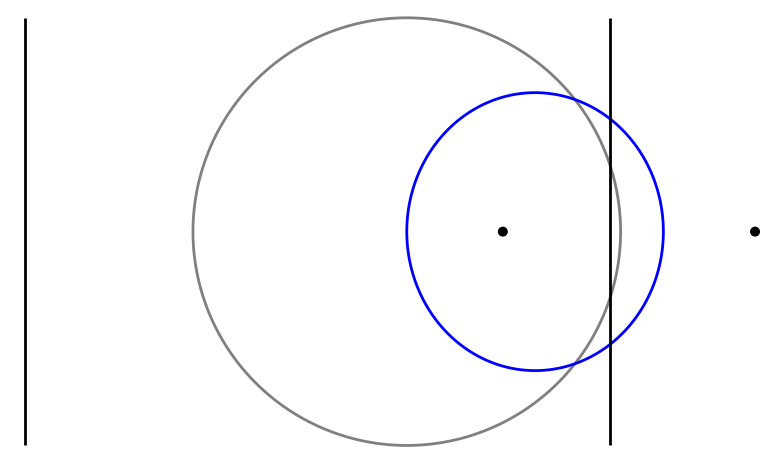} \hspace{.5cm}
\includegraphics[height=.12\textheight]{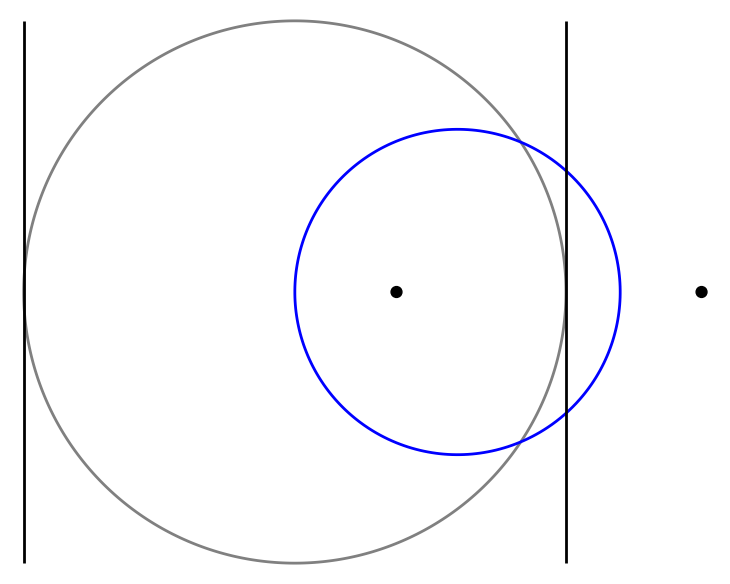} \hspace{.5cm}
\includegraphics[height=.12\textheight]{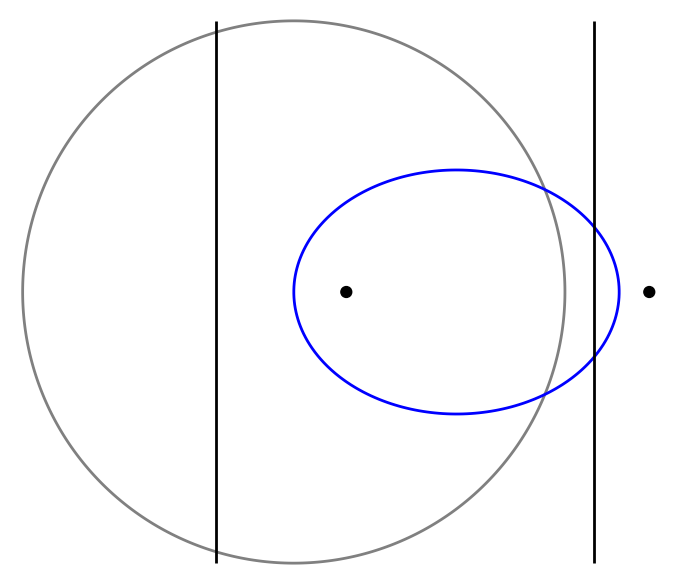}
\end{center}
\caption{Foci and directrices of semihyperbolas.}
\label{fig:SHDir}
\end{figure}

% Another canonical form for semi-hyperbolas:
% (Possibly there is a mistake here...)
% \[
% \frac{(x-c)^2}{a^2} + y^2 = 1, \quad -1 < c-a < 1 < c+a.
% \]
% Its real focal lines are given by the equations
% \[
% x = \frac{c}{1-a}, \quad x = \frac{c}{1+a}.
% \]
% Its pair of real foci (can be found by computing the focal lines of the dual 
% semi-hyperbola) is
% \[
% F_1 = \left(\frac{1-a^2+c^2}{2c}, 0\right), \quad F_2 = (1:0:0)
% \]
% (that is, $F_2$ belongs to the line at infinity). The corresponding directrices 
% are given by the equations
% \[
% d_1 = \left\{x = \frac{c(c^2 - a^2 - 1)}{c^2 + a^2 - 1}\right\}, \quad d_2 =
% \left\{x = c\right\}.
% \]
% The directrix $d_1$ is always hyperbolic, the directrix $d_2$ is tangent to the 
% absolute for $c=1$, hyperbolic for $c<1$, and de Sitter for $c>1$.

\begin{exl}
\label{exl:EllPar}
An elliptic parabola can be brought to a canonical form
\[
\frac{(x-(1-a))^2}{a^2} + \frac{y^2}{b^2} = 1, \quad b^2 < a < 1.
\]
(The condition $b^2 < a$ ensures that the curve stays inside the unit circle.) 
The only pair of real foci:
\[
F_1 = \left(1- \frac{2(a-b^2)}{1-b^2}, 0 \right), \quad F_2 = (1,0).
\]
The corresponding directrices are given by the equations
\[
d_1 = \left\{ x = 1 + \frac{2a(a-b^2)}{-a + 2b^2 - ab^2} \right\}, \quad d_2 = 
\{x=1\}.
\]
The directrix $d_1$ is tangent to the absolute for $a = 2b^2$, hyperbolic for 
$a>2b^2$, and de Sitter for $a<2b^2$.
\end{exl}

\begin{figure}[htb]
\begin{center}
\includegraphics[height=.12\textheight]{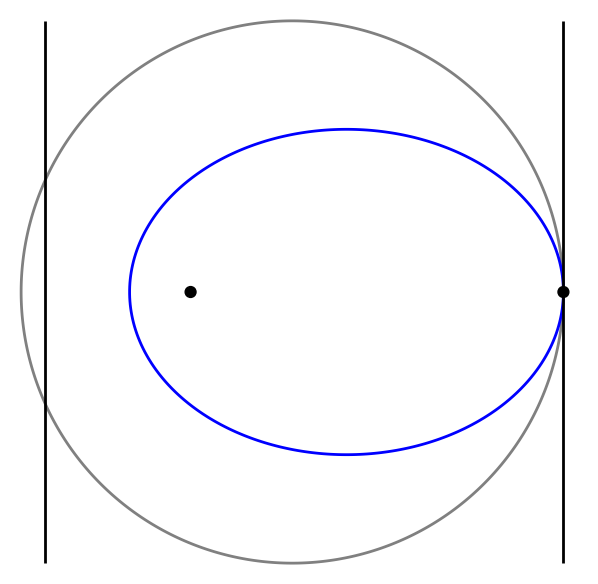} \hspace{.5cm}
\includegraphics[height=.12\textheight]{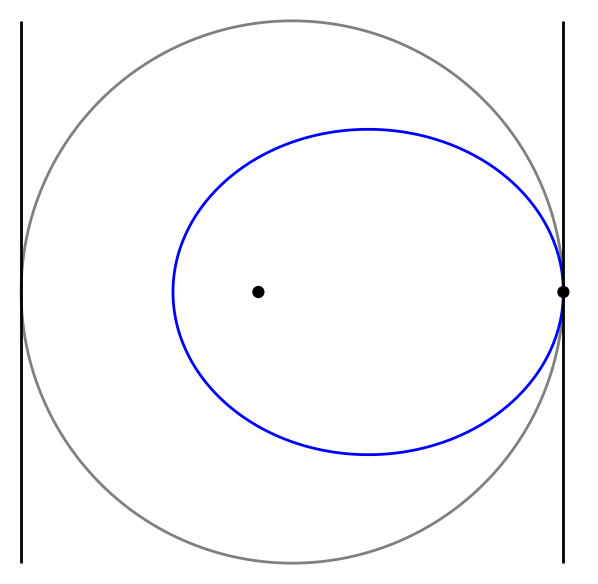} \hspace{.5cm}
\includegraphics[height=.12\textheight]{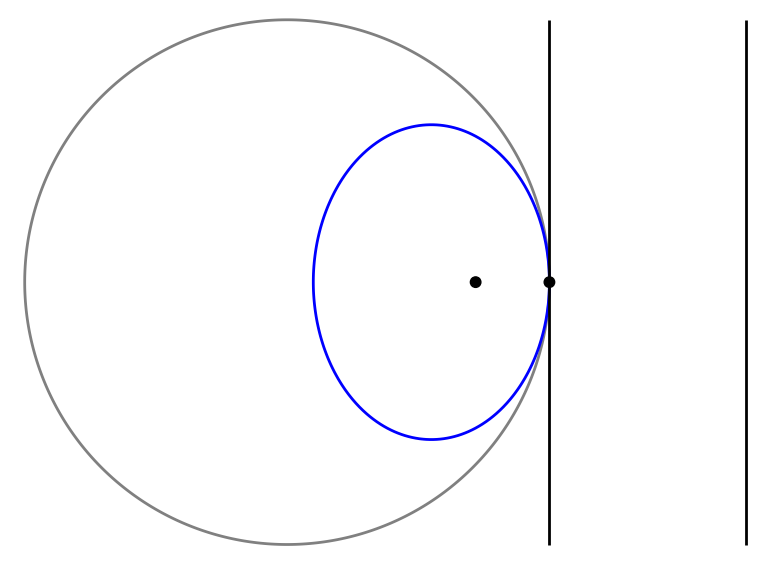}
\end{center}
\caption{Foci and directrices of elliptic parabolas.}
\label{fig:EllParDir}
\end{figure}

\subsection{Families of hyperbolic conics}
\label{sec:FamHyp}
Since the focal lines of a hyperbolic conic are determined by its intersection 
points with the absolute, the conics that share the focal lines form a pencil 
of conics containing the absolute conic. For ellipses and hyperbolas this 
pencil is determined by four distinct points (some of which can form complex 
conjugate pairs). For example, if all four points are real, then the pencil is 
formed by convex and concave hyperbolas and by three pairs of focal lines. 

By polarity between foci and focal lines, the confocal conics (those sharing a 
pair of foci) form a dual pencil. There are several types of confocal nets of 
hyperbolic conics. The simplest ones are formed by concentric cycles and 
pencils of lines through their centers.Confocal nets formed by ellipses, 
hyperbolas, or parabolas are shown on Figures \ref{fig:ConfHypEllHyp} and 
\ref{fig:ConfHypPar}.

\begin{figure}[htb]
\begin{center}
\includegraphics[height=.16\textheight]{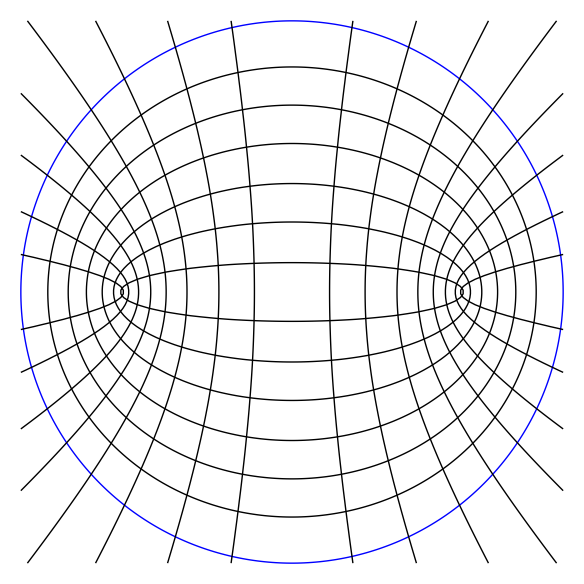} \hspace{.5cm}
\includegraphics[height=.16\textheight]{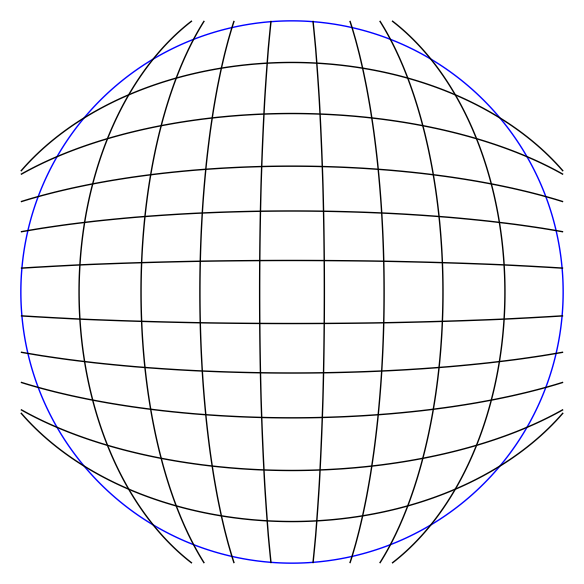} \hspace{.5cm}
\includegraphics[height=.16\textheight]{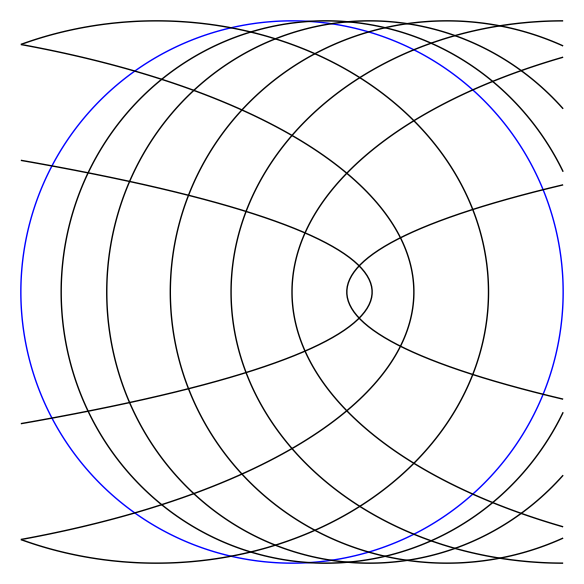}
\end{center}
\caption{Confocal hyperbolic ellipses and hyperbolas.}
\label{fig:ConfHypEllHyp}
\end{figure}

\begin{figure}[htb]
\begin{center}
\includegraphics[height=.16\textheight]{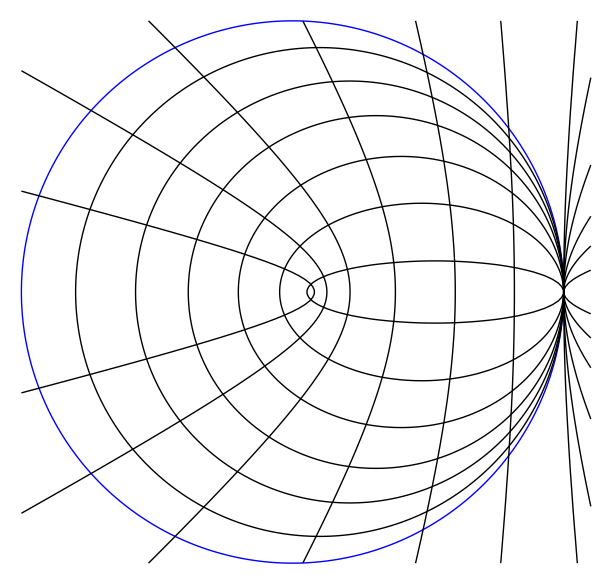} \hspace{.5cm}
\includegraphics[height=.16\textheight]{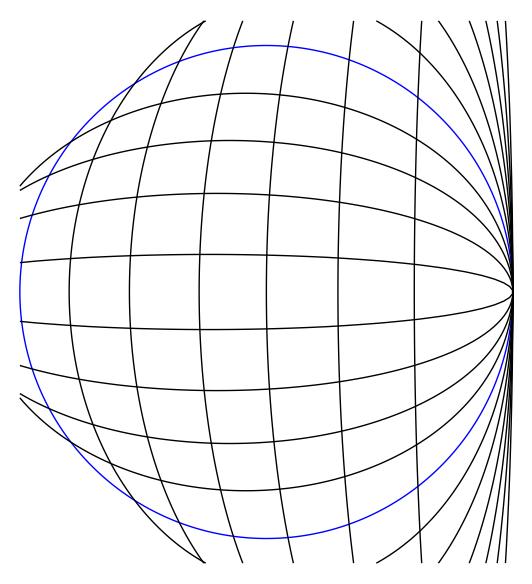} \hspace{.5cm}
\includegraphics[height=.16\textheight]{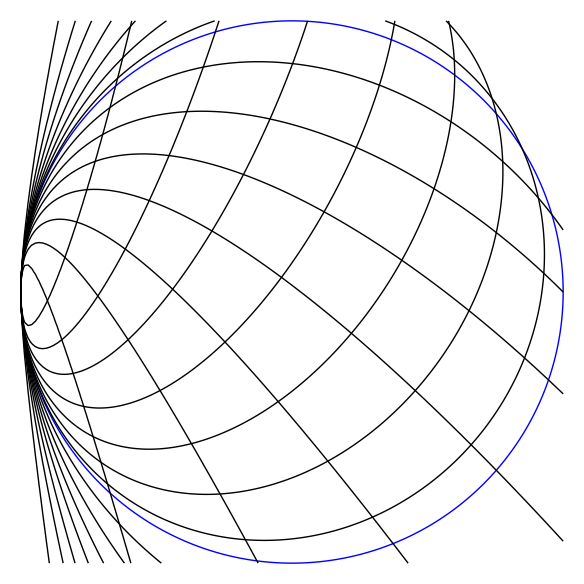}
\end{center}
\caption{Confocal hyperbolic parabolas.}
\label{fig:ConfHypPar}
\end{figure}

Conics that share a focus and a corresponding directrix form a double contact 
pencil, compare Corollary \ref{cor:FocDirPencil} in the spherical case.

\section{Theorems about hyperbolic conics}

\subsection{Ivory's lemma}

Similarly to the Euclidean and to the spherical case (see Theorem 
\ref{thm:IvorySph}), we have the following.

\begin{thm}[Ivory's lemma]
The diagonals in a quadrilateral formed by four confocal hyperbolic conics have 
equal lengths.
\end{thm}
This theorem is proved in \cite{SW04} by a method similar to that used in the 
spherical case, see Section \ref{sec:IvorySph}: one finds a homomorphism 
$F_\lambda \colon \R^3 \to \R^3$ that is self-adjoint with respect to $\Omega$ 
and maps one side of the quadrilateral onto the opposite side.

\subsection{Geometric interpretations of the scalar product}
This section deals with metric properties of hyperbolic conics. In the proofs 
it will be more convenient to use the hyperboloid model of the hyperbolic 
plane. We have
\[
\{x \in \R^3 \mid \Omega(x,x) = -1\} = \H^2 \cup (-\H^2),
\]
where $\H^2$ and $-\H^2$ are two components of the hyperboloid. The distance 
between two points $x,y \in \H^2$ can be computed by the formula
\[
\Omega(x,y) = - \cosh\dist(x,y).
\]

The hyperboloid of one sheet represents a double cover of the de Sitter plane:
\[
\{x \in \R^3 \mid \Omega(x,x) = 1\} = \widetilde{\dS^2}.
\]
As before, a de Sitter point $x \in \widetilde{\dS^2}$ is the pole of a 
hyperbolic line
\[
x^\circ = \{y \in \H^2 \mid \Omega(x,y) = 0\}.
\]
An advantage of $\widetilde{\dS^2}$ over $\dS^2$ is that it allows to give a 
simple description of hyperbolic half-planes:
\[
x^\circ_+ = \{y \in \H^2 \mid \Omega(x,y) \ge 0\}, \quad x^\circ_- = \{y \in 
\H^2 \mid \Omega(x,y) \le 0\}.
\]
The scalar product of two de Sitter points computes the angle or the distance 
between the corresponding hyperbolic lines, and the scalar product of a 
hyperbolic and a de Sitter point computes the distance between a point and a 
line. For details, see Lemma \ref{lem:ScalProd} below.

Finally, there is the isotropic cone, which we divide in two parts
\[
\{x \in \R^3 \mid \Omega(x,x) = 0\} = L \cup (-L),
\]
where $L$ and $-L$ are one-sided cones asymptotic to $\H^2$ and $-\H^2$, 
respectively. For $x \in L$, its polar $x^\circ$ is the plane through $x$ 
tangent to the isotropic cone, which corresponds to a line tangent to the 
absolute. But there is a more interesting object that one can associate with 
a point of $L$. Define
\[
H_x = \{y \in \H^2 \mid \Omega(x,y) = -1\}.
\]
Then $H_x$ is a horocycle centered at $x$ (it is not hard to show that the 
normals to $H_x$ pass through $x$). As a curve in $\R^3$, $H_x$ is a 
parabola; the central projection makes it an ellipse in the 
Beltrami--Cayley--Klein model. 
This ellipse osculates the absolute at the image of the point $x$.
The scalar product $\Omega(x,y)$ for $x \in \H^2$ and $y \in L$ measures the 
distance between $x$ and $H_y$.

To summarize, every vector in $\H^2 \cup L \cup \widetilde{\dS^2} \subset 
\R^3$ corresponds to a geometric object in the hyperbolic plane: a point, a 
(co-oriented) line, or a horosphere; the scalar product of two vectors measures 
the distance between the corresponding objects. An exact formulation is given 
in the lemma below.

\begin{lem}
\label{lem:ScalProd}
\begin{enumerate}
\item
If $x, y \in \H^2$, then
\[
\Omega(x,y) = -\cosh\dist(x,y).
\]
\item
If $x \in \H^2$ and $y \in \widetilde{\dS^2}$, then
\[
\Omega(x,y) = \sinh\dist(x, y^\circ),
\]
where the point-to-line distance $\dist(x, y^\circ)$ is taken positive if the 
vectors $x$ and $y$ lie on the same side from the plane $y^\circ \subset \R^3$, 
and negative 
otherwise.
\item
If $x \in \H^2$ and $y \in L$, then
\[
\Omega(x,y) = - e^{\dist(x, H_y)},
\]
where the distance to a horosphere is taken to be negative for the points 
inside the horosphere and positive for the points outside.
\item
If $x, y \in \widetilde{\dS^2}$, then
\[
\Omega(x,y) =
\begin{cases}
\cos \angle(x^\circ, y^\circ), &\text{ if } x^\circ \cap y^\circ \text{ is 
hyperbolic}\\
\pm 1, &\text{ if } x^\circ \cap y^\circ \text{ is ideal}\\
\pm\cosh \dist(x^\circ, y^\circ), &\text{ if } x^\circ \cap y^\circ \text{ is 
de Sitter}
\end{cases}
\]
Here $\angle(x^\circ, y^\circ)$ is an angle between the lines $x^\circ$ and 
$y^\circ$ occupied by the points $z$ where $\Omega(x,z)$ and $\Omega(y,z)$ have 
different signs. 
\item
If $x \in \widetilde{\dS^2}$ and $y \in L$, then
\[
\Omega(x,y) = e^{\dist(x^\circ, H_y)},
\]
where the distance between a line and a horosphere is the length of the common 
perpendicular taken with the minus sign if the line intersects the horosphere.
\item
If $x, y \in L$, then
\[
\Omega(x,y) = -2 e^{\dist(H_x, H_y)},
\]
where the distance between two horospheres is the length of the common 
perpendicular taken with the minus sign if the horospheres intersect.
\end{enumerate}
\end{lem}

This is proved via an appropriate parametrization of the line spanned by the 
points $x$ and $y$, see e.~g. \cite{Rat, Thu}.

As a simple application of the previous lemma, we can describe the quadratic 
cones that correspond to hyperbolic cycles.

\begin{lem}
\label{lem:Cycles}
For every $c \in \R$ and $p \in \R^3$, the quadratic cone
\[
\Omega(x,x) + c \cdot \Omega(x,p)^2 = 0
\]
corresponds in the hyperbolic-de Sitter plane with the absolute $\Omega$ to a 
cycle centered at the point $p$. Namely, it describes
\[
\begin{cases}
\text{a circle with the center } p, & \text{ if } \Omega(p,p) < 0,\\
\text{a horocycle with the ideal point } p, & \text{ if } \Omega(p,p) = 0,\\
\text{a hypercycle around the line } p^\circ, & \text{ if } \Omega(p,p) > 0.
\end{cases}
\]
\end{lem}
\begin{proof}
The intersection of this quadratic cone with $\H^2 \subset \R^3$ is formed by 
the points $x$ with $\Omega(x,p) = \const$. Lemma \ref{lem:ScalProd} implies 
that these are the points at a constant distance from $p$, if $p \in \H^3$, or 
the points at a constant distance from a horocycle centered at $p$ (hence also 
a horocycle), if $p \in L$, and the points at a constant distance from the line 
$p^\circ$, if $p \in \widetilde{\dS^2}$.
\end{proof}

Results of the following sections are essentially due to Story \cite{Story83}. 
However, he does not care about the position of foci, focal lines etc., stating 
the results in terms of $\Omega(x, \cdot)$ only. Therefore we needed to 
elaborate on the geometric meaning.

\subsection{The focus-directrix property}
\label{sec:FocDirHyp}
Recall that in the Euclidean case the equation
\[
\dist(x,F) = \epsilon \cdot \dist(x,d),
\]
where $F$ is a point and $d$ is a line not passing through $F$, determines 
for $0<\epsilon<1$ an ellipse, for $\epsilon = 1$ a parabola, and for $\epsilon 
> 1$ a hyperbola with a focus $F$ and the corresponding directrix 
$d$. The following theorem gives a similar description of hyperbolic conics.

\begin{thm}
\label{thm:FocDirHyp}
Let $S$ be a non-degenerate hyperbolic conic other than a circle, horocycle, 
and hypercycle.
Let $F$ be a non-ideal focus of $S$, and let $d$ be the 
corresponding directrix. Then the conic consists of all points $x$ that satisfy 
the equation
\begin{equation}
\label{eqn:FocDirHyp}
\delta(F,x) = \epsilon \cdot \delta(d,x)
\end{equation}
for some positive constant $\epsilon$, where
\[
\delta(F,x) =
\begin{cases}
\cosh\dist(x,F^\circ), &\text{ if } F \text{ is a de Sitter point,}\\
\sinh\dist(x,F), &\text{ if } F \text{ is a hyperbolic point,}
\end{cases}
\]
\[
\delta(d,x) =
\begin{cases}
\cosh\dist(x,d^\circ), &\text{ if } d \text{ is a de Sitter line,}\\
e^{\dist(x,H_d)}, &\text{ if } d \text{ is tangent to the absolute,}\\
|\sinh\dist(x,d)|, &\text{ if } d \text{ is a hyperbolic line.}
\end{cases}
\]
\end{thm}

We need a preparatory lemma.
\begin{lem}
\label{lem:OmegaTang}
Let $p$ be a point not on a conic $Q$. Then we have
\[
Q(p,p)Q(x,x) - Q(p,x)^2 \sim t_1(x) \cdot t_2(x),
\]
where $t_1$ and $t_2$ are the (possibly imaginary) tangents to $Q$ through $p$.
\end{lem}
\begin{proof}
The three conics $Q(x,x)$, $Q(p,x)^2$, and $t_1(x) \cdot t_2(x)$ belong to the 
same double contact pencil. Thus we have
\[
\lambda Q(x,x) + \mu Q(p, x)^2 + \nu t_1(x) t_2(x) = 0.
\]
For $x = p$ the third summand vanishes, which implies $\lambda = -\mu Q(p,p)$, 
and the lemma is proved.
\end{proof}

\begin{proof}[Proof of Theorem \ref{thm:FocDirHyp}]
Since the tangents through $F$ to $\Omega$ coincide with the tangents from $F$ 
to $S$, we have by Lemma \ref{lem:OmegaTang}
\[
S(F,F)S(x,x) - S(F,x)^2 \sim \Omega(F,F)\Omega(x,x) - \Omega(F,x)^2,
\]
and hence
\[
S(x,x) \sim \Omega(F,F)\Omega(x,x) - \Omega(F,x)^2 - \lambda S(F,x)^2
\]
for some $\lambda \in \R$. On the other hand
\[
S(F,x) \sim d(x) \sim \Omega(d^\circ, x),
\]
where $d$ is the directrix corresponding to $F$. Hence equation $Q(x,x) = 0$ is 
equivalent to
\[
\Omega(F,F)\Omega(x,x) - \Omega(F,x)^2 = \mu\Omega(d^\circ, x)^2,
\]
for some $\mu \in \R$, which is in turn equivalent to
\[
1 - \frac{\Omega(F,x)^2}{\Omega(F,F)\Omega(x,x)} = \nu 
\frac{\Omega(d^\circ,x)^2}{\Omega(x,x)}.
\]
Lemma \ref{lem:ScalProd} provides metric interpretations of the expressions on 
the left and the right hand side.
\end{proof}

It can be shown that the constant in equation \ref{eqn:FocDirHyp} is equal to
\[
\epsilon = \sqrt{\frac{\Omega(E,E) S(F,d^\circ)}{S(E,E) \Omega(F,d^\circ)}},
\]
where $E$ is the center of the conic corresponding to the focus $F$, that is 
the intersection point of the directrix $d$ and the polar $F^\circ$ of $F$.

A classification of hyperbolic conics according to the nature of the 
(focus, directrix) pair and the value of $\epsilon$ seems to be missing in the 
literature. Below we describe some special cases without going into details.

Let the focus be a hyperbolic point, and the directrix be both hyperbolic.
Then the equation
\[
\sinh\dist(x,F) = \epsilon \cdot |\sinh\dist(x,d)|
\]
describes
\[
\begin{cases}
\text{for } 0 < \epsilon < e^{-\dist(d,F)} & \text{an ellipse,}\\
\text{for } \epsilon = e^{-\dist(d,F)} & \text{an elliptic parabola,}\\
\text{for } e^{-\dist(d,F)} < \epsilon < e^{\dist(d,F)} & \text{a 
semi-hyperbola,}\\
\text{for } \epsilon = e^{\dist(d,F)} & \text{a convex hyperbolic parabola,}\\
\text{for } \epsilon > e^{\dist(d,F)} & \text{a convex hyperbola.}\\
\end{cases}
\]
This can be shown by restricting the equation to the line through $F$ 
perpendicular to $d$.

\begin{exl}
\label{exl:SpecSH}
For a point $F$ and a line $d$ not through $F$ the locus of points satisfying
\[
\dist(x,F) = \dist(x,d).
\]
is a semi-hyperbola whose other focus is the pole of $d$, and the 
corresponding directrix is the polar of $F$, see Example \ref{exl:SH}. In the 
Beltrami--Cayley--Klein model in the unit disk, for $F=(c,0)$ and $d=\{x=-c\}$ 
this 
semi-hyperbola is described by the equation $x = \frac{1-c^2}{4c} y^2$.
\end{exl}

If the focus $F$ is hyperbolic, and the directrix $d$ is tangent to the 
absolute or de Sitter, then for small values of $\epsilon$ we get ellipses. As 
$\epsilon$ increases, the ellipses transition through elliptic parabolas to 
semihyperbolas, see Examples \ref{exl:SH} and \ref{exl:EllPar}.

If for a semihyperbola we take its de Sitter focus, then changing the value of 
$\epsilon$ will transform the semihyperbola into a concave hyperbola, passing 
through a (long or wide) concave hyperbolic parabola or, in an exceptional 
case, through a horocycle.

\subsection{Bifocal properties of hyperbolic conics}
\label{sec:BifocHyp}
The following theorem is an analog of Theorem \ref{thm:SinSin} about spherical 
conics.
\begin{thm}
\label{thm:SinhProd}
\begin{enumerate}
\item
Let $\ell_1$, $\ell_2$ be a pair of lines in the hyperbolic-de Sitter plane.
For a point $x \in \H^2$ let
\[
d_i(x) =
\begin{cases}
\sinh \dist(x, \ell_i), &\text{ if } \ell_i \text{ is hyperbolic}\\
\cosh \dist(x, \ell_i^\circ), &\text{ if } \ell_i \text{ is de Sitter}\\
e^{\dist(x, H_i)}, &\text{ if } \ell_i \text{ is tangent to the absolute.}
\end{cases}
\]
Here $H_i$ is a horosphere centered at the ideal point $\ell_i^\circ$, and 
the distances from a point to a line or to a horosphere are equipped with a 
sign. Then the locus of points that satisfy an equation of the form
\begin{equation}
\label{eqn:DistFocLines}
d_1(x) \cdot d_2(x) = c
\end{equation}
is a hyperbolic conic, and $\ell_1, \ell_2$ is a pair of its focal lines. 
Conversely, for every pair of focal lines of a hyperbolic conic, the points on 
the conic satisfy equation \eqref{eqn:DistFocLines}.

\item
Let $p_1, p_2$ be a pair of points in the hyperbolic-de Sitter plane. For a 
line $\xi$ in the hyperbolic plane let
\[
d_i(\xi) =
\begin{cases}
\sinh \dist(p_i, \xi), &\text{ if } p_i \text{ is hyperbolic}\\
e^{\dist(\xi, H_i)}, &\text{ if } p_i \text{ is ideal,}
\end{cases}
\]
where $H_i$ is a horosphere centered at $p_i$. The distances are equipped 
with a sign. If $p_i$ is a de Sitter point, then for all oriented lines $\xi 
\ne p_i^\circ$ put
\[
d_i(\xi) =
\begin{cases}
\cos\angle(\xi, p_i^\circ), &\text{ if } \xi \cap p_i^\circ \text{ is 
hyperbolic}\\
\pm 1, &\text{ if } \xi \cap p_i^\circ \text{ is ideal}\\
\pm \cosh\dist(\xi, p_i^\circ), &\text{ if } \xi \cap p_i^\circ \text{ is de 
Sitter.}
\end{cases}
\]
The sign convention must be chosen in such a way that $d_i(\xi)$ depends 
continuously on $\xi$. Then the envelope of the lines that satisfy an equation 
of the form
\begin{equation}
\label{eqn:DistFociProd}
d_1(\xi) \cdot d_2(\xi) = c
\end{equation}
is a hyperbolic conic, and $p_1, p_2$ is a pair of its foci. Conversely, for 
every pair of foci of a hyperbolic conic, the tangents to the conic satisfy 
equation \eqref{eqn:DistFociProd}.
\end{enumerate}
\end{thm}
\begin{proof}
Let us prove the first part of the theorem.

Lemma \ref{lem:ScalProd} implies that
\[
d_1(x) \cdot d_2(x) = \lambda \frac{\Omega(x,\ell_1^\circ) 
\Omega(x,\ell_2^\circ)}{\Omega(x,x)}
\]
for some constant $\lambda$. Therefore equation \eqref{eqn:DistFocLines} is 
equivalent 
to
\[
\Omega(x,x) - \ell_1(x) \cdot \ell_2(x) = 0,
\]
where $\ell_i$ are certain linear functionals on $\R^3$ whose kernels are the 
lines $\ell_i$. Thus the solution set of \eqref{eqn:DistFocLines} is the zero 
set of a quadratic form $Q = \Omega - \ell_1 \ell_2$, that is a hyperbolic 
conic. If $x$ is a (real or imaginary) intersection point of $Q$ and $\Omega$, 
then we have $\ell_1(x) \cdot \ell_2(x) = 0$, therefore the lines $\ell_1$ and 
$\ell_2$ pass through all four intersection points of $Q$ and $\Omega$. If some 
of the intersection points coincide (the conic is a parabola or a circle), then 
either one of the lines is an absolute tangent or both lines pass through the 
tangency point (this can be shown by passing to the limit). Hence 
$\ell_1, \ell_2$ is a pair of focal lines of the conic $Q$.

In the opposite direction, let $Q$ be a conic, and let $\ell_1, \ell_2$ be a 
pair of its focal lines. Then the conics $Q$, $\Omega$, $\ell_1 \cdot \ell_2$ 
belong to a pencil. Therefore, up to scalar factors we have
\[
Q = \Omega - \ell_1 \cdot \ell_2.
\]
Thus for all points $x$ on the conic we have $\ell_1(x) \cdot \ell_2(x) = 
\Omega(x,x)$, which implies $d_1(x) \cdot d_2(x) = c$ for some constant $c$.

The proof of the second part is similar. We have
\[
d_1(\xi) \cdot d_2(\xi) = \lambda \frac{\Omega(\xi^\circ, p_1) 
\Omega(\xi^\circ, p_2)}{\Omega(\xi^\circ, \xi^\circ)}.
\]
Therefore equation \eqref{eqn:DistFociProd} is equivalent to
\[
\Omega(\xi^\circ, \xi^\circ) - p_1(\xi^\circ) \cdot p_2(\xi^\circ) = 0,
\]
where $p_i(x)$ is a linear functional proportional to $\Omega(p_i, x)$. This 
equation describes a hyperbolic-de Sitter conic; its polar is the envelope of 
the lines that satisfy equation \eqref{eqn:DistFociProd}.
\end{proof}

\begin{exl}
The locus of points that satisfy
\[
\sinh \dist(x, \ell_1) \cdot \sinh \dist(x, 
\ell_2) = c,
\]
where $\ell_1, \ell_2$ are two hyperbolic lines, is
\begin{itemize}
\item
a concave or convex hyperbola or a pair of lines, if the intersection point of 
$f_1$ and $f_2$ is non-ideal;
%mention the de Sitter hyperbola?
\item
a concave or convex hyperbolic parabola, if the intersection point of $f_1$ 
and $f_2$ is ideal.
\end{itemize}
\end{exl}

\begin{exl}
The locus of points that satisfy $\dist(x, H_1) + \dist(x, H_2) = c$, where 
$H_1, H_2$ are two horocycles, is a hypercycle whose ideal points are the 
centers of $H_1$ and $H_2$. The envelope of the lines that satisfy $\dist(\xi, 
H_1) + \dist(\xi, H_2) = c$ is also a hypercycle.
\end{exl}

\begin{thm}
\label{thm:SumConstHyp}
Let $F_1, F_2$ be a pair of foci of a hyperbolic conic. Then the conic consists 
of all points $x$ that satisfy an equation of the form
\[
|\delta_1(x) + \delta_2(x)| = \const \quad \text{or} \quad |\delta_1(x) - 
\delta_2(x)| = \const.
\]
Here the functions $\delta_i(x)$ are the distances to the foci or to their 
polars or to the corresponding horocycles:
\[
\delta_i(x) =
\begin{cases}
\dist(x,F_i), &\text{ if } F_i \text{ is hyperbolic,}\\
\dist(x,F_i^\circ), &\text{ if } F_i \text{ is de Sitter,}\\
\dist(x,H_i), &\text{ if } F_i \text{ is ideal,}
\end{cases}
\]
where $H_i$ is an arbitrary horocycle centered at $F_i$.
\end{thm}

\begin{proof}
Let $F$ be a focus of the conic. First assume that $F$ does not belong to the 
conic (and hence does not belong to the absolute).
Denote by $t_1, t_2$ the two absolute tangents through $F$ (which are both real 
or complex conjugate to each other). By Lemma \ref{lem:OmegaTang} we have
\[
Q(F,F) Q(x,x) - Q(F,x)^2 \sim t_1(x) \cdot t_2(x) \sim
\Omega(F,F) \Omega(x,x) - \Omega(F,x)^2,
\]
which implies
\[
Q(x,x) - \frac{Q(F,x)^2}{Q(F,F)} = \lambda \left( \Omega(x,x) 
- \frac{\Omega(F,x)^2}{\Omega(F,F)} \right)
\]
for some $\lambda \in \R$. To compute $\lambda$, substitute $x = E$, the center 
conjugate to the axis through $F$. Since the directrix corresponding to $F$ 
also goes through $E$ (see Lemma \ref{lem:FocDir}), we have $Q(F,E)=0$, so that 
$\lambda = \frac{Q(E,E)}{\Omega(E,E)}$.

With the above argument applied to a pair of foci $F_1$, $F_2$ (of which we 
assume that both don't lie on the conic), we obtain
\[
Q(x,x) - \lambda\Omega(x,x) = 
\frac{Q(F_i,x)^2}{Q(F_i,F_i)} - 
\lambda\frac{\Omega(F_i,x)^2}{\Omega(F_i,F_i)} \quad \text{ for }i=1,2.
\]
Introduce the linear functions $\ell_1(x), \ell_2(x)$:
\[
\ell_i(x) =
\begin{cases}
\frac{\Omega(F_i,x)}{\sqrt{\Omega(F_i,F_i)}}, &\text{ if } F_i \text{ is de 
Sitter,}\\
\frac{\Omega(F_i,x)}{\sqrt{-\Omega(F_i,F_i)}}, &\text{ if } F_i \text{ is 
hyperbolic.}
\end{cases}
\]
Since the directrices $Q(F_i,x) = 0$ and the polars $\Omega(F_i,x) = 0$ of the
foci meet at the center $E$, each of the linear functions $Q(F_i,x)$ is a 
linear combination of $\ell_1(x)$ and $\ell_2(x)$. Taking into account that 
$Q(F_i,F_i)$ has the same sign as $\Omega(F_i,F_i)$, we obtain
\[
Q - \lambda\Omega =
\begin{cases}
(a_{i1} \ell_1 + a_{i2} \ell_2)^2 - \lambda \ell_1^2, &\text{ if } F_i \text{ 
is de Sitter,}\\
-(a_{i1} \ell_1 + a_{i2} \ell_2)^2 + \lambda \ell_1^2, &\text{ if } F_i \text{ 
is hyperbolic.}
\end{cases}
\]

We now make a case distinction.

1) Both foci are de Sitter. We have
\[
(a\ell_1 + b\ell_2)^2 - \lambda \ell_1^2 = Q - \lambda\Omega = (c\ell_1 + 
d\ell_2)^2 - \lambda \ell_2^2.
\]
Solving this we obtain $\lambda = a^2 - b^2$. Denoting $\mu = \frac{a}{b}$ we 
see that
\[
Q \sim (\mu^2 - 1) \Omega + \ell_1^2 + \ell_2^2 + 2\mu\ell_1\ell_2.
\]
Since by Lemma \ref{lem:ScalProd} $\frac{\ell_i(x)}{\sqrt{-\Omega(x,x)}} = 
\sinh\dist(x,F_i^\circ)=:\sinh\delta_i$, equation $Q(x,x) = 0$ is equivalent to
\[
(1 - \mu^2) + \sinh^2\delta_1(x) + \sinh^2\delta_2(x) + 2\mu 
\sinh\delta_1(x)\sinh\delta_2(x) = 0,
\]
which factors as
\[
(\cosh(\delta_1(x)+\delta_2(x)) - \mu)(\cosh(\delta_1(x)-\delta_2(x)) + \mu) = 
0.
\]
If $\mu > 0$, then the conic is described by $|\delta_1(x) + \delta_2(x)| = c$; 
if $\mu < 0$, then it is described by $|\delta_1(x) - \delta_2(x)| = c$.

2) Both foci are hyperbolic. In this case we have
\[
-(a\ell_1 + b\ell_2)^2 + \lambda \ell_1^2 = Q - \lambda\Omega = -(c\ell_1 + 
d\ell_2)^2 + \lambda \ell_2^2,
\]
which again results in $\lambda = a^2 - b^2$, so that
\[
Q \sim (\mu^2-1)\Omega - \ell_1^2 - \ell_2^2 - 2\mu\ell_1\ell_2.
\]
Now by Lemma \ref{lem:ScalProd} we have $\frac{\ell_i(x)}{\sqrt{-\Omega(x,x)}} 
= -\cosh\dist(x,F_i) =: -\cosh\delta_i$. Equation $Q(x,x)=0$ is equivalent 
to
\[
\mu^2 - 1 + \cosh^2\delta_1 + \cosh^2\delta_2 + 2\mu\cosh\delta_1\cosh\delta_2 
= 0,
\]
which factors as
\[
(\cosh(\delta_1+\delta_2) + \mu)(\cosh(\delta_1-\delta_2) + \mu) = 0.
\]
If $\mu > 0$, then the conic is empty (in fact, it is the de Sitter ellipse). 
If $\mu < 0$ and $\dist(F_1,F_2) < \arcosh(-\mu)$, then the second factor 
never vanishes, and the conic is described by $\delta_1 + \delta_2 = 
\arcosh(-\mu)$. If $\mu < 0$ and $\dist(F_1,F_2) > \arcosh(-\mu)$, then the 
first factor does not vanish, and the conic is described by $|\delta_1 - 
\delta_2| = \arcosh(-\mu)$.

3) Focus $F_1$ is de Sitter, focus $F_2$ is hyperbolic. We have
\[
(a\ell_1 + b\ell_2)^2 - \lambda \ell_1^2 = Q - \lambda\Omega = -(c\ell_1 + 
d\ell_2)^2 + \lambda \ell_2^2.
\]
This implies $\lambda = a^2 + b^2$, and we obtain
\[
Q \sim (1+\mu^2)\Omega - \ell_1^2 + \ell_2^2 + 2\mu\ell_1\ell_2.
\]
This time we have
\[
\frac{\ell_1(x)}{\sqrt{-\Omega(x,x)}} = \sinh\dist(x,F_1^\circ), \quad 
\frac{\ell_2(x)}{\sqrt{-\Omega(x,x)}} = -\cosh\dist(x, F_2).
\]
Equation $Q(x,x)=0$ is equivalent to
\[
1 + \mu^2 + \sinh^2\delta_1 - \cosh^2\delta_2 + 2\mu\sinh\delta_1\cosh\delta_2 
= 0,
\]
which factors as
\[
(\sinh(\delta_1+\delta_2) + \mu)(\sinh(\delta_1-\delta_2) + \mu) = 0.
\]
As in the previous case, only one of the factors can vanish (depending on the 
relation between $\mu$ and $\dist(F_2, F_1^\circ)$), and the conic is described 
by one of the equations $\delta_1 + \delta_2 = \const$ or $\delta_1 - \delta_2 
= \const$.

It remains to deal with the case when one or both foci are ideal. If a focus 
$F$ is ideal, then we claim that
\[
Q(x,x) - \lambda\Omega(x,x) = \Omega(F,x) \cdot m(x),
\]
where $\lambda = \frac{Q(E,E)}{\Omega(E,E)}$ as before, and $m(x)$ is some 
linear function. Indeed, the choice of $\lambda$ ensures that the conic $Q - 
\lambda\Omega$ goes through the point $E$. This point belongs to the line 
$\Omega(F,x)=0$, which is a common tangent of $Q$ and $\Omega$ at the point 
$F$, and hence is also a tangent of $Q-\lambda\Omega$ at $F$. It follows 
that this line is contained in the conic $Q - \lambda\Omega$, and the statement 
is proved.

Now, assuming that in a pair of foci $(F_1, F_2)$ the first one is de Sitter, 
and the second one is ideal, denote $\ell_2(x) = \Omega(F_2,x)$. We have
\[
(a\ell_1 + b\ell_2)^2 - \lambda\ell_1^2 = Q - \lambda\Omega = \ell_2 \cdot m.
\]
It follows that $\lambda = a^2$, so that
\[
Q \sim \mu^2 \Omega + 2\mu \ell_1 \ell_2 + \ell_2^2,
\]
where $\mu = \frac{a}{b}$. By Lemma \ref{lem:ScalProd}, we have
\[
\frac{\ell_1(x)}{\sqrt{-\Omega(x,x)}} = \sinh\delta_1, \quad 
\frac{\ell_2(x)}{\sqrt{-\Omega(x,x)}} = -e^{\delta_2},
\]
hence the equation $Q(x,x)=0$ is equivalent to
\[
e^{2\delta_2} - \mu e^{\delta_2 + \delta_1} + \mu e^{\delta_2 - \delta_1} - 
\mu^2 = 0,
\]
which factors as
\[
(e^{\delta_2+\delta_1} + \mu)(e^{\delta_2-\delta_1} - \mu) = 0.
\]
Depending on whether $\mu$ is positive or negative, the conic is described by 
one of the equations $\delta_1 + \delta_2 = \const$ or $\delta_1 - \delta_2 = 
\const$.

The case of an ideal and a hyperbolic focus is similar. If both foci are ideal, 
then we have
\[
Q - \lambda\Omega = a\ell_1 \cdot \ell_2,
\]
which implies that $Q(x,x) = 0$ is equivalent to $\delta_1 + \delta_2 = 
\const$.
\end{proof}

Similarly to the spherical case (Theorem \ref{thm:ConstSum}), Theorem 
\ref{thm:SumConstHyp} has a dual that deals with the angles that a tangent 
to the conic makes with a pair of focal lines (for ultraparallel lines, the 
angle becomes the length of a common perpendicular). We 
don't list all the possibilities here, but here is one interesting particular 
case. Take two rays in the hyperbolic plane starting at the same point; then 
the segments with the endpoints on these rays that bound a triangle of constant 
area envelope a branch of a convex hyperbola, see Figure \ref{fig:HypAreaHyp}.

\begin{figure}[htb]
\begin{center}
\includegraphics[width=.35\textwidth]{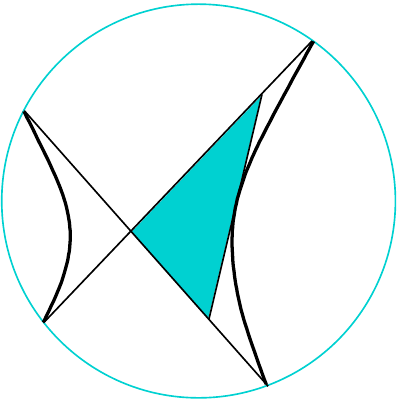}
\end{center}
\caption{A pair of focal lines and a tangent to a convex hyperbola bound a 
triangle of constant area.}
\label{fig:HypAreaHyp}
\end{figure}

A special case of Theorem \ref{thm:SumConstHyp} with $F_1$ hyperbolic, 
$F_2$ de Sitter, and zero difference of distances:
\[
\dist(x,F_1) - \dist(x,F_2^\circ) = 0
\]
is described in Example \ref{exl:SpecSH}.

Theorem \ref{thm:SumConstHyp} implies that at every point of the conic the 
tangent to the conic forms equal angles with the gradients of functions 
$\delta_1$ and $\delta_2$.
% These gradients are tangent to the geodesics 
% joining the point with the foci; if a focus is ideal, then the geodesics are 
% orthogonal to the horospheres centered at the focus; if a focus is de Sitter, 
% then the geodesics are perpendicular to the polar of the focus.
From this we obtain an optical property of the foci, similar to the 
Euclidean one.

\begin{thm}
\label{thm:OpticalHyp}
Let $(F_1, F_2)$ be a pair of foci of a hyperbolic conic. Then every light ray 
originating from $F_1$ reflects from the conic in such a way that it either 
passes through $F_2$ or continues a ray originating from $F_2$.
\end{thm}

The rays originating at an ideal or de Sitter point can be defined in two 
ways: either as half-lines in a projective model of the hyperbolic-de Sitter 
plane or as the rays issued by points on a horocycle or on a line in directions 
orthogonal to that horocycle or a line. Figure \ref{fig:HypOptic} shows two 
examples.

\begin{figure}[htb]
\begin{center}
\includegraphics[width=.9\textwidth]{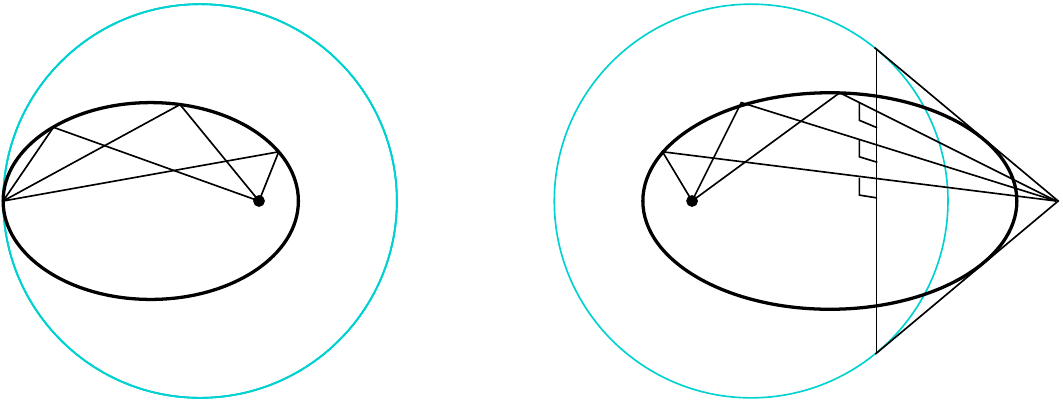}
\end{center}
\caption{Optical properties of elliptic parabolas and semi-hyperbolas.}
\label{fig:HypOptic}
\end{figure}

A statement dual to Theorem \ref{thm:OpticalHyp} says that every segment 
tangent to the conic and having endpoints on a pair of focal lines is bisected 
by the point of tangency. Similar to the Euclidean and the spherical case, 
there is a generalization of Theorem \ref{thm:OpticalHyp} and of its dual, see 
Theorem \ref{thm:Bisect1}. All of these are proved in \cite{Story83} in a 
direct way, but the arguments are quite intricate.

Reflective properties of some hyperbolic conics were described in 
\cite{Per96} within the Poincare model of the hyperbolic plane.

\end{document}